\documentclass[10pt,reqno]{amsart}
\usepackage{appendix}
\usepackage{amsmath,amsfonts,amssymb,amscd,amsthm,amsbsy,bbm, epsf,calc,enumerate,color,graphicx,verbatim,cite,import}
\usepackage{color}
\usepackage{datetime,appendix}
\usepackage{chapterbib,epstopdf}
\usepackage[mathlines]{lineno} 

\usepackage{float}
\usepackage{graphicx}
\usepackage{dsfont}
\usepackage{ifpdf}
\ifpdf
    \usepackage[colorlinks,citecolor=black,linkcolor=black]{hyperref}
\fi

\newcommand{\beq}{\begin{equation}}
\newcommand{\eeq}{\end{equation}}

\newcommand{\bbR}{{\mathbb{R}}}

\newcommand{\bbZ}{{\mathbb{Z}}}

\newcommand{\calO}{{\mathcal {O}}}

\newcommand{\calN}{{\mathcal N}}
\newcommand{\calB}{{\mathcal B}}

\newcommand{\calL}{{\mathcal L}}

\newcommand{\calH}{{\mathcal H}}

\newcommand{\calG}{{\mathcal G}}

\newcommand{\lb}{\label}

\newtheorem{definition}{Definition}[section]
\newtheorem{thm}{Theorem}[section]
\newtheorem{theorem}[thm]{Theorem}

\newtheorem{lem}[thm]{Lemma}
\newtheorem{lemma}[thm]{Lemma}

\newtheorem{corollary}[thm]{Corollary}

\newtheorem{proposition}[thm]{Proposition}

\theoremstyle{definition}

\newtheorem{remark}{Remark}[section]

\definecolor{darkgreen}{rgb}{0,0.4,0}

\newcommand{\R}{\mathbb{R}}

\def\un{\underline}

\def\Xint#1{\mathchoice
   {\XXint\displaystyle\textstyle{#1}}%
   {\XXint\textstyle\scriptstyle{#1}}%
   {\XXint\scriptstyle\scriptscriptstyle{#1}}%
   {\XXint\scriptscriptstyle\scriptscriptstyle{#1}}%
   \!\int}
\def\XXint#1#2#3{{\setbox0=\hbox{$#1{#2#3}{\int}$}
     \vcenter{\hbox{$#2#3$}}\kern-.5\wd0}}

\def\mint{\Xint-}

\newcommand{\eps}{\varepsilon}

\numberwithin{equation}{section}

\def\dist{\text{\rm dist}}

\def\Xint#1{\mathchoice
{\XXint\displaystyle\textstyle{#1}}%
{\XXint\textstyle\scriptstyle{#1}}%
{\XXint\scriptstyle\scriptscriptstyle{#1}}%
{\XXint\scriptscriptstyle\scriptscriptstyle{#1}}%
\!\int}
\def\XXint#1#2#3{{\setbox0=\hbox{$#1{#2#3}{\int}$ }
\vcenter{\hbox{$#2#3$ }}\kern-.6\wd0}}

\newcommand{\pa}{\partial}
\newcommand{\na}{\nabla}
\newcommand{\D}{\Delta}
\DeclareMathOperator*{\sgn}{sgn}

\title[Convergence of Free Boundaries]{Convergence of Free Boundaries in the Incompressible Limit of Tumor Growth Models}

\author[Jiajun Tong and Yuming Paul Zhang]{Jiajun Tong and Yuming Paul Zhang}

\address{(J.~Tong) Beijing International Center for Mathematical Research \\ Peking University\\ Beijing 100871, China}
\email{\tt tongj@bicmr.pku.edu.cn}

\address{(Y.~Zhang) Department of Mathematics \& Statistics \\ Auburn University\\ Auburn, AL 36849, United States}
\email{\tt yzhangpaul@auburn.edu}

\begin{document}

\vspace{18mm} \setcounter{page}{1} \thispagestyle{empty}

\begin{abstract}
We investigate the general Porous Medium Equations with drift and source terms that model tumor growth.
Incompressible limit of such models has been well-studied in the literature, where convergence of the density and pressure variables are established, while it remains unclear whether the free boundaries of the solutions exhibit convergence as well.
In this paper, we provide an affirmative result by showing that the free boundaries converge in the Hausdorff distance in the incompressible limit.
To achieve this, we quantify the relation between the free boundary motion and spatial average of the pressure, and establish a uniform-in-$m$ strict expansion property of the pressure supports.
As a corollary, we derive upper bounds for the Hausdorff dimensions of the free boundaries and show that
the limiting free boundary has finite $(d-1)$-dimensional Hausdorff measure.
\end{abstract}

\maketitle


\vspace{5pt}
\vspace{5pt}


\section{Introduction}
\label{sec: intro}
Consider the Porous Medium Equation (PME) with both drift and source terms:
\begin{equation}\label{main}
\partial_t\varrho =\nabla\cdot (\varrho\nabla p)+\nabla\cdot (\varrho\,  b (x,t))+\varrho f(x,t,p) \quad \hbox{in } Q_T:=\mathbb{R}^d\times (0,T),
\end{equation}
equipped with a bounded, non-negative, and compactly supported initial data $\varrho(\cdot,0)$. 
Here $d\geq 2$, $T>0$, $b:Q_T \to \mathbb{R}^d$ is a given vector field, and $f:Q_T\times[0,\infty)\to \bbR$ describes a pressure-limited growth rate.
The unknowns $\varrho: Q_T\to [0,+\infty)$ and
$p: Q_T\to [0,+\infty)$ represent an evolving density and is corresponding pressure, respectively, and they are related by
\beq\label{main2}
p=P_m(\varrho):=\frac{m}{m-1}\varrho^{m-1}\quad\text{with }m>1.
\eeq
With this, the term $\nabla \cdot (\varrho\nabla p)$ in \eqref{main} can be written as a nonlinear diffusion $\Delta \varrho^m$, which admits degeneracy when the density $\varrho$ is close to $0$. 
Such diffusion can effectively model nonlinear smoothing behavior in various physical settings, such as fluid flows in porous media and population dynamics (see e.g., \cite{BGHP,TBL,hw95,W}). 
The advection term in \eqref{main} models transport of agents by a background flow, while the source term accounts for the pressure-dependent change of $\varrho$. 
In view of this, \eqref{main}--\eqref{main2} are commonly used to model a time-varying distribution of tumor cells under the influence of an external drift as well as their own growth and death \cite{byrne1996modelling,friedman2007mathematical,friedman2004hierarchy}.
It is referred to as a compressible tumor growth model because $\varrho$ and $p$ satisfy the compressible constitutive law \eqref{main2} \cite{david2021free}.
It is worth mentioning that the equation for $p$, which reads
\begin{equation}\label{1.1}
\partial_t p=(m-1) p(\Delta p+\nabla\cdot b +f)+\nabla p\cdot(\nabla p+ b),
\end{equation}
plays an important role in the study of this type of models.

One key feature of the degenerate diffusion is the property of finite-speed propagation. That is, if the initial data is non-negative, bounded, and compactly supported, the positive set of $\varrho$ stays bounded within any finite time.
Hence, whenever $p(\cdot,t)=0$ in some open domain of the space, there appears a {\it free boundary} separating the region where $\varrho$ is positive from that where $\varrho=0$; 
it is defined to be the set $\pa \{\varrho(\cdot,t)>0\}$, or equivalently, $\pa \{p(\cdot,t)>0\}$.
Since \eqref{main} can be rewritten as
$$
\pa_t \varrho - \nabla\cdot \big((\nabla p + {b}(x,t))\varrho\big)=\varrho f(x,t,p),
$$
one can formally deduce the outward normal velocity $V$ of the free boundary whenever it is locally a sufficiently smooth hypersurface
\begin{equation*}
V = - (\nabla p+{b})\cdot \nu = |\nabla p| -{b}\cdot \nu \quad \hbox{on } (x,t) \in \Gamma:=\partial\{p>0\},
\end{equation*}
where $\nu$ denotes the outward unit normal vector in space at a boundary point $(x,t)$. Such motion of the free boundary agrees with the dynamics governed by the Darcy's law.
Since $\varrho$ solves a diffusion equation, one can expect that the free boundary gets regularized by the term $|\nabla p|$ as time goes by.


Given $m>1$, let $(\varrho_m,p_m)$ denote the solution to \eqref{main}--\eqref{main2}. 
Under certain conditions, as $m\to\infty$, $(\varrho_m,p_m)$ will converge in a suitable sense to the unique weak solution $(\varrho_\infty,p_\infty)$ of a Hele-Shaw-type flow
\beq\lb{1.2}
\left\{
\begin{aligned}
&\partial_t\varrho_\infty=\Delta\varrho_\infty+\nabla\cdot(\varrho_\infty { b })+\varrho_\infty f(x,t,p_\infty) &&\quad\text{ in }Q_T,\\
&\Delta p_\infty+\nabla\cdot{ b }+f(x,t,p_\infty)=0&&\quad\text{ on }\{p_\infty>0\},\\
&p_\infty(1-\varrho_\infty)=0,\quad \varrho_\infty\leq 1&&\quad\text{ in }Q_T.
\end{aligned}
\right.
\eeq
Such convergence is usually called \emph{the incompressible limit} of \eqref{main}--\eqref{main2}, and \eqref{1.2} is referred as an incompressible model.
When $b=0$, $f=0$ and suitable boundary conditions are prescribed, \eqref{1.2} corresponds to the classical Hele-Shaw model, which describes the dynamics of a fluid injected into the narrow gap between two horizontally-placed parallel plates \cite{HS1898}.
Many problems in the fluid dynamics and the mathematical biology can be treated as the Hele-Shaw model or its variants; 
readers are referred to related studies on the fluid dynamics \cite{Alazard2019LyapunovFI,Alazard2023TheHS,CJK,CJK2,Crdoba2008InterfaceET,dibenedetto1984ill,dong21,dong2,EJ,entov95,escher1997classical,howison1986cusp,howison1992complex,kim2003uniqueness,cheng2014global}, tumor growth models \cite{Collins2023FreeBR,jacobs2022tumor,kim2021interface,PQV}, and population dynamics \cite{maury2010,CKY}, whereas the list is by no means exhaustive.
The incompressible limit as $m\to \infty$ has been justified in many models that are similar to \eqref{main}--\eqref{main2}. For example, \cite{PQV,kim2018porous,David_S,KM2023} concern the case $b=0$ and $f=f(p)$; \cite{guillen22} considers the case $b=0$ and $f=f(x,t)$; and \cite{AKY,kim2019singular,david2022convergence,chu2022} study the equations with advections. 
\cite{CKY,He2022IncompressibleLO} studied the model with chemotaxis via Newtonian interaction, and very recently, \cite{He2023IncompressibleLO} further addressed the case with both growth and chemotaxis.
Besides, the incompressible limit of tumor growth models with nutrient was analyzed in \cite{PQV,david2021free}.
Let us mention that the incompressible limit is also a classic problem in the Navier-Stokes equation and related fluid models \cite{bresch2014singular,Lions1998IncompressibleLF,lions1999free,Masmoudi2001IncompressibleIL,VZ}. 
For our problem \eqref{main}--\eqref{main2}, we will present a proof of its incompressible limit in Theorem \ref{thm: incompressible limit} for completeness.

In the incompressible model, $p_\infty$ serves as the Lagrange multiplier corresponding to the constraint $\varrho_\infty \leq 1$.
The boundary of the set $\{p_\infty(\cdot,t)>0\}$ naturally defines a free boundary.
In the tumor growth modeling, it characterizes the time-varying front of the domain inhabited and saturated by the tumor cells. 
In particular, when $\varrho_\infty$ only takes the values $0$ and $1$, which is called a patch solution, the dynamics of $\varrho_\infty$ can be reduced to that of the free boundary.
In this special case, one can also derive the velocity law of the free boundary formally
\beq\lb{0.01}
V = (-\nabla p - {b})\cdot \nu = |\nabla p| - {b} \cdot \nu \quad \text{on }\partial\{p>0\}.
\eeq
Note that this is the same as that for \eqref{main}.
See \cite{kim2018porous,kim2019singular} for discussions on the general non-patch case.

%

\medskip

Both the models \eqref{main}--\eqref{main2} and \eqref{1.2} feature free boundaries.
Numerous studies have addressed the regularity of these boundaries, see for example \cite{CVWlipschitz,C1alpha,kimzhang21,alltime,flatness} for the PME with $m$ fixed,  \cite{CJK,CJK2,dong21,dong2} for the Hele-Shaw, and \cite{kim2022regularity} for general settings with advection and source terms. 
On the other hand, as is mentioned above, the incompressible limit has been well-studied, where the convergence is established on the level of the density $\varrho$ and the pressure $p$.
However, it is not clear whether the free boundaries will exhibit convergence in any good sense as $m\to \infty$.
The primary goal of this paper is to provide an affirmative answer to this question by demonstrating that, under suitable assumptions, for all finite times, these free boundaries converge in the Hausdorff distance as $m\to\infty$.
The proof crucially relies on quantifying the free boundary propagation in terms of spatial average of the pressure, and establishing a uniform-in-$m$ strict expansion property of support of the solutions $\Omega_{p_m}(t):=\{p_m(\cdot,t)>0\}$.
Moreover, we can bound the Hausdorff dimensions of both the free boundaries in the finite-$m$ cases and a certain ``good part" of the free boundary in the limiting case.
In what follows, we shall first introduce each of these results and sketch the ideas of proving them.
Readers are directed to Proposition \ref{L.5.3}, Theorems \ref{T.main_weaker_assumption}--\ref{T.main} and Theorem \ref{thm: estimate Hausdorff dim of FB} for their precise statements.

\subsection{Uniform strict expansion of the support relative to streamlines}

Our first main result is that, under suitable conditions, the supports of the pressure variables strictly expand relative to the streamlines defined by $-b$ (see \eqref{ode}) and uniformly in $m$.
In the seminal paper \cite{CFregularity} which considers the PME, such a property was obtained via a compactness argument under the assumption \eqref{introgr} below, and thus the constants there may depend on $m$. 
To obtain $m$-independent estimates on the strict expansion, we 
prove {\it propagation of the strict expansion property} along the streamlines, i.e., if the free boundary strictly expands relative to the streamlines at the initial time, then it should do so for all finite time and uniformly for all large $m$.
Such a 
property, yet with the constants depending on $m$, was previously employed in some PME-type equations in one space dimension \cite{caffarelli1979regularity,dim1}.

\medskip

Let us explain the strategy of the proof with a highlight of our contribution.
First of all, we prove the classic Aronson-B\'enilan estimate (AB estimate for short) for the general equations \eqref{main}--\eqref{main2} under necessary assumptions; see Proposition \ref{T.3.1}.
It is used to bound the super-harmonicity of $p$ in space and quantify the decay rate of $p$ when moving forward in time along the streamlines.
In particular, it allows us to prove that the support $\Omega_{p}(t)$ is non-decreasing in $t$ with respect to the streamlines. 
Such an estimate was originally observed by \cite{AB} for the PME, providing a pointwise lower bound for $\Delta p$, and it has been extended to many PME-type equations with drifts and source terms (see e.g.~\cite{PQV,kimzhang21,chu2022}). 
A weaker integral version of the AB estimate was proved in \cite{david2021free} for some tumor growth models with nutrients (also see \cite{David_S}). 
Nevertheless, we remark that a pointwise AB estimate is crucial in studying the free boundary regularity.

Secondly, we establish a quantitative relation between propagation of the free boundaries and local spatial average of the pressure in \eqref{main}--\eqref{main2}.
See Lemma \ref{L.4.1} and Lemma \ref{L.4.2}.
Roughly speaking, one can show that a large average pressure can effectively accelerate the motion of the free boundary relative to the streamlines, while a small average hinders that.
Such an argument originates from \cite{CFregularity} on the free boundary regularity of the PME, and it was also applied to the PME with advection in the second author's previous work \cite{kimzhang21}. 
For our purpose, we need to refine this result, not only by ensuring its applicability to the general model with the drift and source terms, but also by proving its uniformity for all large $m>1$, which is a new observation.



Finally, we prove the propagation of the strict expansion property by quantifying the expansion outcomes from \cite{CFregularity} and then promoting the strict expansion from the initial time to all finite positive times.  
The key result is Proposition \ref{L.5.3}.
It shows that, if the supports of $p_m$ strictly expands relative to the streamlines uniformly in $m$ at the initial time, then for any $t_0\in [\eta_0,T)$ with any fixed $\eta_0>0$ and any free boundary point $x_0\in\Omega_{p_m}(t_0)$, if we let $x_0$ move slightly backward in time by $s$ along the streamline, the resulting point must lie outside $\Omega_{p_m}(t_0-s)$, and its distance to $\Omega_{p_m}(t_0-s)$ is at least $Cs^\gamma$ $(\gamma>4)$ which is uniform in $m$, $x_0$ and $t_0$.
In other words, the support $\Omega_{p_m}(t)$ should expand faster relative to the streamlines by a definite amount and this holds uniformly in $m$.
Proposition \ref{L.5.3} also provides a quantitative characterization of weak non-degeneracy of the pressure variable, i.e., spatial average of the pressure near the free boundary must have a uniform-in-$m$ lower bound.
Note that in view of \eqref{0.01}, this heuristically agree with the claim that 
the free boundary should move faster than the convective flow.
The rigorous justification crucially relies on the  above-mentioned results in Section \ref{S3}.


\subsection{Convergence of the free boundaries}
Our second main result is the convergence of the free boundaries. 
Let $f$ and $b$ satisfy suitable conditions, and let $p_m$ solve \eqref{1.1} in $Q_T$ with a non-negative initial data $p_m^0$ which we will assume to be uniformly bounded and uniformly compactly supported in $m$.
Assume $p_m$ to be space-time continuous (see the discussion on its regularity after Definition \ref{def1.1}).
Denote $\Omega_{p_m}(t):=\{p_m(\cdot,t)>0\}=\{\varrho_m(\cdot,t)>0\}$ as before. 
Suppose that
\begin{enumerate}[(i)]
\item 
$\varrho_m^0 = P_m^{-1}(p_m^0)$ converges to some $\varrho^0$ in $L^1(\bbR^d)$, and that the Hausdorff distance between $\Omega_{p_m}(0)$ and $\Omega_{p_l}(0)$ diminishes as the finite $m,l$ go to $\infty$;
\item $\{p_m\}_{m}$ forms a Cauchy sequence in $L^1(Q_T)$, which can be justified in the standard incompressible limit;
and

\item the support of $p_m$ strictly expands relative to streamlines at time $0$ uniformly in $m$ (see more discussions on this in Section \ref{sec: assumptions} and Section \ref{sec: strict expansion}).
\end{enumerate}
Then we can prove convergence of $\Omega_{p_m}(t)$: 
for any $\eta_0\in (0,T)$ 
and $t\in [\eta_0,T)$, 
\[
\text{the Hausdorff distance between }\Omega_{p_m}(t)\text{ and } \Omega_{p_l}(t)\text{ diminishes as $l,m\to\infty$}.
\]
Convergence of the free boundaries is also addressed: after any positive time $\eta_0$, as $l,m\to\infty$,
\beq
\text{the space-time Hausdorff distance between the free boundaries of }{p_m}\text{ and }{p_l}\text{ diminishes}.
\label{eqn: FB convergence}
\eeq
We can further prove convergence results involving the solution of the limit problem as well as its free boundary, which is a bit more subtle nevertheless.
Let $(\varrho_\infty,p_\infty)$ be the weak solution to \eqref{1.2} with the initial data $\varrho^0$, and denote $\Omega_{p_\infty}(t): = \{p_\infty(\cdot,t)>0\}$.
Then we can show that, whenever $m\gg 1$,
\begin{align*}
&\text{$\Omega_{p_\infty}(t)$ is contained in a small neighborhood of $\Omega_{p_m}(t)$, and any free boundary}\\
&\text{point of $p_m$ must lie close to the free boundary of $p_\infty$ in the space-time.}
\end{align*}
However, interestingly, if we exchange $p_\infty$ and $p_m$ in this statement, it fails to hold under the current assumptions; see Remark \ref{rmk: counterexample for FB convergence to p_inf}.
To obtain improved convergence results, we need to additionally assume that $\Omega_{p_m}(0)$ should converge to $\Omega_{p_\infty}(0)$ in the Hausdorff distance as $m\to \infty$.
See the precise statements of the above results in Theorems \ref{T.main_weaker_assumption}--\ref{T.main}.

In \eqref{eqn: FB convergence}, the free boundary of $p_m$ is considered as a space-time set, and the use of the space-time Hausdorff distance instead of the spatial Hausdorff distance at each time is not due to technical difficulties, but it is rather essential. 
Indeed, the drift term in \eqref{main}--\eqref{main2} may induce topological changes of the supports of the solutions, resulting in formation of holes inside the supports. When these holes get filled up, the topological boundaries of the supports will undergo drastic changes, which may lead to a large Hausdorff distance between the free boundaries of the solutions with different indices.
For example, imagine that both $p_m$ and $p_l$ admit a tiny hole at the same spot inside their supports which lies far from their respective exterior boundaries. 
If the holes disappear at slightly different times, even though $\Omega_{p_m}(t)$ and $\Omega_{p_l}(t)$ might be close in the Hausdorff distance at each time instant, $\pa\Omega_{p_m}(t)$ and $\pa\Omega_{p_l}(t)$ can have a large Hausdorff distance.
%
%
%
%
This issue can be addressed by allowing to compare the free boundaries of different solutions at slightly different times.
In fact, we manage to estimate the distance between $\partial \Omega_{p_m}(t)$ and $\pa\Omega_{p_l}(t-s)$ for  $s$ being small.




\medskip

Now let us sketch the ideas behind the proof.
We basically want to upgrade the $L^1(Q_T)$-convergence of 
$p_m$ as $m\to \infty$ to that of the supports of the solutions and the free boundaries.

\begin{enumerate}[(1)]
\item We first show that for any $x_0\in \Omega_{p_m}(t_0)$ with $t_0>0$, it must be close to $\Omega_{p_l}(t_0)$ as long as $m,l\gg 1$.
Although it is not precise, the idea is to trace $x_0$ back to the initial time along the streamline, and study the resulting point $x_0'$.
We can show that, if $x_0$ is not close to $\Omega_{p_l}(t_0)$, $x_0'$ must lie outside the initial support of $p_m$, so the streamline passing through $(x_0,t_0)$ should cross a free boundary point $(x_0'',t_0'')$ of $p_m$ with $t_0''\leq t_0$.
Thanks to the weak non-degeneracy of $p_m$ at $(x_0'',t_0'')$ and the uniform decay estimate for the pressure, we find that a large $d(x_0,\Omega_{p_l}(t_0))$ will lead to a large $\|p_m-p_l\|_{L^1(Q_T)}$, which contradicts with the $L^1(Q_T)$-convergence of the pressures when $m,l\gg 1$.

\item The above result implies that the Hausdorff distance between $\Omega_{p_m}(t_0)$ and $\Omega_{p_l}(t_0)$ should be small whenever $m,l\gg 1$.
We shall improve this to the convergence of the free boundaries.
This requires estimating the distances from a free boundary point of $p_m$ to the space-time set $\{p_l(\cdot,\cdot)>0\}$, and to its complement. 
The former follows from the previous result, while for the latter, it suffices to use the strict expansion property of $p_m$ and the fact that $p_m$ and $p_l$ are close in $L^1(Q_T)$.

\item So far we have studied convergence of the supports and the free boundaries of the pressure variables with large but finite indices.
When it comes to convergence results involving the limiting problem, the basic idea is to pass to the limit in the finite-$m$ case, but several additional difficulties arise.
Firstly, when taking the incompressible limit, the convergence of $p_m$ to the limiting pressure $p_\infty$ is only in the space-time $L^p$-sense, which is relatively weak.
Also, $p_\infty$ is not defined pointwise in the space-time, so in order to discuss $\Omega_{p_\infty}(t)$ and its boundary, we have to specify its pointwise value in a suitable way.
Moreover, several tools described before are not available for the limiting solution.
\end{enumerate}
It is worth highlighting that this argument does not rely on the regularity of the free boundaries, which can have rather complicated behavior when a general drift term is present.

\subsection{Hausdorff dimensions of the free boundaries}
Our last main result is an estimate for the Hausdorff dimensions of the free boundaries for the finite-$m$ problems.
Combining this with the convergence of the free boundaries, we can further conclude that, in the limiting problem with the drift and source terms, a suitably defined ``good part" of the free boundary (see \eqref{6.112}) has finite $(d-1)$-dimensional Hausdorff measure. 
The precise statement is given in Theorem \ref{thm: estimate Hausdorff dim of FB}.

For patch solutions to the Hele-Shaw model with growth, \cite{PQV} proved that the postive set of the density has finite perimeter by deriving a $BV$ estimate for the density, and \cite{MPQ} further proved that its boundary has finite $(d-1)$-dimensional Hausdorff measure.
In \cite{kim2019singular}, the authors used the sup-convolution technique to show that, in a Hele-Shaw-type model with drift and source terms, 
for a certain class of general initial data, the positive set of pressure has finite perimeter; also see \cite{kim2018porous} for the case without drift.
%
%
Our argument is inspired from \cite{kim2019singular}.
However, there are new challenges in our problem.
Firstly, the limiting $\varrho_\infty$ might take the values $0$ and $1$ only (or in the case of \cite{kim2019singular}, the density in the exterior region is assumed to be strictly less than 1 with a positive gap), so the finite $BV$ norm of $\varrho_\infty$ indeed implies the finite perimeter of $\{\varrho_\infty=1\}$; 
whereas for each finite $m$, $\varrho_m$ should be continuous, so $\varrho_m$ having a finite $BV$ norm does not imply finite perimeter of its free boundary, letting alone the issue that the boundary of a finite-perimeter set may not have finite $(d-1)$-dimensional Hausdorff measure \cite[Example 1.10]{Giusti1984MinimalSA}.
Secondly, Lemma 5.1 in \cite{kim2019singular} works only for equations with time-independent advections and sources,  
while we want to deal with more general cases. 
To overcome these difficulties, we apply both the inf- and sup-convolution constructions to show a novel $L^1$-stability of solutions with some perturbed initial data. 
Using this and the weak non-degeneracy again, we find that, with some $d_m$ decreasing to $(d-1)$ as $m\to \infty$, the $d_m$-dimensional Hausdorff measure of the free boundary $\pa\Omega_{p_m}(t)$ is finite. 
Combining this with the convergence of the free boundaries in the Hausdorff distance, we can further deduce that the ``good part" of the limiting free boundary has finite $(d-1)$-dimensional Hausdorff measure.
See the details in Section \ref{S6}.

\subsection{Other related works}
In addition to the abundance of literature listed above, let us mention some other works on various convergence issues of the supports and the free boundaries of solutions in tumor growth and related models.

For \eqref{main}--\eqref{main2} with a fixed $m>1$ and $(b,f) = (\nabla \Phi,0)$ where $\Phi$ is a convex potential, \cite{kimlei} considered the convergence of the free boundary as $t\to +\infty$. 
Later \cite{AKY} proved the incompressible limit of this problem with a subharmonic $\Phi$ and a patch initial data.
It also obtained, among many other results, convergence of the sets $\Omega_{p_m}(t)$ to $\Omega_{p_\infty}(t)$  in the Hausdorff distance \cite[Theorem 3.5]{AKY} by using a viscosity solution approach.

For an incompressible tumor growth model with nutrient, \cite{jacobs2022tumor} proved in the case of zero nutrient diffusion that, under suitable conditions, the support of the patch solution $\varrho$ becomes rounder and rounder as $t\to +\infty$, and its boundary admits $C^{1,\alpha}$-regularity \cite[Corollary 5.5, Corollary 5.15, and Theorem 6.9]{jacobs2022tumor}.
More recently, \cite{Kim2022TumorGW} studied the same model with non-zero nutrient diffusion, with the diffusion coefficient denoted by $D$.
They proved under suitable assumptions that, as $D\to 0$, the free boundary $\partial \{p_D>0\}$ in the finite-$D$ case converge in the Hausdorff distance to that in the zero-diffusion case for every suitably large time. 
Their argument relies on the regularity of the free boundary in the limiting zero-diffusion case.

Let us also mention that, \cite{Liu2017AnAF} developed a numerical scheme to accurately capture the front propagation in the PME-type tumor growth models.
Numerical evidence was provided in some model problems to show the proximity of the free boundaries in the case $m\gg 1$ with the one in the incompressible model.

\subsection{Organization of the paper}
We first introduce our notations and assumptions in Section \ref{sec: preliminaries}.
Some basic results on the model \eqref{main}--\eqref{main2} with finite $m$ are also discussed.
In Section \ref{sec: AB estimate}, we prove the classic AB estimate, and state the result on the incompressible limit of \eqref{main}--\eqref{main2} whose proof will be presented in Appendix \ref{sec: proof of incompressible limit}.
Quantitative relation between spatial average of the pressure and propagation of the free boundaries is established in Section \ref{S3}.
Section \ref{sec: strict expansion} is devoted to the strict expansion property of the support of the solutions: we first look into several conditions that guarantee the strict expansion at the initial time in Section \ref{sec: strict expansion at t=0}, and then show in Section \ref{sec: strict expansion at later time} that such property can propagate to all finite times.
In Section \ref{S3} and Section \ref{sec: strict expansion}, the $m$-dependence in all the estimates are carefully tracked in order to ensure that those results are uniformly applicable to all large $m$. 
We prove the convergence of the supports and the free boundaries in Section \ref{sec: convergence of fb}, and estimate the Hausdorff dimensions of the free boundaries in Section \ref{S6}. 
We highlight once again Proposition \ref{L.5.3}, 
Theorems \ref{T.main_weaker_assumption}--\ref{T.main} and Theorem \ref{thm: estimate Hausdorff dim of FB} as the main results of this paper.
Finally, proofs of two lemmas in Section \ref{sec: strict expansion at t=0} will be provided in Appendices \ref{sec: proof of strict expansion lemma 1} and \ref{sec: proof of L.4.4}, respectively.

\subsection*{Acknowledgement}
The authors would like to thank Inwon Kim and Zhennan Zhou for useful discussions.
J.~Tong is partially supported by the Peking University Start-up Grant. Y.~P.~Zhang is partially supported by the Auburn University Start-up Grant.

\section{Preliminaries}
\label{sec: preliminaries}
\subsection{Notations}

We will use the following notations.
\begin{itemize}
\item Fix $T\in (0,\infty)$, and let $Q_T:=\mathbb{R}^d\times (0,T)$.
\item Let $B(x,r):= \{y\in\mathbb{R}^d:\,|y-x|< r\}$, and $B_r:=B(0,r)$.
\smallskip



\item We write 
\[
\|{ b }\|_{C_{x,t}^{2,1}}:=\sup_{t\in (0,T)} \|b(\cdot,t)\|_{C^2_x(\bbR^d)}+\|\partial_t b(\cdot,t)\|_{C^1_x(\bbR^d)}.
\]
Note that $\|{ b }\|_{C_{x,t}^{2,1}}\geq \sup_{(x,t)\in Q_T}|b(x,t)|$ by the definition. Also denote 
\[
\|f\|_{\dot{C}^{1}_{x,t,p}}:=\sup_{(x,t,p)\in Q_T\times [0,\infty)} |\pa_x f(x,t,p)|+|\partial_tf(x,t,p)|+|\partial_pf(x,t,p)|,
\]
\[
\|f(\cdot,\cdot,0)\|_\infty:=\sup_{(x,t)\in Q_T}|f(x,t,0)|,
\]
and
\[
\|f_+\|_\infty:=\max\left\{\sup_{(x,t,p)\in Q_T\times [0,\infty)} f(x,t,p),\,0\right\}.
\]
Later, we will assume $\|f\|_{\dot{C}_{x,t,p}^1}$, $\|f(\cdot,\cdot,0)\|_\infty$ and $\|f_+\|_\infty$ to be finite.


\smallskip

\item $\nabla b$ denotes the spatial gradient of $b$, $\nabla\cdot b$ denotes the spatial divergence of $b$, and 
\[
\|\nabla b\|_\infty:= \sup_{(x,t)\in Q_T}\|\nabla b(x,t)\|_2.
\] 
Here $\|\cdot\|_2$ denotes the Frobenius norm of matrices.


\item For a continuous, non-negative function $p:Q_T \to \R$, we denote
$$\Omega_p:=\{(x,t)\in Q_T:\,p(x,t)>0\}, \quad \Omega_p(t):=\{p(\cdot,t)>0\}
$$
and
$$
\Gamma_p(t):=\partial\Omega_p(t),\quad \Gamma_p :=\bigcup_{t\in(0,T)}\big(\Gamma_p(t)\times\{t\}\big).
$$
We may omit the subscript $p$ whenever it is clear from the context.

\smallskip

\item For two sets $U,V\subseteq \bbR^d$ (or $\bbR^{d+1}$), the Hausdorff distance between them is defined by
\[
d_H(U,V):=\max\left\{\sup_{x\in U}d(x,V),\,\sup_{y\in V}d(y,U)\right\},
\]
where $d(x,V):=\inf\{|x-y|:\, y\in V\}$.

\smallskip

\item We write
\[
\mint_{B(x,r)}f(y)\, dy:=\frac{1}{|B(x,r)|}\int_{B(x,r)} f(y)\,dy,
\]
where $|B(x,r)|$ is the volume of $B(x,r)$.

\smallskip






\item
Given a suitably smooth $b =b(x,t)$, \textit{streamlines} associated with the convective vector field $-b$ (cf.~\eqref{main})  
are defined as the unique solution $X(x_0,t_0;t)$ of the following ODE: given $x_0\in \bbR^d$ and $t_0\geq 0$,
\begin{equation}\label{ode}
\left\{\begin{aligned}
    \partial_t X(x_0,t_0;t)&={-}{ b }(X(x_0,t_0;t),t_0+t), \quad t\geq -t_0,\\
    X(x_0,t_0;0)&=x_0.
    \end{aligned}\right.
\end{equation}
We shall write $X(t):=X(0,0;t)$.

\smallskip

\item Throughout the paper, we will use $C$, $C_*$, $C_j$ and $c_j$ $(j = 0,1,2,3)$ etc., to denote various \textit{universal constants}, i.e., constants that only depend on (see the assumptions below)
\[
d,\, T,\, \|{ b }\|_{C_{x,t}^{2,1}},\,\|f\|_{\dot{C}_{x,t,p}^1}, \, \|f(\cdot,\cdot,0)\|_\infty,\,\|f_+\|_\infty,\,
R_0
\]
and the constants in the condition. 
In particular, these constants are always independent of $m$ unless otherwise stated. 
Their values may change from line to line.
We will use the notation $C_m$ to represent constants additionally depending on $m$. 
\end{itemize}


\subsection{Assumptions}
\label{sec: assumptions}

We list a few main assumptions needed in the rest of the paper.
Some other special assumptions will be introduced when necessary.
\begin{itemize}
\item Throughout the paper, we will always assume
\beq\lb{1.7}
\|b\|_{C_{x,t}^{2,1}}+\|f\|_{\dot{C}_{x,t,p}^1} +\|f(\cdot,\cdot,0)\|_{\infty}+\|f_+\|_{\infty}<\infty,
\eeq
and
\beq\lb{cond}
\sigma:=\inf_{(x,t,p)\in Q_T\times [0,\infty)}\nabla\cdot{ b }(x,t)+f(x,t,p)-\partial_pf(x,t,p)p>0.
\eeq
These are the key assumptions needed for the Aronson-B\'enilan estimate; see Proposition \ref{T.3.1}.

\item 
We take the initial pressures $p_m^0 = p_m^0(x)$ $(m>1)$ to be continuous in $\bbR^d$ and satisfy
\beq
\sup_{m>1}\sup_{\bbR^d} p_m^0 < +\infty\quad\mbox{and}\quad {\mathrm{supp\,}p_m^0\subset B_{R_0}}
\label{1.8}
\eeq
for some $R_0>0$.
Let (cf.~\eqref{main2})
\beq\lb{main3}
\varrho_m(\cdot,0)= \varrho_m^0(\cdot):=\left(\frac{m-1}{m}p_m^0\right)^\frac1{m-1}
\eeq
be the initial data for \eqref{main}--\eqref{main2}.

\item 
Our convergence result relies on the assumption that $\{p_m\}_{m}$ converges to $p_\infty$ in $L^1(Q_T)$ (see Section \ref{sec: convergence of fb}), where $p_m$ and $p_\infty$ are the pressures in the compressible and the incompressible models respectively.
To verify this, we shall prove (part of) the classic incompressible limit result.
For that purpose, we will also assume 
\beq
\mbox{for some $\varrho^0\geq 0$ with $\mathrm{supp\,}\varrho^0\subset B_{R_0}$, $\varrho_m^0\to \varrho^0$ in $L^1(\bbR^d)$ as $m\to +\infty$},
\label{assumption: L^1 convergence of initial density}
\eeq
and 
\beq
\sup_m \big\|\Delta (\varrho_m^0)^m\big\|_{L^1} + \big\|\na \varrho_m^0\big\|_{L^1}<+\infty.
\label{assumption: uniform BV norm for initial data}
\eeq



\item 
We will need $\Omega_{p_m}(t):=\{p_m(\cdot,t)>0\}$ to be strictly expanding at time $0$ with respect to streamlines and uniformly in $m$. 
For this purpose, we assume that the initial domain $\Omega_{p_m}(0)$ has a Lipschitz boundary, and the initial pressure satisfies the sub-quadratic growth near the free boundary:
\begin{equation}
\label{introgr}
p_m^0(x)\geq \gamma_0 \big(d(x,\Omega_{p_m}(0)^c)\big)^{2-\varsigma_0} \hbox{ for some } \gamma_0>0,~\varsigma_0\in (0,2).
\end{equation}
For the PME, this condition has been known for a long time to imply the strict expansion \cite{aronson1983initially,CFregularity}; when there is drift in the equation, such strict expansion should be understood as that relative to the streamlines \cite{kimzhang21}.
However, \eqref{introgr} is not enough to guarantee the uniformity of strict expansion for all large $m>1$, 
so we shall further assume either one of the following conditions (see Lemma \ref{L.5.11} and Lemma \ref{L.4.4}):
\begin{enumerate}
\item $p_m^0$ satisfies
\beq\lb{R.1.1}
\inf_{x\in\bbR^d} \Delta p_m^0(x)+\nabla\cdot b (x,0)+f\big(x,0,p_m^0(x)\big)\geq 0;
\eeq
or 
\smallskip
    \item $\{\Omega_{p_m}(0)\}_m$ satisfies the uniform interior ball condition, i.e., there exists $r>0$ such that, for any $m>1$ and any $x\in\Gamma_{p_m}(0)$, we can find an open ball $B$ with radius $r$ such that $B\subset \Omega_{p_m}(0)$ and $x\in\overline{B}$.
    Moreover, we need $\sigma> 2d\sup_{x\in\bbR^d}|\nabla b(x,t)|$ for all $t>0$ sufficiently small, where $\sigma$ is from \eqref{cond}.
\end{enumerate}
It is not clear whether the smallness assumption on $\|\nabla b\|_\infty$ in (2) can be removed.

\end{itemize}

Since we are interested in the asymptotics as $m\to\infty$, we will mainly focus on the large-$m$ case in the sequel, although many of our results can be extended to $m>1$ easily.


\subsection{Preliminary results}
\label{sec: prelim results on PME model}
In this subsection, we review some known results on the equations \eqref{main}--\eqref{main2} with $m>1$ fixed.
For brevity, we shall omit the subscripts of $\varrho_m$ and $p_m$ in this part. 
We start from introducing the notion of weak solutions to \eqref{main}--\eqref{main2}.

\begin{definition}\label{def1.1}
Fix $m>1$.
Let $\varrho^0$ be bounded and non-negative, and satisfy $\varrho^0\in L^1(\mathbb{R}^d)\cap L^\infty(\mathbb{R}^d)$.
Let $T>0$. 
We say that a non-negative and bounded $\varrho:\mathbb{R}^d\times [0,T)\to[0,\infty)$ is a subsolution (resp.~supersolution) to \eqref{main} with the initial data $\varrho^0$ if \begin{equation}\label{definitionsol}
 \varrho\in C([0,T),L^1(\mathbb{R}^d))\cap L^2([0,T)\times\mathbb{R}^d)\;\text{ and }\; \varrho^m\in L^2([0,T), \dot{H}^1(\mathbb{R}^d)),
\end{equation}
and for all non-negative $\phi\in C_c^\infty(\mathbb{R}^d\times [0,T))$,
\begin{equation}
\label{def sol2}
\begin{split}
\int_0^T\int_{\mathbb{R}^d} \varrho\,\phi_t\, dx\, dt
\geq (\mbox{resp.\,}\leq) &\; 
-\int_{\mathbb{R}^d} \varrho^0(x)\phi(0,x)\, dx\\
&\;+\int_0^T\int_{\mathbb{R}^d}(\nabla \varrho^m+\varrho b)\nabla\phi-\varrho f(x,t,P_m(\varrho))\phi\, dx\, dt.
\end{split}
\end{equation}
We say that $\varrho$ is a weak solution to \eqref{main} if it is both a sub- and super-solution of \eqref{main}. 
We also say that $p:= P_m(\varrho)$ is a weak solution (resp.~super-/sub-solution) to \eqref{1.1} with the initial data $p(\cdot,0) = p^0$ if $\varrho$ is a weak solution (resp.~super-/sub-solution) to \eqref{main}.

\end{definition}

Existence of the weak solutions to  \eqref{main}--\eqref{main2} has been proved in, for example, \cite{book,IKYZ,chu2022}.
Beyond \eqref{main}--\eqref{main2}, well-posedness of general degenerate parabolic-type equations has been established in e.g.~\cite{alt1983quasilinear,AB,BGHP,Bertsch,chu2022,Jacobs2020TheP,Jacobs2020DarcysLW}. When there is no source term, \cite{IKYZ} proved the uniform-in-time $L^\infty$-estimate of the solutions. \cite{di1982continuity,dib83,IKYZ} proved H\"{o}lder regularity of the solutions. 
Throughout the paper, we will assume that for each $m>1$, $p_m$ is a solution, which is continuous in $\bbR^d\times [0,T)$, to \eqref{1.1} with the initial data $p_m^0$ described in the previous subsection.

We will need the following comparison principle, which also implies the uniqueness of the weak solution. The proof can be found in \cite[Theorem 9.1]{chu2022}.

\begin{theorem}
Let $\underline{\varrho}$ and  $\bar{\varrho}$ be, respectively, a sub-solution and a super-solution to \eqref{main} with bounded, non-negative and compactly supported initial data $\underline{\varrho}^0$ and  $\bar{\varrho}^0$. If $\underline{\varrho}^0\leq \bar{\varrho}^0 $, then $\underline{\varrho}\leq \bar{\varrho}$.
\end{theorem}

It is always convenient to work with classical solutions to \eqref{main}.
The following result states that weak solutions can always be approximated by classical ones.
As a result, once we obtain a priori estimates for smooth solutions, we can conclude that the same estimates hold for weak solutions by taking the limit. 

\begin{lemma}\label{approximation}
Fix $m>1$. 
Let $\varrho$ be a weak solution to \eqref{main} in $Q_T$ with bounded, non-negative initial data $\varrho^0$.
Suppose that ${ b }^k$ and $f^k$ are smooth functions that converge to ${ b }$ and $f$ uniformly in $Q_T$ and $Q_T\times [0,\infty)$  as $k\to \infty$.
Then there exists a sequence of strictly positive classical solutions $\varrho^k$ to \eqref{main}, with $b$ and $f$ replaced by $b^k$ and $f^k$, such that $\varrho^k\to\varrho$ locally uniformly in $\bbR^d\times (0,T)$ as $k\to \infty$.

\end{lemma}
Its proof is standard so we skip it. We refer the readers to \cite[Chapter 3]{book} and \cite{chu2022} for proofs in simpler cases.

\medskip 

In the following lemma, we prove that $\{p_m\}_{m>1}$ are uniformly bounded in $Q_T$, and their supports have a priori uniform bound as well.

\begin{lemma}\label{uniformb}
Assume \eqref{1.8} and that
$\sup_{t\in [0,T)}\|b(\cdot, t)\|_{C^1_x},\, \|f_+\|_\infty<+\infty$. 
Let $p_m$ solve \eqref{1.1} with the initial data $p_m^0$. Then $p_m$ is uniformly bounded in $Q_T$ by a universal constant. 
Moreover, there exists a universal $R = R(t)$ for $t\in [0,T)$, such that $\mathrm{supp\,}p_m(\cdot,t)\subset \overline{B_{R(t)}}$
where $R(t)$ only depends on the universal constants in the assumptions.
\end{lemma}

\begin{proof}
The proof is similar to that of \cite[Lemma 2.1]{David_S}, using a barrier argument.
By \eqref{1.1},
\beq
\partial_t p_m \leq (m-1)p_m(\Delta p_m + \|\nabla \cdot b \|_{\infty} + \|f_+\|_{\infty}) + |\nabla p_m|^2 + |\nabla p_m|\|b\|_{\infty}.
\label{eqn: pressure equation inequality subsolution}
\eeq
Take
\[
\varphi(x,t) := \frac{C}{2}\big(R(t)^2 - |x|^2\big)_+,
\]
with $C>0$ and $R = R(t)$ to be determined.
We want $\varphi$ to satisfy
\[
\partial_t \varphi \geq (m-1)\varphi(\Delta \varphi + \|\nabla \cdot b \|_{\infty} + \|f_+\|_{\infty}) + |\nabla \varphi|^2 + |\nabla \varphi|\|b\|_{L^\infty}.
\]
Since
\begin{align*}
\partial_t \varphi =&\; CR'(t)R(t)\mathds{1}_{\{|x|\leq R(t)\}},\\
\nabla \varphi =&\; -Cx\mathds{1}_{\{|x|\leq R(t)\}},\\
\Delta \varphi =&\; -Cd\mathds{1}_{\{|x|\leq R(t)\}} + CR(t)\delta_{\{|x|= R(t)\}},
\end{align*}
it suffices to choose $C$ and $R(t)$ such that
\[
C \geq d^{-1}\big(\|\nabla \cdot b \|_{L^\infty} + \|f_+\|_{\infty}\big)\quad
\mbox{and}\quad
R'(t) = CR(t) + \|b\|_{L^\infty}.
\]
In addition, if we take $R(0)$ to be suitably large so that {$\varphi(x,0)\geq p_m^0(x)$} (cf.\;\eqref{1.8}), we conclude that $\varphi(x,t)\geq p_m(x,t)$ for all $t\in [0,T]$ by \cite[Lemma 2.6]{kimzhang21} and the comparison principle.
Since $\varphi$ is bounded, compactly supported, and independent of $m$, this proves the desired claim.
\end{proof}

\begin{remark}
In view of Lemma \ref{uniformb}, some assumptions of the main theorems can be weakened. 
For instance, the $C^1$-seminorm of $f$ in the assumption \eqref{1.7} may be restricted to the region $p\in [0,\sup_{m>1}\|p_m\|_{L^\infty(Q_T)}]$ instead of the whole state space. 
Secondly, although we did not assume $f$ to be bounded (from below) in the state space (cf.~\eqref{1.7}), the boundedness of $p_m$ and the assumption $\|f\|_{\dot{C}_{x,t,p}^1}+\|f(\cdot,\cdot,0)\|_{\infty}<\infty$ actually implies that $f(x,t,p_m)$ in the region of interest is uniformly bounded.
Therefore, in the sequel, we shall simply assume $f$ to be bounded in the $(x,t,p)$-state space without loss of generality, i.e., $\|f\|_{\infty}<+\infty$.

Besides, 
instead of \eqref{cond}, it suffices to assume
\[
\inf_{(x,t,p)\in B_R\times [0,T]\times[0,C]}\nabla\cdot{ b }(x,t)+f(x,t,p)-f_{p}(x,t,p) p>0
\]
for some sufficiently large $R=R(T)$ and $C=C(T)>0$. 
\end{remark}

The next result is standard for the PME-type tumor growth models. 
\begin{theorem}
\label{thm: properties of solutions to PME}
Assume \eqref{1.8}--\eqref{main3}.
Also assume $\|f\|_{\infty}+\sup_{t\in [0,T)}\|b(\cdot, t)\|_{C^1_x}<+\infty$ and $\partial_p f\leq 0$.
Let $\varrho_m$ be the continuous solution to \eqref{main} in $Q_T$ with the initial data $\varrho_m^0$.
Then 
\begin{enumerate}
\item  $t\mapsto \int_{\bbR^d}\varrho_m(x,t)\,dx$ is uniformly Lipschitz continuous in $t\in [0,T)$ for all $m>1$;
\medskip

\item Suppose $\varrho_m'$ is another solution to \eqref{main} in $Q_T$ with 
the initial data $\varrho_m'(\cdot,0)$ satisfying \eqref{1.8}--\eqref{main3} as well.
Then there exists $C$ independent of $m$ such that, for all $t\in [0,T)$, 
\[
\left|\int_{\bbR^d}(\varrho_m-\varrho_m')(x,t)\, dx\right|\leq C\int_{\bbR^d}|\varrho_m-\varrho_m'|(x,0)\, dx.
\]
\end{enumerate}
\end{theorem}
We shall omit its proof; one may follow the argument in \cite{PQV} which studies a simpler case.



\section{The Aronson-B\'enilan Estimate}
\label{sec: AB estimate}
In this section, we establish the classic AB estimate, which is a semi-convexity estimate for the pressure variable $p_m$, with explicit dependence on $m$. 
In the following proposition, we allow $p^0$ to be discontinuous and have unbounded support.

\begin{proposition}\label{T.3.1}
Assume \eqref{1.7} and \eqref{cond}, and let $p_m\in L^\infty(Q_T)$ be a solution to \eqref{1.1} with non-negative initial data $p_m^0$ such that $(p_m^0)^{\frac{1}{m-1}}\in L^1(\mathbb{R}^d)\cap L^\infty(\mathbb{R}^d)$. Then there exists a constant $C_0$ independent of $m$ and $T$ such that
\begin{equation}\label{ineq fund}
    \Delta p_m(x,t)+\nabla\cdot{ b }(x,t)+f(x,t,p_m(x,t)) \geq -\frac{1}{m-1}\left({C_0}+\frac{1}{t}\right)
\end{equation}
in $Q_T$ in the sense of distribution.
Here the constant $C_0$ has the expression
\beq\lb{3.1}
C_0=C_d\left(1+\sigma^{-1}\right)\left(1+\|p_m\|_{\infty}\right)\left(1+\|\pa_p f\|_\infty\right)
\left(1+\|{ b }\|_{C_{x,t}^{2,1}}^2+\|f\|_{\dot{C}_{x,t}^1}^2+\|f\|_\infty\right),
\eeq
where $C_d$ is dimensional, $\|f\|_{\dot{C}_{x,t}^1}:=\underset{{Q_T\times [0,\|p_m\|_\infty]}}{\sup}\left[|\na_x f|+|\partial_t f|\right]$ and $\|f\|_\infty:=\underset{{ Q_T\times [0,\|p_m\|_\infty]}}{\sup}|f|$.

\end{proposition}

\begin{proof}
In view of Lemma \ref{approximation}, it suffices to consider smooth ${ b }$ and $f$, and strictly positive smooth solutions.  Indeed, if \eqref{ineq fund}--\eqref{3.1} hold for the approximate smooth solutions, the conclusion follows by passing to the limit.


Assume that $p_m$ is strictly positive and smooth.
Let
\[
q(x,t):=\Delta p_m(x,t)+F(x,t,p_m(x,t)),\quad F(x,t,p_m(x,t)):=\nabla\cdot{ b }(x,t)+f(x,t,p_m(x,t)).
\]
$F$ is uniformly bounded since $p_m$ is uniformly bounded.
For simplicity, let us write 
\begin{align*}
&f_t:=\partial_tf(x,t,p)|_{p=p_m(x,t)},\quad  f_p:=\partial_pf(x,t,p)|_{p=p_m(x,t)},\\
&
F_t:=\na\cdot \pa_t b+f_t,\quad 
F_x:=\na\cdot \pa_x b + \pa_xf(x,t,p)|_{p = p_m(x,t)}.
\end{align*}
Then by \eqref{1.1} and direct calculation,
\beq
\partial_t p_m =(m-1)p_m q+\nabla p_m\cdot(\nabla p_m+b ),\lb{3.2}
\eeq
and 
\begin{align*}
&\partial_t \left[f(x,t,p_m(x,t))\right]=f_t+f_p\partial_t p_m =f_t+(m-1)f_{p}\,p_mq+f_{p}\nabla p_m\cdot(\nabla p_m+ b ),\\
&\nabla [F(x,t,p_m(x,t))] = F_x +f_p\nabla p_m.
\end{align*}
Now using \eqref{1.1} and the notation $(p_m)_i:=\partial_{x_i}(p_m)$, we get
\beq\lb{3.3}
\begin{aligned}
    q_t&=F_t+f_{p}(p_m)_t+(m-1)p_m\Delta q+2(m-1)\nabla p_m\nabla q+(m-1)q\Delta p_m+2\sum_{i,j} (p_m)_{ij}{{b}^i_j}\\
    &\quad +\nabla \Delta p_m\cdot b +\nabla p_m\cdot \Delta{{ b }}+2\nabla p_m\nabla\Delta p_m+2\sum_{i,j}| (p_m)_{ij}|^2\\
    &=F_t+f_p\big((m-1)p_mq+\nabla p_m\cdot(\nabla p_m+b )\big)+(m-1)\big(p_m\Delta q+q(q-F)\big)+2(m-1)\nabla p_m\nabla q\\
    &\quad +2\nabla p_m \nabla (q-F)+2\sum_{i,j} (p_m)_{ij}{{b}^i_j}+2\sum_{i,j}| (p_m)_{ij}|^2+\nabla (q-F)\cdot{ b }+\nabla p_m\cdot \Delta{{ b }}\\
    &=(m-1)(p_m\Delta q+q(q-F+f_{p}\,p_m))+2m\nabla p_m\nabla q-2\nabla p_m \cdot F_x+2\sum_{i,j} (p_m)_{ij}{{b}^i_j}\\
    &\quad +2\sum_{i,j}| (p_m)_{ij}|^2+\nabla q\cdot{ b }- F_x\cdot{ b }+\nabla p_m\cdot \Delta{{ b }}+F_t-f_p|\nabla p_m|^2\\
    &\geq (m-1)(p_m\Delta q+q(q-F+f_{p}\,p_m))+2m\nabla p_m\nabla q-\eps|\nabla q|^2- (1+f_{p})|\nabla p_m|^2-A_\eps\\
    &=:\calL_m(q),
\end{aligned}
\eeq
where we used the Young's inequality, and
\[
A_\eps:= \sup_{(x,t)\in \bbR^d\times [0,T)}\Big| 2| F_x|^2+\sum_{i,j}|b_j^i|^2/2+|{ b }|^2/(4\eps)+F_x\cdot{{ b }}+|\Delta{ b }|^2/2-F_t\,\Big|.
\]
We shall view $p_m>0$ as a known function, so $\calL_m$ in \eqref{3.3} is a quasilinear elliptic operator.

Let us now apply a barrier argument to show that $q$ is uniformly bounded from below for all $m>1$.
With some $\tau,C_1,C_2>0$ to be chosen, such that $C_1\geq C_2\|p_m\|_\infty$, we set
\[
w:=-\frac{C_1-C_2p_m}{m-1}-\frac{1}{(m-1)(t+\tau)}.
\]
It is clear that $w\leq -\frac{1}{(m-1)(t+\tau)}<0$, and, since $p_m^0$ is smooth, by taking $\tau>0$ to be sufficiently small, we have $q\geq w$ at $t=0$.
Next since $F-f_{p}\,p_m\geq \sigma$ by the assumption and $\Delta p_m=q-F$, we obtain
\begin{align*}
&\;(m-1)(p_m\Delta w+w(w-F+f_{p}\,p_m))\\
=&\;C_2p_m\Delta p_m+(m-1)w^2+(m-1)w(-F+f_{p}\,p_m)\\
\geq &\; C_2p_m q-C_2p_mF+\frac{1}{(m-1)(t+\tau)^2}-(m-1)w\sigma\\
\geq &\; \frac{1}{(m-1)(t+\tau)^2}+C_2p_m q-C_2\|p_m\|_\infty\|F\|_\infty+(C_1-C_2\|p_m\|_\infty)\sigma\\
&\; +\frac{C_2}{m-1}\nabla p_m\cdot(\nabla p_m+{ b })-\frac{C_2}{m-1}\left(\frac32|\nabla p_m|^2+\frac12\|{ b }\|_\infty^2\right).
\end{align*}
Thus, also using \eqref{3.2}, we get for $\calL_m$ from
the last line of \eqref{3.3},
\begin{align*}
\calL_m(w)&\geq \frac{1}{(m-1)(t+\tau)^2}+\frac{C_2}{m-1}\big((m-1)p_m q+\nabla p_m\cdot(\nabla p_m+{ b })\big)+(C_1-C_2\|p_m\|_\infty)\sigma\\
&\qquad -\frac{C_2}{m-1}\left(\frac32|\nabla p_m|^2+\frac12\|{ b }\|_\infty^2\right)-C_2\|p_m\|_\infty\|F\|_\infty+\frac{2C_2m}{m-1}|\nabla p_m|^2-\frac{ C_2^2\eps}{(m-1)^2}|\na p_m|^2\\
&\qquad -(1+f_p)|\nabla p_m|^2-A_\eps\\
&\geq \frac{1}{(m-1)(t+\tau)^2}+\frac{C_2}{m-1}(p_m)_t+(C_1- C_2\|p_m\|_\infty)\sigma\\
&\qquad +\left(\Big(2+\frac{1}{2(m-1)}\Big)C_2-\frac{C_2^2\eps}{(m-1)^2}-1-\|f_p\|_\infty\right) |\nabla p_m|^2- A_\eps',
\end{align*}
where $A'_\eps:=A_\eps+\frac{C_2}{2(m-1)}\|{ b }\|_\infty^2+C_2\|p_m\|_\infty\|F\|_\infty$.

Now we set $C_1:=C_2\|p_m\|_\infty+A_{\eps}'/\sigma$, and define
\begin{align*}
&C_2:=1+\|f_p\|_\infty,\quad \eps:=1/(1+\|f_p\|_\infty) &&\text{if }m\geq 2, \\
&C_2:=4(m-1)(1+\|f_p\|_\infty),\quad \eps:=1/(16(1+\|f_p\|_\infty)) &&\text{if }m\in(1,2).
\end{align*}
Note that $\|p_m\|_\infty$, $\frac1\eps$ and $\frac{C_2}{m-1}$ are bounded from above by a constant independent of $m$, and so is $A'_\eps$.
With these choices of parameters, it follows that
\[
w_t=\frac{1}{(m-1)(t+\tau)^2}+\frac{C_2}{m-1}(p_m)_t \leq \calL_m(w).
\]
Recall \eqref{3.3} and $q\geq w$ at $t=0$. Therefore by the comparison principle, we conclude that
\[
\Delta p_m+\nabla\cdot b +f
=q\geq w 
\geq -\frac{1}{m-1}\left(C_1+\frac{1}{t+\tau}\right)
\geq -\frac{1}{m-1}\left(C_1+\frac{1}{t}\right),
\]
which is \eqref{ineq fund} for smooth solutions for all $m>1$.

Finally, \eqref{3.1} is obtained via tracking the dependence of $C_1$. We comment that $\|F\|_\infty\leq \|b\|_{C^1_{x,t}}+\|f\|_\infty$ and $\|f\|_\infty\leq \|f(\cdot,\cdot,0)\|_\infty+\|f_p\|_\infty\|p_m\|_\infty$.
\end{proof}

\begin{remark}\lb{R.1}



Improvements of \eqref{ineq fund} are possible under strong assumptions. 
\begin{enumerate}[(1)]
\item If we further assume \eqref{R.1.1}, i.e., $q(x,0)\geq 0$,
then \eqref{ineq fund} can be improved to become
\begin{equation}\lb{R.1.2}
    \Delta p_m(x,t)+\nabla\cdot{ b }(x,t)+f(x,t,p_m(x,t)) \geq -\frac{C_0}{m-1}
    \quad \hbox{ in }\mathbb{R}^d\times (0,T).
\end{equation}
Indeed, it suffices to take $\tau\to +\infty$ in the above proof.

\item If $b\equiv 0$ and $f(x,t,p)=f(p)$, then instead of \eqref{cond}, one can assume
\beq\lb{cond'}
f(p),\,-f_p(p)\geq 0.
\eeq
This is because, under the new condition, 
\eqref{3.3} gives
\begin{align*}
q_t\geq (m-1)(p_m\Delta q+q(q-F+f_{p}\,p_m))+2m\nabla p_m\nabla q- f_p|\nabla p_m|^2=:\calL_m(q).
\end{align*}
Let $w$ be the same as before.
By \eqref{cond'} and picking $C_1:= 2C_2\|p_m\|_\infty$, we find
\[
(m-1)w(-F+f_pp_m)-C_2p_mF\geq 0,
\]
and thus,
\[
\calL_m(w)\geq \frac{C_2}{m-1}(p_m)_t+\frac{1}{(m-1)(t+\tau)^2}+\left(\frac{2m-1}{m-1}C_2-{f_p}\right) |\nabla p_m|^2 \geq w_t.
\]
The rest of the proof is identical.
\end{enumerate}
\end{remark} 

Next we state a monotonicity property of the positive set of a solution along the streamlines over time. Recall that $\Omega_{p_m}(t)=\{p_m(\cdot,t)>0\}$.

\begin{lemma}\label{streamline}
For $m>1$, let $p_m$ solve \eqref{1.1}. Then for any $x_0\in\Omega_{p_m}(t_0)$ with $t_0>0$, 
\begin{equation}\label{mono stream}
    p_m(X(x_0,t_0;s),t_0+s)\geq e^{-C_{t_0}s}p_m(x_0,t_0)>0,
\end{equation}
where $C_{t_0}:=C_0+\frac{1}{t_0}$.
Consequently, for $X(x,t;s)$ given in \eqref{ode},
\[
\{X(x,t;s)\,|\, x\in {\Omega_{p_m}(t)}\}\subseteq \Omega_{p_m}(t+s) \quad \hbox{for all } s,t>0.
\]
\end{lemma}

\begin{proof}
Fix $x_0\in\Omega_{p_m}(t_0)$ with $t_0>0$. It suffices to consider smooth approximations of $p_m$, and prove that for any $s>0$, $p_m(X(x_0,t_0;s),t_0+s)$ has a positive lower bound independent of the approximations.

By Proposition \ref{T.3.1}, we have $\Delta p_m+\nabla\cdot{ b }+f\geq -\frac{C_{t_0}}{m-1}$ for $t\geq t_0$. It follows from \eqref{1.1} that for all $s>0$,
\begin{align*}
   \partial_s p_m(X(x_0,t_0;s),t_0+s)= ((p_m)_t -\nabla p_m\cdot { b })(X(x_0,t_0;s),t_0+s)
    \geq -C_{t_0} p_m(X(x_0,t_0;s),t_0+s),
\end{align*}
which yields \eqref{mono stream}.
\end{proof}

Now we state a result on the incompressible limit of the system \eqref{main}--\eqref{main2} as $m\to +\infty$.
Let us point out that this is mainly for obtaining the $L^1$-convergence of $p_m$ to $p_\infty$ in $Q_T$, which will be used as a key assumption in Section \ref{sec: convergence of fb} to show the convergence of the free boundaries.
As a result, the following theorem, as well as the conditions associated to it, can be replaced by any result that would imply the $L^1(Q_T)$-convergence of the pressure.
Here for simplicity, we only present the incompressible limit result without justifying the complementarity condition (i.e.\;the second equation in \eqref{1.2}), as that part is not needed for proving the pressure convergence.
Interested readers may consult the literature mentioned in Section \ref{sec: intro} for more in-depth discussions on the incompressible limit.

\begin{theorem}
\label{thm: incompressible limit}
Assume \eqref{1.7}, and that $|\pa_{pp}f |+ |\pa_{tp}f|$ is locally finite in $Q_T\times [0,+\infty)$. 
Let $\varrho_m^0$ and $\varrho^0$ satisfy \eqref{1.8}--\eqref{assumption: uniform BV norm for initial data}.
Additionally, we assume either (a) $\pa_p f\leq 0$ and \eqref{R.1.1} hold; or (b) $\pa_p f\leq -\alpha$ for some $\alpha>0$.

Let $\varrho_m\geq0$ solve \eqref{main} in $Q_T$ with the initial data $\varrho_m^0$. 
Then there exists a unique weak solution $(\varrho_\infty, p_\infty)$ to 
\begin{align*}
&\partial_t\varrho_\infty =\D p_\infty + \na\cdot (\varrho_\infty  b )+\varrho_\infty f(x,t,p_\infty)\mbox{\;\;in distribution,}\\
&\varrho_\infty\leq 1,\quad
p_\infty(1-\varrho_\infty)=0\mbox{\;\;almost everywhere}
\end{align*}
in $Q_T$ with the initial data $\varrho_\infty(x,0) = \varrho^0(x)$, satisfying that 
\begin{enumerate}[(i)]
\item $\varrho_\infty , p_\infty\in L^\infty\cap BV (Q_T)$, and $\na p_\infty \in L^2(Q_T)$;
\item $\varrho_\infty$ and $p_\infty$ are compactly supported in $\bbR^d\times [0,T]$;
\item for any $q\in [1,+\infty)$, 
\[
\mbox{$\varrho_m\to \varrho_\infty$ in $L^q(Q_T)$, and $p_m\to p_\infty$ in $L^q(Q_T)$ as $m\to +\infty$}.
\]
\end{enumerate}
In particular, $\{p_m\}_m$ converges to $p_\infty$ in $L^1(Q_T)$.
\end{theorem}

As is mentioned before, the incompressible limit has been justified for various special cases of \eqref{main}--\eqref{main2}.
%
For example, this has been proved under the conditions that
\[
\text{$\{\varrho_m^0\}_{m>1}$ and $\{p_m^0\}_{m>1}$ satisfy suitable uniform bounds, and }
\lim_{m,l\to\infty} \big\|\varrho_m^0-\varrho_l^0\big\|_{L^1}=0,
\]
and either one of the following assumptions holds:
\begin{enumerate}
\item $b\equiv 0$, and $f=f(p)$ being suitably smooth satisfies that $f_p(p)<0$ and $f(p_M)=0$ for some $p_M>0$
\cite[Theorem 2.1]{PQV};
\smallskip

\item $b=\nabla \Phi(x,t)$ is suitably smooth, and $f = f(p)$ satisfies the same assumptions as in the previous case 
\cite[Theorem 1.1]{David_S}.
\smallskip

\item \eqref{main2} is modified into a more general form, and $b = b(x,t)$ and $f=f(x,t)$ are smooth 
\cite[Theorem 2.5]{chu2022}.

\end{enumerate}
However, we assumed $b$ and $f$ to have more general forms in \eqref{main}--\eqref{main2}.
Although the proof is standard, for the sake of completeness, we shall present it in Appendix \ref{sec: proof of incompressible limit}.

\section{Expansion of Positive Sets along Streamlines}\lb{S3}

In this section we study finer properties on the expansion of the positive set $\{p_m>0\}$ along the streamlines determined by the drift $b$.

The idea originates from \cite{CFregularity}, which studied the PME, and it is used later in \cite{kimzhang21}. The key step is to measure the time the free boundary moves away from a given point by a distance $R$, in terms of the average of the pressure in a ball of size $R$. Then one is able to obtain a Hausdorff distance estimate of the free boundaries in terms of the local spatial $L^1$-norm of the pressure. 
More importantly, we observe that the constants in this property are independent of $m$, making it possible to study the convergence of the free boundaries as $m\to\infty$.

In this section, we will drop the subscript $m$ from $p_m$, but the dependence of constants on $m$ will be tracked carefully. The condition \eqref{cond} is assumed in the following lemmas only for the purpose of having the conclusions from Proposition \ref{T.3.1}; see Remark \ref{R.2}.

\begin{lemma}\label{L.4.1}
Assume \eqref{1.7} and \eqref{cond}. Let $m\geq2$, and let $p=p_m$ be given as in Proposition \ref{T.3.1}. There exists a universal constant $c_0\ll 1$ such that, for any $\eta_0>0$,
the following holds for all $t_0\geq \eta_0$ and $x_0\in\bbR^d$
with 
$\tau\leq \min\{c_0,c_0(m-1)\eta_0,\eta_0\}$:
for any given $R>0$, 
if
\begin{equation}\label{if}
p(\cdot,t_0)=0 \hbox{ in } B(x_0,R)\quad \hbox{and}\quad\mint_{B(X(x_0,t_0;\tau),R)}p(x,t_0+\tau)\, dx\leq \frac{c_0R^2}{\tau},
\end{equation}
then
\begin{equation}\label{then}
p(x,t_0+\tau)=0\quad \text{for}\quad x\in B(X(x_0,t_0;\tau),R/6).
\end{equation}
\end{lemma}
\begin{proof}
Without loss of generality, we suppose $x_0=0$ and shift $t_0$ to $0$. Let us consider the re-scaled the pressure variable $\bar{p}(x,t):=\frac{\tau}{R^2}p(Rx,\tau t)$, which satisfies
\[
\bar{p}_t=(m-1)\bar{p}(\Delta \bar{p}+\nabla\cdot \bar{b}+\bar{f})+|\nabla \bar{p}|^2+\nabla \bar{p}\cdot \bar{b}.
\]
Here
\begin{equation}\label{rescale v}
    \bar{ b }(x,t):=\frac{\tau}{R} { b }(Rx,\tau t)\quad\text{and}\quad \bar{f}(x,t,\bar{p}):=\tau f\left(Rx,\tau t, \frac{R^2}{\tau}\bar{p}\right).
\end{equation}
We also denote 
\[
\bar{X}(t):=\frac{1}{R}X(0,0;\tau t),\quad v(x,t):=\bar{p}(x+\bar{X}(t),t)
\]
where $v$ satisfies
\beq\lb{4.3}
{v}_t - (m-1){v}(\Delta {v}+\bar{F})- |\nabla {v}|^2-\nabla{v}\cdot (\bar{ b }(x+\bar{X},t)-\bar{ b }(\bar{X},t))=0
\eeq
with $\bar{F}(x,t,v):=\nabla\cdot \bar{ b }(x+\bar{X},t)+\bar{f}(x+\bar{X},t,v)$.

\medskip

From the assumption \eqref{if} and the change of variables, it follows that
\beq\lb{4.2}
\mint_{B_1}v(x,1)\, dx=\mint_{B(\bar{X}(1),1)}\bar{p}(x,1)\, dx\leq c_0.
\eeq
Next, having in mind that $t_0\geq \eta_0$ has been shifted to $0$, we apply Proposition~\ref{T.3.1} to find that, for any $t\geq 0$,
\beq
\Delta p+\nabla\cdot{ b }+f\geq -\frac{1}{m-1}\left(C+\frac{1}{\eta_0}\right) =: -\frac{1}{m-1}C_{\eta_0}.
\label{eqn: define C_eta_0}
\eeq
%
Hence, 
\begin{equation}\label{bound00}
\Delta v+\bar{F}\geq -\frac{\tau}{m-1}C_{\eta_0}, 
\end{equation}
and thus, for some universal $C>0$,
\beq\lb{def.eps}
\Delta v \geq  -\frac{\tau}{m-1}C_{\eta_0}-\bar{F} \geq -\eps\quad\text{with}\quad\eps:=\left(\frac{1}{(m-1)\eta_0}+C\right)\tau.
\eeq
Here we took $C$ in the definition of $\eps$ to be suitably large so that $|\bar{F}|\leq C \tau \leq \eps$; we will use this fact later. 
In addition, by the assumption on $\tau$, we can make $\eps\in (0,1)$ by taking $c_0$ to be small.
Observe that $v+\eps |x|^2/(2d)$ is non-negative and subharmonic thanks to \eqref{def.eps}. By the Harnack's inequality and \eqref{4.2}, for all $x\in B_{1/2}$,
\begin{equation}\label{est v t1}
\begin{aligned}
v(x,1)&
\leq -\frac{\eps|x|^2}{2d}+Cv(0,1)\\
&\leq -\frac{\eps|x|^2}{2d}
+ C\mint_{B_1}v(y,1)+\frac{\eps|y|^2}{2d}\, dy\leq C(c_0+\eps),
\end{aligned}
\end{equation}
where $C$ is some dimensional constant.

Note that $v$ is smooth in its positive set thanks to the classic parabolic theory.  So it follows from \eqref{4.3} and \eqref{bound00} that, in the positive set of $v$,
\begin{equation*}
\begin{aligned}
    v_t(x,t) & =(m-1){v}(\Delta {v}+\bar{F})+|\nabla {v}|^2+\nabla{v}\cdot (\bar{ b }(x+\bar{X},t)-\bar{ b }(\bar{X},t))\\
    &\geq -C_{\eta_0}\tau {v}+|\nabla {v}|^2-|\nabla{v}||x|\|\nabla\bar{ b }\|_\infty.
\end{aligned}
\end{equation*}
Due to Young's inequality, $\|\nabla\bar{ b }\|_\infty\leq C\tau\leq C\eps$, in the positive set of $v$,
\begin{equation}\lb{4.4}
\begin{aligned}
    v_t(x,t) & \geq -C_{\eta_0}\tau {v}-|x|^2\|\nabla\bar{ b }\|_\infty^2
    \geq -C_{\eta_0}\tau {v}-C\eps^2|x|^2,
\end{aligned}
\end{equation}
Also, because $v$ is continuous and non-negative, the same inquality holds weakly in $\bbR^d$.
Since $\eps\in(0,1)$ and $C_{\eta_0}\tau\leq 1+C\tau\leq C$, the Gronwall's inequality implies that
\[
{v}(x,1)\geq e^{-C(1-t)}{v}(x,t)-C\eps^2|x|^2(1-t) \geq e^{-C}{v}(x,t)-C\eps^2\quad \hbox {in } B_{1/2}\times (0,1).
\]
Combining this with \eqref{est v t1} yields for all $(x,t)\in B_{1/2} \times (0,1)$ and for some $C_1\geq 1$,
\begin{equation}\label{est v 111}
\begin{aligned}
{v}(x,t)\leq e^{C}(v(x,1)+C\eps^2)
\leq C_1(c_0+\eps),
\end{aligned}
\end{equation}
i.e., $v$ is uniformly small in $B_{1/2}\times(0,1)$. 

From here, we proceed with a barrier argument to conclude the lemma. 
For $t\in (0,1)$, we denote $\sigma(t):=C_1(c_0+\eps)(1+ \frac{t}4)$, and $r(t):=\frac{1}{3}-\frac{t}{6}$; besides, define
\[
\Sigma:=\left\{(x,t)\,|\,x\in B_{\frac12}\setminus B_{r(t)},\,t\in (0,1) \right\}.
\]
Let $\varphi(x,t)$ be the solution to
\[
\left\{
\begin{aligned}
-\Delta\varphi&=\eps \quad &&\text{ in }\Sigma,\\
\varphi&=C_1(c_0+\eps)(1+ t/4)\quad &&\text{ on }\partial B_{\frac12},\\
\varphi&=0\quad &&\text{ on }\partial B_{r(t)}.
\end{aligned}
\right.
\]
We also define $\varphi=0$ for $x\in B_{r(t)}$ and $t\in (0,1)$.

We will show that $\varphi$ is a supersolution to \eqref{4.3} in $B_{1/2}\times (0,1)$.
Let us only consider the case when $d\geq 3$. From the equation of $\varphi$, it is easy to obtain that
\[
\varphi(x,t)=a_1(t)|x|^{2-d}+a_2(t)-\frac{\eps}{2d}|x|^2,
\]
where
\[
a_1(t):=\frac{ \frac{\eps r(t)^2}{2d}-\frac{\eps}{8d}-\sigma(t)}{r(t)^{2-d}-2^{d-2}}\quad\text{and}\quad a_2(t):=\sigma(t)+\frac{\eps}{8d}-a_1(t) 2^{d-2}.
\]
When $d=2$, $\varphi$ takes the form
\[
\varphi(x,t)=a_1(t)\ln|x|+a_2(t)-\eps|x|^2/4,
\]
and the rest of the argument is similar.

Let us drop the $t$-dependence from the notations of $a_1(t)$, $\sigma(t)$ and $r(t)$. Note that  $r\in(\frac16,\frac13)$ and $\sigma\geq \sigma(0)=C_1(c_0+\eps)\geq \eps$. 
By direct calculation, we get
\begin{align*}
 a_1'&=(r^{2-d}-2^{d-2})^{-1}\left(\frac{\eps rr'}{d}-\sigma'\right)+(r^{2-d}-2^{d-2})^{-2}(d-2)r^{1-d}r'\left(\frac{4\eps r^2-\eps}{8d}-\sigma\right)\\
&\geq (r^{2-d}-2^{d-2})^{-1}\left(-\frac{\eps r}{6d}-\frac\sigma4+\frac{d-2}2 \left(\frac{\eps}{16d}+\sigma\right)\right)  \geq 0.
\end{align*}
Here we used the facts that $r' = -\frac{1}{6}$, $\sigma' = \frac{1}{4}\sigma(0)\leq \frac{1}{4}\sigma$, and $4r^2 \leq \frac{4}{9} <\frac{1}{2}$.
Then we further derive that
\[
\varphi_t=\sigma'+(|x|^{2-d}-2^{d-2})a'_1 \geq \sigma'=C_1(c_0+\eps)/4\quad\text{in }\Sigma.
\]
Since $|a_1(t)|\leq C(\eps+\sigma(t))\leq C(c_0+\eps)$ for $t\in (0,1)$, there exists a universal constant $C=C(C_1)>0$, such that
\[
|\nabla\varphi|\leq C(|a_1|+\eps)\leq C(c_0+\eps)\quad\text{ in }B_{1/2}\times (0,1).
\]
Morover, by \eqref{def.eps}, for $(x,t)\in B_1\times(0,1)$,
\[
|\bar{F}|\leq \eps\quad\text{and}\quad |\bar{ b }(x+\bar{X},t)-\bar{ b }(\bar{X},t)|\leq \|\nabla{ \bar{b} }\|_\infty|x|\leq C\eps.
\]
Combining the above estimates, we find in $\Sigma$ that
\begin{align*}
&\;\varphi_t - (m-1)\varphi(\Delta \varphi+\bar{F})- |\nabla \varphi|^2-\nabla\varphi\cdot (\bar{ b }(x+\bar{X},t)-\bar{ b }(\bar{X},t))\\
\geq &\; {C_1(c_0+\eps)}/{4}+(m-1)\varphi(\eps-\bar{F})-C(c_0+\eps)^2-C(c_0+\eps)\eps\\
\geq &\; {C_1(c_0+\eps)}/{4}-C(c_0+\eps)^2,
\end{align*}
which is non-negative provided that $(c_0+\eps)\leq \frac{C_1}{4C}$.
This is achieved if we take $\tau$ as in the assumption and let $c_0$ be sufficiently small and yet universal.
Therefore, we conclude that $\varphi$ is a supersolution to \eqref{4.3} in $\Sigma$. In view of \cite[Lemma 2.6]{kimzhang21}, $\varphi$ is also a supersolution in $B_{1/2}\times (0,1)$.

By the assumption $v(x,0) = 0$ in $B_{1/2}$ and thus ${v}\leq \varphi$ on $\{|x|\leq \frac{1}{2},\,t=0\}$.
On the lateral boundary, \eqref{est v 111} and the equation of $\varphi$ yield that
\[
v\leq C_1(c_0+\eps)\leq \varphi\quad \text{ for }(x,t)\in\partial B_{1/2}\times (0,1),
\]
Hence, by the comparison principle, we have
$v\leq \varphi$ in $B_{1/2} \times (0,1)$.
In particular, 
\[\bar{p}(x+\bar{X}(1),1)={v}(x,1)\leq \varphi(x,1)=0\]
for $|x|<\frac{1}{6}$.
This completes the proof of the lemma.
\end{proof}

\begin{corollary}\lb{C.2.2}
Under the assumptions of Lemma \ref{L.4.1}, 
there exists a universal constant $c_0\in (0,1)$ such that the following holds for all $t_0\geq \eta_0$ and
$\tau\leq \min\{c_0,c_0(m-1)\eta_0,\eta_0\}$. If $p(\cdot,t_0)=0$ in $B(x_0,R)$ and $(X(x_0,t_0;\tau),t_0+\tau)\in\Gamma$, 
then
\[
\mint_{B(X(x_0,t_0;\tau),R)}p(x,t_0+\tau)\, dx\geq \frac{c_0R^2}{\tau}.
\]
\end{corollary}

The next lemma states that if the spatial $L^1$-average of the pressure is large locally near the free boundary, then the positive set of $p$ should expand with respect to the streamlines. 
We highlight once again that, unlike \cite{CFregularity,kimzhang21}, the constants in the proof are independent of $m$.

\begin{lemma}\label{L.4.2}
Under the assumptions of Lemma \ref{L.4.1}, there exists a universal $c_0\ll 1$ such that the following holds for any $t_0\geq \eta_0$ and $\lambda>0$. If $C_1\geq 1$ and $c_2,\tau\in (0,1)$ satisfy 
\[
C_1 \min\{\lambda,\lambda^2\}\geq 1/c_0,\quad c_2\lambda\leq c_0,\quad\text{and}\quad  \tau\max\{\lambda,1\}\leq \min\{c_0,c_0(m-1)\eta_0,\eta_0\},
\]
and if
\begin{equation}\lb{4.11}
\mint_{B(x_0,R)}p(x,t_0)\, dx\geq C_1\frac{R^2}{\tau}\quad\text{for some }R>0,
\end{equation}
then
\[
p(X(x_0,t_0;\lambda \tau),t_0+\lambda\tau)\geq c_2\frac{R^2}{\tau}.
\]
\end{lemma}

\begin{proof}
As before, set $(x_0,t_0)=(0,0)$ by shifting the coordinates.
Define $C_{\eta_0}$ as in \eqref{eqn: define C_eta_0}. 
Let 
$\eps$ be defined by \eqref{def.eps}.
Then by assuming $c_0\ll 1$ and yet universal, we have
\beq\lb{4.21}
C_{\eta_0}\tau\lambda \leq 2,\quad  C_1\min\{\lambda,\lambda^2\}\gg 1,\quad c_2\lambda \ll 1,\quad\text{and}\quad \eps\lambda \ll 1.
\eeq
All the bounds here can be made independent of $m$ and $\eta_0$.

Consider the density variable $\varrho(x,t):=(\frac{m-1}{m}p(x, t))^{\frac{1}{m-1}}$ and its rescaled version
\[
\bar{\varrho}(x,t):=\left(\frac{\tau}{R^2}\right)^{\frac{1}{m-1}}\varrho(Rx,\tau t).
\]
Then ${\xi}(x,t):=\bar{\varrho}(x+\bar{X},t)$ solves
\[
\xi_t=\Delta \xi^m+\nabla\cdot\big(\xi\, (\bar{ b }(x+\bar{X},t)-\bar{ b }(\bar{X},t)) \big)+\xi \bar{f}(x,t,v),
\]
where $\bar{f}$, $\bar{ b }$ and $\bar{X}$ are from the proof of Lemma \ref{L.4.1}.


\medskip

Define $Y(t) := \int_{B_1} {\xi}(x,t)^m dx$.
Let us first show that $Y(t)$ stays sufficiently positive for $t\in [0,\lambda]$. 
Since $\bar{X}(0)=0$, the assumption \eqref{4.11} gives that
\begin{align*}
    Y(0)&=\mint_{{{B_1}}}{\xi}(x,0)^mdx=\left(\frac{\tau}{R^2}\right)^{\frac{m}{m-1}}\mint_{B_R}\varrho(x+\bar{X}(0),0)^m dx\\
    &=\left(\frac{m-1}{m}\right)^\frac{m}{m-1}\mint_{B_R} \left(\frac{\tau}{R^2}p(x,0)\right)^{\frac{m}{m-1}}dx\\
    &\geq c\left(\frac{\tau}{R^2} \mint_{B_R}p(x,0)\, dx\right)^{\frac{m}{m-1}} \geq c\,C _1^\frac{m}{m-1}\geq c\,C _1.
\end{align*}
Note that, since $m\geq 2$ and $C_1\geq 1$, $c\in (0,1)$ can be taken as a universal constant. 

By \eqref{4.4} and the fact ${v}(x,t)=\frac{m}{m-1}{\xi}^{m-1}(x,t)$, there exists $C>0$ such that for all $\eps\in (0,1)$,
\begin{equation}
    \label{est v tilde sq prime}
\begin{aligned}
    ({\xi}^m)_t&\geq -C_{\eta_0}\tau {v}^{\frac{m}{m-1}}-C\eps^2v^\frac{1}{m-1}|x|^2\\
    &\geq -C_{\eta_0}\tau {\xi}^m-C \eps^2 |x|^2{\xi}\geq -C_{\eta_0}\tau{\xi}^m-C\eps^2 \quad \hbox{ for } (x,t)\in B_1\times [0,T).
\end{aligned}
\end{equation}
Recall that $C_{\eta_0}\tau\lambda\leq 2$ by \eqref{4.21}. 
Then \eqref{est v tilde sq prime} implies that, for $t\in (0,\lambda]$,
\begin{equation}
    \label{Y lower}
    Y(t)\geq e^{-C_{\eta_0}\tau t}Y(0)-C \eps^2 t\geq e^{-2}c\,C_1 -C \eps^2 \lambda \geq c\,C _1=:c_3,
\end{equation}
where $c$'s are small universal constants.
The third inequality above can be achieved by taking $c_0$ to be suitably small and yet universal.


\medskip

Next, 
we claim that for some universal constant $C>0$,
\begin{equation}\label{est Y}
\int_0^t Y(s)\, ds \leq C\int_0^t {\xi}(0,s)^mds + CY(t)^{\frac{1}{m}}\quad\text{ for all }t\in (0,1/\tau).
\end{equation}
When $m\in [2,d]$, this follows from the proof of \cite[Lemma 2.3]{CFregularity} for PME and that of \cite[Lemma 4.3]{kimzhang21} for advection PME. It is clear that the constant $C$ is independent of $m\in [2,d]$. 
In what follows, we shall prove the claim for $m\geq d$.

Following \cite{CFregularity}, we define for $d\geq 3$ the Green's function $G$ as
\begin{equation}\label{green}
G(x):=|x|^{2-d}+\frac12(d-2)|x|^2-\frac{d}2.
\end{equation}
Then for some dimensional constant $C_d>0$,
\begin{equation}
    \label{prop G}
    \Delta G= -C_d\delta(x)+d(d-2)\chi_{B_1},\quad 
    G\geq 0,\quad \text{and}\quad G = |\nabla G|=0 \text{ on }\partial B_1.
\end{equation}
We shall only focus on the case $d\geq 3$ in the sequel.
When $d=2$, we instead define $G(x)=-\log |x|+\frac{1}{2}(|x|^2-1)$, and the rest of the argument is similar.



The equation for $\xi$ and direct computation yield that
\begin{equation}\lb{4.5}
\begin{aligned}
&\;\frac{d}{dt}\left(\int_{B_1}G(x)\xi(x,t)\, dx\right)\\
= &\;
\int_{B_1}\Delta G(x) \xi(x,t)^m dx
-\int_{B_1}\nabla G(x)\cdot (\bar{ b }(x+\bar{X},t)-\bar{ b }(\bar{X},t))\xi(x,t)\, dx\\
&\;+\int_{B_1}G(x)\bar{f}(x,t,v)\xi(x,t)\, dx\\
=:&\; A_1+A_2+A_3.
\end{aligned}
\end{equation}
For $A_1$, applying the first identity in \eqref{prop G}, we obtain
\[
A_1= -C_d\, \xi(0,t)^m+C\int_{B_1}\xi(x,t)^m\,dx,
\]
For $A_2$, since $\|\nabla\bar{ b } \|_\infty \geq C\tau$,
\begin{equation} \label{diffence Vx}
\begin{aligned}
    A_2&= \int_{B_1}(d-2)(|x|^{-d}-1)x\cdot(\bar{ b }(x+\bar{X},t)-\bar{ b }(\bar{X}))\xi(x,t)\,dx\\
    &\geq -C\tau \int_{B_1}(|x|^{-d}-1)|x|^2\xi(x,t)\,dx\geq-C \tau \int_{B_1}G(x) \xi(x,t)\,dx.
\end{aligned}
\end{equation}
Lastly, for $A_3$, since $\bar{f}/\tau $ is uniformly bounded, we have
\[
A_3\geq -C\tau \int_{B_1}G(x) \xi(x,t)\,dx.
\]
Combining them with \eqref{4.5} yields
\[
\frac{d}{dt}\left(\int_{B_1}G(x)\xi(x,t)\, dx\right)\geq -C_d\xi(0,t)^m+C\int_{B_1}\xi(x,t)^mdx-C\tau \int_{B_1}G(x)\xi(x,t)\,dx,
\]
which implies
\begin{align*}
    e^{C \tau  t}\int_{{{B_1}}}G(x){\xi}(x,t)\,dx\geq -C_d\int_0^t e^{C \tau  s}{\xi}(0,s)^mds+C \int_0^t e^{C \tau s} Y(s)\, ds.
\end{align*}
It follows that for all $t\in (0,1/\tau )$ and $m> 1$,
\beq\label{4.7}
\int_0^t Y(s)\,ds\leq  C \int_0^t {\xi}(0,s)^m ds+ C\int_{B_1}G(x){\xi}(x,t)\, dx,
\eeq
where $C>0$ is a universal constant.
Now by H\"{o}lder's inequality,
\[
\int_{B_1}G(x){\xi}(x,t)\, dx\leq \left(\int_{B_1}G(x)^\frac{m}{m-1}dx\right)^{\frac{m-1}{m}}\left(\int_{B_1}{\xi}(x,t)^mdx\right)^{\frac{1}{m}}.
\]
Since $m\geq d$, there exists a universal $C>0$ independent of $m$, such that
\[
\int_{B_1}G(x)^{\frac{m}{m-1}}dx\leq C\int_{B_1} |x|^{\frac{m(2-d)}{m-1}}dx+C\leq C.
\]
Hence, we conclude with \eqref{est Y} from \eqref{4.7}.

\medskip

Now suppose that $p(X(\lambda\tau),\lambda\tau)\leq  c_2\frac{R^2}{\tau}$ for some choice of $c_2>0$ satisfying \eqref{4.21}.
In terms of ${\xi}=\bar{\varrho}(\cdot+\bar{X},\cdot)$, we have
\[
{\xi}(0,\lambda)^m\leq C c_2^\frac{m}{m-1}\leq Cc_2,
\]
where $C$ is universal as $m\geq 2$.
Then also by \eqref{est v tilde sq prime}, we obtain for $t\in (0,\lambda]$ that
\[
{\xi}(0,t)^m\leq C\xi(0,\lambda)^m+C \eps^2 \lambda
\leq Cc_2+C \eps^2 \lambda.
\]
Combining this with \eqref{est Y} yields for all $t\in (0,\lambda]$ that
\[
\int_0^t Y(s)\, ds \leq C\left( c_2\lambda+\eps^2\lambda^2+ Y(t)^{\frac{1}{m}}\right).
\]
In view of \eqref{Y lower}, if we further assume $c_0$ to be sufficiently small, so that (also see \eqref{4.21})
\begin{equation}
\label{cond c2}
c_3^\frac1m=c\,C_1^\frac{1}{m}\geq c_2\lambda+\eps^2\lambda^2,
\end{equation}
then $CY(t)^{1/m}\geq c_2\lambda+\eps^2\lambda$.
Hence, for $t\in (0,\lambda]$, 
\[
C Y(t)^\frac{1}{m}\geq \int^t_0 Y(s)\,ds,
\]
where $C>0$ is universal.

Writing $Z(t):=\int_0^t Y(s)\, ds$, we obtain
\begin{align*}
    Z'(t)\geq C^{-m}Z^m(t)\quad\text{for }t\in[0,\lambda).
\end{align*}
Instead of using initial data $Z(0)=0$, we use $Z(\frac{\lambda}{2})$. Indeed, it follows from \eqref{Y lower} that
$Z(\frac{\lambda}{2})\geq \frac{\lambda}2 c_3$. Then by solving the differential inequality, we obtain
\begin{equation}
    \label{est Z}
    Z(t+\lambda/2)^{m-1}\geq \left((c_3\lambda/2)^{-m+1}-(m-1)C^{-m}t\right)^{-1}\quad \text{for $t\in (0,{\lambda}/{2})$.}
\end{equation}
Notice the right-hand side of \eqref{est Z} goes to $+\infty$ as
\[
t\to \frac{C}{m-1}\left(\frac{2C}{c_3\lambda}\right)^{m-1}=:C_{m,c_3\lambda}.
\]
Since $Z(t+\lambda/2)$ should be well-defined for $t\in (0,\lambda/2)$, to obtain a contradiction, it suffices to have $C_{m,c_3\lambda}\leq \frac\lambda2$.
Since $m\geq 2$, this can be achieved if $c_3\lambda^{\frac{m}{m-1}}\gg 1$.
This is equivalent to $C_1 \lambda^\frac{m}{m-1}\gg 1$ (cf.\;\eqref{Y lower}) and it is guaranteed by \eqref{4.21}. 

Finally, because of the contradiction, we conclude that $p(X(\lambda\tau),\lambda\tau)\geq  c_2\frac{R^2}{\tau}$. 
This completes the proof.
\end{proof}

\begin{remark}\lb{R.2}
In view of Remark \ref{R.1}(1), if we assume \eqref{R.1.1}, then \eqref{R.1.2} holds with the constant being uniformly for all time. Thus we can replace $C_{\eta_0}$ by a universal constant that does not depend on $\eta_0$, and the conclusion of Lemma \ref{L.4.2} holds for all $t_0\geq 0$ and with $\tau\max\{\lambda,1\}\leq c_0$ for some $c_0\ll 1$. Similarly, this is also true for Lemma \ref{L.4.1} and Corollary \ref{C.2.2}.

\end{remark}

\begin{remark}
\label{rmk: uniform c_0}
We have introduced several $c_0$'s, which are all universal constants. 
For simplicity, in the rest of the paper, we will define $c_0$ as the smallest one among those $c_0$'s from Lemma \ref{L.4.1}, Corollary \ref{C.2.2} and Lemma \ref{L.4.2}. 
We additionally assume $c_0<1$.
\end{remark}

As a corollary of the preceding two lemmas, we can prove a dichotomy of the free boundary points.


\begin{corollary}\label{C.4.3}
Given $(x_0,t_0)\in \Gamma$ with $t_0\geq \eta_0>0$,
denote
\[
\Upsilon(x_0,t_0):=\left\{(X(x_0,t_0;-s),t_0-s),\; s\in (0,t_0)\right\}.
\]
Then the following is true:
\begin{enumerate}
\item Either {\rm(a)} ${\Upsilon}(x_0,t_0)\subset \Gamma$ or {\rm(b)} ${\Upsilon}(x_0,t_0)\cap \Gamma=\varnothing$.
\item If {\rm(b)} holds,  then there exist positive constants $C_*,\gamma,\tau$ such that for all $s\in (0,\tau)$
\beq\lb{5.1}
\begin{aligned}
&\varrho(x,t_0-s)=0 \quad\text{if}\quad |x-X(x_0,t_0;-s)|\leq C_* s^\gamma ;\\
&\varrho(x,t_0+s)>0 \quad\text{if}\quad |x-X(x_0,t_0;s)|\leq C_*s^\gamma.
\end{aligned}
\eeq
\end{enumerate}
\end{corollary}


Based on our Lemmas \ref{L.4.1} and \ref{L.4.2}, the proof of Corollary \ref{C.4.3} is parallel to that of Theorems 3.1--3.2 in \cite{CFregularity}. 
A sketch of the proof for part (1) can be found in \cite{kimzhang21}.
However, for our purpose, it is crucial to further characterize the dependence of the constants $C_*,\gamma,\tau$ above, as we need them to be independent of $m$ and the choice of the free boundary points. 
This will be addressed in the next section.

\section{Uniform Estimates for Strict Expansion}
\label{sec: strict expansion}
In this section, we want to show that, if the the support of the solution strictly expands with respect to streamlines at the initial time and uniformly for all $m\geq 2$, then such property still holds for all times.
To be more precise, we make the assumption that
\begin{enumerate}
    \item[{\bf (S)}] \label{assumption: strict expansion at initial time} There exists $\tau_0>0$ such that for all $m\geq 2$ and for all $\tau\in(0,\tau_0]$, we can find $r_\tau>0$ satisfying
\[
\Omega_{p_m}(\tau)\text{ contains the $r_\tau$-neighborhood of }\{X(x,0;\tau)\,|\,x\in\Omega_{p_m}(0)\}.
\]  
{Let us assume that $r_\tau$ is continuous in $\tau$.}
\end{enumerate}
In what follows, we first discuss some general conditions that guarantee {\bf (S)}, and then we show that such strict expansion property propagates to later times.

\subsection{Strict expansion at the initial time}
\label{sec: strict expansion at t=0}
It has been known for a long time that, under the assumption \eqref{introgr}, the positive set of solutions to the PME strictly expands at the initial time; see for example \cite{aronson1983initially,CFregularity}. 
In a similar spirit, we shall prove in the following lemma that this holds for the PME with source and drift terms as well, where the strict expansion should be understood as that with respect to the streamlines. 
The proof is postponed to Appendix \ref{sec: proof of strict expansion lemma 1}.

\begin{lemma}\lb{L.5.12}
Suppose that $\Omega_{p_m}(0)$ is a bounded domain with Lipschitz boundary and \eqref{introgr} holds. 
Then there exists $\delta_m>0$ such that for any $\tau\in (0,\delta_m]$ there exists $r_{\tau,m}>0$ satisfying
\[
\Omega_{p_m}(\tau)\text{ contains the $r_{\tau,m}$-neighborhood of }\{X(x,0;\tau)\,|\,x\in\Omega_{p_m}(0)\}.
\]
\end{lemma}

However, one cannot hope for such strict expansion to be uniform in $m$. 
Indeed, the limiting Hele-Shaw flow is known to exhibit the waiting time phenomenon \cite{book,choi2006waiting,kim2022regularity}:
%
%
if $\Omega(0)$ is locally like a cone of small angle at a boundary point, then for the limiting problem, the streamline starting at the vertex of the cone lies on the free boundary for a short time. 
In other words, in the limiting problem, $\Omega(0)$ may not strictly expand relative to the streamlines at some free boundary points.

In view of this, we need some extra assumptions to guarantee {\bf (S)}. 
Let us discuss two results in this direction.
The first one is to assume \eqref{R.1.1}. 
We remind that with \eqref{R.1.1}, Lemma \ref{L.4.2} is valid for all $t\geq 0$ instead of for $t\geq \eta_0>0$; see Remark \ref{R.2}.

\begin{lemma}\lb{L.5.11}
Suppose that $\Omega_{p_m}(0)$ is a bounded domain with Lipschitz boundary, and \eqref{cond}, \eqref{introgr} and \eqref{R.1.1} hold. Then {\bf(S)} holds 
and $r_\tau$ can be selected as $\frac{1}{2}\tau^{2/\varsigma_0}$ with $\varsigma_0$ from \eqref{introgr}.
\end{lemma}

\begin{proof}
For brevity, let us drop $p_m$ from the subscripts of $\Omega_{p_m}$ and $\Gamma_{p_m}$.
Let $x_0\in \Omega(0)^c$ be close to $\Gamma(0)$ with $R:=2d(x_0,\Omega(0))$. We are going to apply Lemma \ref{L.4.2} with the $x_0$ and $R$, and $t_0=0$, $\lambda=1$ and $\tau\in [ R^{\varsigma_0/2},c)$ for some universal $c>0$. Indeed, due to \eqref{introgr} and that $\Omega(0)$ has a Lipschitz boundary, the condition \eqref{4.11} holds as long as $R$ is sufficiently small. Then Lemma \ref{L.4.2} and Remark \ref{R.2} yield that 
\[
p_m(X(x_0,0;\tau),\tau)>0.
\]
Thus we obtain for all $\tau>0$ being sufficiently small but uniform in $m$, and $\tilde{r}_\tau:= R=\tau^{2/\varsigma_0}$, then
\beq\lb{4.8}
\left\{X(x,0;\tau)\,|\, x=x_1+x_2,\, x_1\in B_{\tilde{r}_\tau}\text{ and }x_2\in \Omega(0)\right\}\subseteq \Omega(\tau).
\eeq

Next we show that
\beq
\begin{split}
\{y = y_1+X(x_2,0;\tau) &|\, y_1\in B_{\tilde{r}_\tau/2}\text{ and }x_2\in \Omega(0)\}\\
&\subseteq 
\{X(x,0;\tau)\,|\, x=x_1+x_2,\, x_1\in B_{\tilde{r}_\tau}\text{ and }x_2\in \Omega(0) \}.
\end{split}
\label{eqn: nbhd of the flow is contained in the flow of a bigger nbhd}
\eeq
Once this is done, we can combine it with \eqref{4.8} to obtain {\bf(S)} with $r_\tau = \tilde{r}_\tau/2 = \frac12 \tau^{2/\varsigma_0}$, where $\tau$ needs to be sufficiently small but uniform in $m$.

For any $x_1,x_2\in \bbR^d$ such that $d(x_j,\Omega(0))\leq \tilde{r}_\tau$ ($j =1,2$),
\[
\frac{d}{dt}|X(x_1,0;t)-X(x_2,0;t)|\leq \|\nabla  b\|_\infty |X(x_1,0;t)-X(x_2,0;t)|,
\]
so we have $|X(x_1,0;\tau)-X(x_2,0;\tau)|\leq C|x_1-x_2|$ when $\tau$ is smaller than a universal constant. 
Hence,
\[
\frac{d}{dt}|X(x_1,0;t)-X(x_2,0;t)-(x_1-x_2))|\leq \|\nabla b \|_\infty |X(x_1,0;t)-X(x_2,0;t)|\leq C|x_1-x_2|.
\]
Combining this with $|X(x_1,0;0)-X(x_2,0;0)-(x_1-x_2)|=0$ yields that, 
when $\tau$ is smaller than a universal constant, 
\beq
|X(x_1,0;\tau)-X(x_2,0;\tau)-(x_1-x_2)|\leq \frac13|x_1-x_2|.
\label{eqn: close to an identity map}
\eeq

Now take arbitrary $x_2\in \Omega(0)$ and $y_1\in B_{\tilde{r}_\tau/2}$, we want to show that there exists $x_1\in B_{\tilde{r}}$ such that 
$X(x_1+x_2,0;\tau) = y_1+X(x_2,0;\tau)$, which will directly imply \eqref{eqn: nbhd of the flow is contained in the flow of a bigger nbhd}.
Let 
\[
x_{1,1} := y_1,\quad 
y_{1,1}: = y_1-\big(X(x_{1,1}+x_2,0;\tau)-X(x_2,0;\tau)\big).
\]
By \eqref{eqn: close to an identity map}, $|y_{1,1}|\leq \frac13 |y_1|$.
Then for $k\geq 2$, we inductively define
\[
x_{1,k} = x_{1,k-1}+y_{1,k-1},\quad 
y_{1,k} = y_{1,k-1}-\big(X(x_{1,k}+x_2,0;\tau)-X(x_{1,k-1}+x_2,0;\tau)\big).
\]
Again by \eqref{eqn: close to an identity map}, 
$|y_{1,k}|\leq \frac13|y_{1,k-1}|$.
We thus obtain $\{x_{1,k}\}_{k=1}^\infty$ as a Cauchy sequence, satisfying that $|x_{1,k}|\leq \frac32|y_1|$ for all $k\in \mathbb{Z}_+$.
Assume that it converges to $x_1\in B_{\tilde{r}_\tau}$.
Then by the continuity of the map $X(\cdot,0;\tau)$ and the definition of $y_{1,k}$, we find that 
$0 = y_1-(X(x_1+x_2,0;\tau)-X(x_2,0;\tau))$, which proves the desired claim.
\end{proof}

We provide another strict expansion result that is uniform in $m$. Instead of \eqref{R.1.1}, we make another two assumptions: the uniform interior ball condition on $\{\Omega_{p_m}(0)\}_m$ and the smallness assumption on $\|\nabla b\|_\infty$.


\begin{lemma}\lb{L.4.4}
Assume \eqref{cond} and \eqref{introgr}. 
Suppose that $\{\Omega_{p_m}(0)\}_m$ satisfies the uniform interior ball condition with some constant $r>0$, i.e., for any $m>1$ and any $x\in\Gamma_{p_m}(0)$, there exists an open ball $B$ of radius $r$ such that $B\subset \Omega_{p_m}(0)$ and $x\in\overline{B}$.
Furthermore, assume
\[
\sigma> 2d\sup_{x\in\bbR^d}|\nabla b(x,t)|\quad\text{for all $t>0$ sufficiently small},
\]
where $\sigma$ is from \eqref{cond}. Then {\rm\bf (S)} holds for all $m>1$. 
\end{lemma}
We postpone its proof to Appendix \ref{sec: proof of L.4.4}.

\medskip

At the end of the subsection, we show that the free boundary cannot expand too fast for any time. 
The proof is similar to the last part of the proof of Lemma \ref{L.4.1}, and the Aronson-B\'enilan estimate will not be applied.

\begin{proposition}\lb{P.4.6}
There exists $C>0$ independent of $m>1$ such that for any $\delta\in(0,1)$ and $t\in [0,T-\delta)$,
\[
\Omega_{p_m}(t+\delta)\text{ is contained in the $C\delta^\frac{1}{2}$-neighborhood of $\{X(x,t;\delta)\,|\,x\in \Omega_{p_m}(t)\}$}.
\]
\end{proposition}
\begin{proof}
To prove this proposition, it suffices to show that there exists $c>0$ such that for any $x_0\in \bbR^d$ and $t_0\geq 0$ and $R\in (0,1)$, if $p_m(\cdot,t_0)=0$ in $B(x_0,R)$, then $p_m(\cdot,t_0+cR^2)=0$ in $B\big(X(x_0,t_0;cR^2),\frac{R}3\big)$. The general conclusion follows from iteratively applying this claim.

\medskip

Let us recall $\|p_m\|_{L^\infty(\bbR^n\times [0,T])}\leq C_1$ for some $C_1>0$ uniformly for all $m>1$.
Take $(x_0,t_0)$ such that $\dist(x_0,\Gamma(t_0))=R\in(0,1)$. Without loss of generality, suppose $x_0=0$ and $t_0=0$. 
With $\bar{X}(t):=\frac{1}{R}X(0,0;\tau t)$, we define
\[
v(x,t):=\frac{\tau}{R^2}{p}_m(Rx+R\bar{X}(t),\tau t),
\]
which satisfies $v(x,0)=0$ for $x\in B_1$ and
\beq\lb{5.5}
{v}_t - (m-1){v}(\Delta {v}+\bar{F})- |\nabla {v}|^2-\nabla{v}\cdot (\bar{ b }(x+\bar{X},t)-\bar{ b }(\bar{X},t))=0
\eeq
with $\bar{ b }(x,t):=\frac{\tau}{R} { b }(Rx,\tau t)$, $ \bar{f}(x,t,v):=\tau f(Rx,\tau t, \frac{R^2}{\tau}v)$, and 
\[
\bar{F}(x,t,v):=\nabla\cdot \bar{ b }(x+\bar{X},t)+\bar{f}(x+\bar{X},t,v).
\]
Then there is $C_2\geq 1$ such that for all $(x,t)\in B_1\times(0,1)$,
\beq\lb{5.2}
|\bar{F}(x,t)|\leq C_2\tau\quad\text{ and }\quad |\bar{ b }(x+\bar{X},t)-\bar{ b }(\bar{X},t)|\leq \|\nabla{ b }\|_\infty|x|\leq C_2\tau.
\eeq
By taking $\tau$ to be small, we assume $\eps:=C_2\tau<1$.

Now let us construct a supersolution $\varphi$.
For $t\in (0,1)$, let $\sigma(t):=\frac{C_1\tau}{R^2}(1+\frac{t}2)$, $r(t):=\frac{2}{3}-\frac{t}{3}$ and
\[
\Sigma:=\left\{(x,t)\,|\, x\in B_1\setminus B_{r(t)},\,t\in (0,1)\right\}.
\]
Then set $\varphi(x,t)$ to be the solution to
\beq\lb{5.4}
\left\{
\begin{aligned}
-\Delta\varphi&= \eps\quad &&\text{ in }\Sigma,\\
\varphi&=\sigma(t) \quad &&\text{ on }\partial B_{1},\\
\varphi&=0\quad &&\text{ on }\partial B_{r(t)}.
\end{aligned}
\right.
\eeq
We define $\varphi(x,t)= 0$ if $x\in B_{r(t)}$.
Then we can argue as in the proof of Lemma \ref{L.4.1} to obtain that, given $\tau=cR^2$ with $c\ll 1$ being a universal constant, $\varphi$ is a supersolution to \eqref{5.5} in $B_1\times (0,1)$.
Since ${v}(x,0) = 0$ in $B_{1}$ and
$v\leq \frac{C_1\tau}{R^2}\leq \varphi$ for $(x,t)\in\partial B_{1}\times (0,1)$, 
the comparison principle yields
${v}\leq \varphi$ in $B_{1} \times (0,1)$.
In particular, 
\[
\frac{\tau}{R^2}p_m(Rx+R\bar{X}(1),\tau)\leq \varphi(x,1)=0\]
for $|x|<\frac{1}{3}$. This proves the claim and the conclusion follows. 
%
\end{proof}

\subsection{Strict expansion after the initial time}
\label{sec: strict expansion at later time}
In this subsection, we show strict and uniform-in-$m$ expansion of solutions after time $0$.
The point is to propagate the strict expansion property of the support of solutions from the initial time to {\it all finite times} uniformly for all values of $m$ and regardless of possible topological changes on the free boundary.
This will be achieved in Lemma \ref{P.4.2}, where we will assume \eqref{5.22} below.
Note that several of our estimates rely on the AB estimate \eqref{ineq fund}, which has a singularity at time $0$.
Though it is not obvious, the assumption \eqref{5.22} is made to overcome this difficulty.
We will prove \eqref{5.22} in Lemma \ref{P.4.5} by using {\bf (S)}.
The main result of this section will be presented in Proposition \ref{L.5.3}.

In the rest of the section, we take $m\geq 2$.

\smallskip

The following lemma states that, given the strict expansion of the free boundaries at a time scale $\tau$, the free boundaries should expand strictly at smaller time scales.
Thanks to the results shown in Section \ref{S3}, we can prove this as in \cite[Theorem 3.2]{CFregularity}, while we further need to follow the streamlines.

\begin{lemma}\lb{L.4.3}
There exist $\gamma > 4$ and $c>0$ such that the following holds for all $m\geq2$. Let $(x_0,t_0)\in\Gamma$, and let $\tau\ll1$ satisfy
\beq\lb{cS}
0<4\tau/3< c_0\min\left\{1,{t_0}/2\right\},
\eeq
where $c_0$ is from Remark \ref{rmk: uniform c_0}.
If for some $R>0$ we have
\beq\lb{5.22}
p_m(\cdot,t_0-\tau)=0\quad\text{in }  B(X(x_0,t_0;-\tau),R),
\eeq
then for any $ s\in [0,\tau]$,
\[
p_m(\cdot,t_0-s)=0\quad\text{in }B\left(X(x_0,t_0;-s),c(s/{\tau})^\gamma R\right).
\]
\end{lemma}


\begin{proof}
We will write $p=p_m$.
Let $t_1:=t_0-\tau$, $t_2:=t_0-\lambda\tau$ for $\lambda:=(1-\gamma^{-1})\in (\frac34,1)$ with $\gamma$ to be chosen. We start from proving that $x_0$ cannot be too close to $ \{X(x,t_2;\lambda \tau)\,|\,x\in \Gamma(t_2)\}$. Suppose for contradiction that for some $(x_2,t_2)\in\Gamma$ and $y_1:=X(x_2,t_2;\lambda\tau)$,
\beq\lb{4.10}
d(x_0,y_1)=d(x_0,\{X(x,t_2;\lambda \tau)\,|\,x\in \Gamma(t_2)\})<\alpha R,
\eeq
where $\alpha\in (0,\frac{1}{2})$ is to be chosen. 

It follows from 
the ODE of streamlines that for $\tau\ll1$, 
\beq\lb{5.21}
|X(x_0,t_0;-\lambda\tau)-x_2|=|X(x_0,t_0;-\lambda\tau)-X(y_1,t_0;-\lambda\tau)|\leq  e^{ \lambda\tau\|\nabla{ b }\|_\infty}|x_0-y_1|\leq 2\alpha R,
\eeq
\beq\lb{5.21'}
|X(x_0,t_0;-\tau)-X(x_2,t_2;\lambda\tau-\tau)|=|X(x_0,t_0;-\tau)-X(y_1,t_0;-\tau)|\leq  e^{ \tau\|\nabla{ b }\|_\infty}|x_0-y_1|\leq 2\alpha R.
\eeq
By the assumption that $p(\cdot,t_1)=0$ in $B(X(x_0,t_0;-\tau),R)$, 
\eqref{5.21'} implies that 
\[
p(\cdot,t_1)=0\quad\text{in }B(X(x_2,t_2;-(1-\lambda)\tau),(1-2\alpha)R).
\]
With this and the fact $x_2\in\Gamma(t_2)$, applying Corollary \ref{C.2.2} yields that
\[
\mint_{B(x_2,(1-2\alpha)R)}p(x,t_2)\, dx\geq \frac{c_0(1-2\alpha)^2R^2}{(1-\lambda)\tau}.
\]
Thus, also using \eqref{5.21}, we find
\beq\lb{5.20}
\mint_{B(X(x_0,t_0;-\lambda\tau), R)}p(x,t_2)\, dx\geq \frac{c_0(1-2\alpha)^{n+2}R^2}{(1-\lambda)\tau}.
\eeq

Now take $C_1$ and $c_2$ from Lemma \ref{L.4.2} with $\lambda\in [\frac34,1]$.
Then take $\alpha=(1-\gamma^{-1})^\gamma$ and $\gamma$ to be sufficiently large (and thus $\lambda$ is close to $1$) such that
\[
\frac{c_0(1-2\alpha)^{n+2}}{1-\lambda}\geq {C_1}. 
\]
As a consequence, \eqref{5.20} yields
\[
\mint_{B(X(x_0,t_0;-\lambda\tau),R)}p(x,t_2)\, dx\geq \frac{C_1R^2}{{\tau}}.
\]
Also note that \eqref{cS} implies $\tau<\min\{c_0,c_0(m-1)(t_0-\tau),t_0-\tau\}$. Thus we can apply Lemma \ref{L.4.2} to get $p(x_0,t_0)>0$ which contradicts with the assumption that $(x_0,t_0)\in\Gamma$. Thus, we conclude
\[
d(x_0,\{X(x,t_2;\lambda \tau)\,|\,x\in \Gamma(t_2)\})\geq \alpha R.
\]

By iteration, we get for all $n\geq 1$ and $t_{n+1}:=t_0-\lambda^{n}\tau$,
\[
d(x_0,\{X(x,t_{n+1};\lambda^{n} \tau)\,|\,x\in \Gamma(t_{n+1})\})\geq \alpha^{n} R.
\]
By the ODE of streamlines, we know for any $x\in\Gamma(t_{n+1})$ and $\tau\ll 1$,
\[
|X(x_0,t_0;-\lambda^{n}\tau)-x|\geq  e^{ -\lambda^{n}\tau\|\nabla{ b }\|_\infty}|x_0-X(x,t_{n+1};\lambda^{n}\tau)|\geq \alpha^{n} R/2.
\]
Thus, we get for all $\tau\ll1$,
\[
p(\cdot,t_0-\lambda^{n}\tau)=0\quad\text{ in }B(X(x_0,t_0;-\lambda^{n}\tau),\alpha^{n}R/2). 
\]

Finally, note that  by Lemma \ref{streamline}, \eqref{5.22} holds with $\tau$ replaced by $\beta\tau$ for any $\beta\in [1,\frac{4}{3}]$. Because $\lambda\in(\frac34,1)$, by replacing $\tau$ by $\beta\tau$ with $\beta\in [1,\frac{4}{3}]$ in the above argument, we can conclude the assertion of the lemma with $\gamma =\log_\lambda \alpha$.
\end{proof}

The next goal is to propagate the strict expansion property of the free boundaries under the assumption \eqref{5.22} to all finite times. Our approach will quantify the constants in Corollary \ref{C.4.3}, and meanwhile, ensuring that our estimates remain independent of $m$. For simplicity, we shall drop $p_m$ from the notations $\Omega_{p_m}$ and $\Gamma_{p_m}$.

\begin{lemma}\lb{P.4.2}
Let $R>0$ and let $\tau,t_0$ satisfy \eqref{cS}.
There exists a universal constant $\alpha\in (0,1)$ (independent of $m,\tau,t_0,R$) such that
\begin{enumerate}
\item If \eqref{5.22} holds {for all $x_0\in \Gamma(t_0)$}, then for all $n\in \bbZ_+$ and $x\in\Gamma(t_0+n\tau)$ we have
\[
d\big(X(x,t_0+n\tau;-\tau), \Gamma(t_0+(n-1)\tau)\big)\geq \alpha^nR.
\]
\item 
Instead, if 
\[
d\big(X(x_1,t_0+\tau;-\tau), x_0\big)< \alpha R
\]
for some $x_0\in\Gamma(t_0)$ and $x_1\in\Gamma(t_0+\tau)$, then 
\[
d\big(X(x_0,t_0;-\tau), \Gamma(t_0-\tau)\big)<R. 
\]
\end{enumerate}
\end{lemma}

\begin{proof}
Let us assume \eqref{5.22}. It follows from Lemma \ref{L.4.3} that
\[
d\big(X(x_0,t_0;-s),\Gamma(t_0-s)\big)\geq  c(s/\tau)^\gamma R
\]
holds for all $ s\in [0,\tau]$. Denote $\alpha_s:=c(s/\tau)^\gamma$ and $R_{1,s}:=\alpha_s R$. Since $(x_0,t_0)\in\Gamma$, it follows from Corollary \ref{C.2.2} that
\beq\lb{5.23'}
\mint_{B(x_0,R_{1,s})}p(x,t_0)\, dx\geq \frac{c_0R_{1,s}^2}{s},
\eeq
which holds for all $x_0\in\Gamma(t_0)$ and $s\in[0,\tau]$ uniformly.  

Now let $c_0$, $C_1$ and $c_2$ satisfy the conditions of Lemma \ref{L.4.2} with $\lambda=1$. Choose $s:= c_0\tau/(2^{d+2} C_1)<\tau$ and set $\alpha:=\alpha_s$ (then $\alpha$ is independent of $\tau,R$) and $R_1:=R_{1,s}=\alpha R$ with this choice of $s$. Then
\eqref{5.23'} yields for all $z\in B(x_0,R_1)$ that
\[
\mint_{B(z,2R_1)}p(x,t_0)\, dx\geq 2^{-d}\mint_{B(x_0,R_1)}p(x,t_0)\, dx\geq  \frac{C_1 (2R_1)^2}{\tau}.
\]
By Lemma \ref{L.4.2}, we get for all $z$ such that $d(z,\Gamma(t_0))\leq R_1$,
\[
p(X(z,t_0;\tau),t_0+\tau)\geq \frac{c_2(2R_1)^2}{\tau}.
\]

Since $x_0\in \Gamma(t_0)$ is arbitrary, we also get for any $x_1\in\Gamma(t_0+\tau)$,
\[
d\big(X(x_1,t_0+\tau;-\tau),\Gamma(t_0)\big)\geq R_1=\alpha R.
\]
With this, we can apply Lemma \ref{L.4.3} again (with $t_0$ and $R$ replaced by $t_0+\tau$ and $R_1$) to get
\[
d\big(X(x_1,t_0+\tau;-s),\Gamma(t_0+\tau-s)\big)\geq  \alpha_s R_1
\]
holds for all $ s\in [0,\tau]$. Identical arguments as the above yield that, for any $x_2\in\Gamma(t_0+2\tau)$, 
\[
d\big(X(x_2,t_0+2\tau;-\tau),\Gamma(t_0+\tau)\big)\geq R_2=\alpha^2 R.
\]
By iterating this argument, for any $x\in\Gamma(t_0+n\tau)$, we obtain that
\[
d\big(X(x,t_0+n\tau;-\tau),\Gamma(t_0+(n-1)\tau)\big)\geq \alpha^n R.
\]

The second claim follows from the first part of the proof.
\end{proof}

In the following lemma, we prove \eqref{5.22} with the assumption {\rm\bf (S)}. 

\begin{lemma}\lb{P.4.5}
Assume {\rm\bf (S)}.
{Given any $t_0$ sufficiently small, for any $\tau$ that is sufficiently small and satisfies \eqref{cS}, there exists $R>0$ such that \eqref{5.22} holds for all $x_0\in \Gamma(t_0)$.
Here the smallness requirements of $t_0$, $\tau$, and $R$ should all depend on {\rm\bf (S)} and the universal constants.}

\end{lemma}

\begin{proof}
The condition {\rm\bf (S)} yields for each $t_0\in(0,1)$ small enough, there is $R_0>0$ such that
\[
\Omega(t_0)\text{ contains the $2R_0$-neighborhood of }\{X(x,0;t_0)\,|\,x\in\Omega(0)\}.
\]
By Proposition \ref{P.4.6}, for all $t_*<t_0$ being sufficiently small,
\beq\lb{5.77}
\Omega(t_0)\text{ contains the $R_0$-neighborhood of }\{X(x,t_*;t_0-t_*)\,|\,x\in\Omega(t_*)\}.
\eeq
Then for any $\tau>0$ sufficiently small, we can ensure that 
\begin{enumerate}[(i)]
\item \eqref{cS} holds with $t_0$ there replaced by $t_*$. 
Indeed, it suffices to take $3\tau/c_0\leq t_*$ and $\tau<c_0/2$;
\item 
Up to a slight adjustment of $t_*$ (so that \eqref{5.77} is still true), we may assume that $N:=(t_0-t_*)/\tau\geq 2$ is a positive integer, 
which depends only on {\rm\bf (S)} and the universal constants.
\end{enumerate}
Note that we kept $\tau$ arbitrary as long as it is small enough.

Assume that $R>0$ satisfies,
for some $x_{-1}\in\Gamma(t_0-\tau)$ and $x_0\in\Gamma(t_0)$,
\beq\lb{5.78}
d\big(X(x_0,t_0;-\tau), x_{-1}\big) < R.
\eeq
We shall show that $R$ cannot be too small compared with $R_0$.
Since $t_0-2\tau\geq t_*$, the second claim of Lemma \ref{P.4.2} and \eqref{5.78} imply that, for some universal $\alpha\in(0,1)$ and some $x_{-2}\in \Gamma(t_0-2\tau)$, 
\[
d\big(X(x_{-1},t_0-\tau;-\tau), x_{-2}\big)<\alpha^{-1}R.
\]
Recall that $t_0=t_*+N\tau$.
By iteration, we obtain a sequence of points $\{x_{-1},\ldots,x_{-N}\}\subseteq\bbR^d$ such that, $x_{-j}\in \Gamma(t_0-j\tau)$ for $j\in\{1,\ldots,N\}$, and
\[
d\big(X(x_{-j},t_0-j\tau;-\tau), x_{-j-1}\big)<\alpha^{-j}R.
\]
Indeed, this can be done up to $j = N$ because \eqref{cS} holds with $t_0$ replaced by $t_*$ (see the conditions of Lemma \ref{P.4.2}).

For $0\leq j\leq N$, denote
\[
z_j:=X(x_{-j},t_0-j\tau;j\tau).
\]
Using the ODE of streamlines and that $N\tau< t_0< 1$, one can get for $0\leq j\leq N-1$, 
\[
d(z_j,z_{j+1})\leq e^{\|\nabla b\|_\infty (j+1)\tau}d\big(X(x_{-j},t_0-j\tau;-\tau), x_{-j-1}\big)<e^{\|\nabla b\|_\infty}\alpha^{-j}R.
\]
Therefore, there is $C_{\alpha,N}>0$ depending only on $\tau,t_0$, {\rm\bf (S)} and universal constants such that
\[
d\big(X(x_{-N},t_*;N\tau),x_0\big)\leq \sum_{j=1}^{N-1} d(z_j,z_{j+1})\leq C_{\alpha,N}R.
\]
Since $x_{-N}\in \Gamma(t_*)$ and $x_0\in\Gamma(t_0)$, we deduce from \eqref{5.77} that $C_{\alpha,N}R\geq R_0$, which implies (cf.\;\eqref{5.78})
\[
\inf_{x_0\in \Gamma(t_0)\atop x_{-1}\in \Gamma(t_0-\tau)} d\big(X(x_0,t_0;-\tau), x_{-1}\big) \geq C_{\alpha,N}^{-1} R_0.
\]
This finishes the proof.
\end{proof}

Combining these lemmas, we obtain the main result of the section: the strict expansion property of the free boundaries propagates from the initial time, which is given by {\rm\bf (S)}, to all finite times. This is a quantified version of \eqref{5.1} and the constants are independent of $m$. Moreover, we show that the free boundary is weakly non-degenerate in the sense that, on average, $p_m$ near the free boundary should be less degenerate than having quadratic growth.




\begin{proposition}\lb{L.5.3}
Assume {\bf (S)}, let $m\geq 2$ and $T\geq 1$, and let $\gamma>4$ from Lemma \ref{L.4.3}. 
For any $\eta_0\ll 1$, there exist positive constants $\tau\ll1$ and $C_*$ depending on {\bf (S)}, $T,\eta_0$ and the universal constants, such that  
\[
p_m(x,t_0-s)=0 \quad\text{ if }\quad |x-X(x_0,t_0;-s)|\leq C_* s^\gamma \text{ and }s\in [0,\tau]
\]
holds uniformly for all $(x_0,t_0)\in\Gamma$ with $t_0\in [\eta_0,T)$. This is equivalent to that
\[
\{X(x,t_0;-s)\,|\,x\in \Omega(t_0)\}\text{ contains the $C_*s^\gamma$ neighbourhood of }\Omega(t_0-s).
\]

Moreover, there exist $c_\tau,r_\tau>0$ depending only on $T,\eta_0,\tau$ and the universal constants such that for any $r\in (0,r_\tau)$, and $(x_0,t_0)\in \Gamma$ with $t_0\in [\eta_0,T)$,
\[
\mint_{B(x_0,r)}p_m(x,t_0)\,dx\geq c_\tau r^{2-\frac{1}\gamma}.
\]
\end{proposition}

\begin{proof}
First, we upgrade the conclusion of Lemma \ref{P.4.2}. It follows from Lemma \ref{P.4.5} that, for any $\eta>0$ sufficiently small, there exists $\tau>0$ such that, {for all $\beta\in [1,2]$, both \eqref{cS} and \eqref{5.22} hold with $\beta \eta$ in the place of $t_0$, and with $R>0$ being uniform for all $x_0\in \Gamma(\beta\eta)$.}
Then Lemma \ref{P.4.2} gives that, for all $n\geq 1$ and $x\in\Gamma(\eta+n\tau)$,
\beq\lb{4.12}
d\big(X(x,\eta+n\tau;-\tau), \Gamma(\eta+(n-1)\tau)\big)\geq \alpha^n R,
\eeq
where $\alpha\in (0,1)$ is universal and $R$ depends on $\tau, \eta$ and {\bf (S)}.
Thanks to the way we chose $\eta$, \eqref{4.12} holds with $\eta$ replaced by $\beta \eta$ for any $\beta \in [1,2]$.
Hence, we can further obtain for all $t\in [3\eta,T)$ and $x\in\Gamma(t)$ that
\[
d\big(X(x,t;-\tau),\Gamma(t-\tau)\big)\geq \alpha^{t/\tau} R\geq \alpha^{T/\tau}R=:R_{\tau,T}.
\]
Finally, by Lemma \ref{L.4.3}, there exist universal constants $c>0$ and $\gamma\geq 4$ such that for any $s\in [0,\tau]$,
\beq\lb{4.13}
p(\cdot,t-s)=0\quad\text{in }B \big(X(x,t;-s),c(s/\tau)^{\gamma} R_{\tau,T}\big)
\eeq
for all $t\in [3\eta,T)$ and $x\in\Gamma(t)$.
The result improves the conclusion of Lemma \ref{P.4.2}. Let us emphasize that $R_{\tau,T}$ is uniform for all $m\geq 2$ and $x\in \Gamma(t)$ with $t\in [3\eta,T)$. 
Then the desired claim holds with $\eta_0: = 3\eta$.

\medskip 

Next, fix $(x_0,t_0)\in\Gamma$ with $t_0\geq \eta_0$. We also denote $y_0:=X(x_0,t_0;-s)$ and $r_s:=c(s/\tau)^{\gamma} R_{\tau,T}$ for $s\in [0,\tau]$. 
It follows from \eqref{4.13} that $p_m(\cdot,t_0-s)=0$ in $B(y_0,r_s)$ for $s\in [0,\tau]$.
Since $(x_0,t_0)\in\Gamma$, Corollary \ref{C.2.2} implies that for some $c_0>0$,
\[
\mint_{B(X(y_0,t_0-s;s),r_s)}p_m(x,t_0)\, dx \geq \frac{c_0r_s^2}{s}.
\]
Since $X(y_0,t_0-s;s)=X(X(x_0,t_0;-s),\tau_0-s;s)=x_0$, we obtain that
\[
\mint_{B(x_0,r_s)}p_m(x,t_0)\, dx \geq \frac{c_0r_s^2}{s}=c_\tau r_s^{2-\frac1\gamma},
\]
where $c_\tau:=c_0\tau^{-1}(cR_{\tau,T})^{\frac1\gamma}$.
Since $r_s$ can be arbitrary in $[0,cR_{\tau,T}]$,
$r_\tau$ in the statement can be selected as $cR_{\tau,T}$.
\end{proof}

\section{Convergence of the Free Boundaries}
\label{sec: convergence of fb}
In this section, let us prove convergence of the free boundaries. 
Fix $T>0$ and let $p_m\geq0$ solve \eqref{1.1} in $Q_T$ with continuous initial data $p_m^0$. For $M\geq 1$, define
\beq
\beta_M:= \sup_{M\leq m,l\leq \infty}\|p_m -p_l \|_{L^1(Q_T)}.
\label{eqn: def of beta_M}
\eeq
We will take
\beq\lb{L1}
\lim_{M\to\infty}\beta_M=0
\eeq
as an assumption; this has been justified under suitable conditions, e.g.\;in Theorem \ref{thm: incompressible limit}. Moreover, we assume that the Hausdorff distance between the initial supports of pressures converges to $0$, i.e.,
\beq\lb{L2}
\gamma_M:=\sup_{M\leq m,l < \infty}d_H\big(\{p_m^0(\cdot)>0\},\,\{p_l^0(\cdot)>0\}\big)\to 0\quad\text{as }M\to\infty.
\eeq

We start with the following lemma, which says that it is not likely that one solution $p_l$ has a void region while $p_m$ does not when $l$ and $m$ are large.

\begin{lemma}\lb{L.6.3}
Assume \eqref{1.7}--\eqref{1.8}, \eqref{L1}--\eqref{L2}, and that the conclusion of Proposition \ref{L.5.3} holds. Let $t_0\in (0,T)$ and then $r\ll \min\{1,t_0\}$.
There exists some universal $A\gg 1$ and $M\gg 1$ that depends on $r$ and the assumptions such that, for any $m,l\in [M,\infty)$ and any $x_0\in \Omega_{p_m}(t_0)$, it holds that
\[
B(x_0,Ar)\cap \Omega_{p_l}({t_0})\neq \varnothing.
\]
\end{lemma}
\begin{proof}
Assume for contradiction that $B(x_0,Ar)\subseteq \Omega_{p_l}({t_0})^c$. Then Lemma \ref{streamline} and the space-time continuity of $p_l$ imply that
\[
X(x,{t_0};-{t_0})\in \Omega_{p_l}(0)^c\quad\text{for all }x\in B(x_0,Ar).
\]
For any $C>0$, if $A=A(C,T)$ is sufficiently large, we get from the ODE of streamlines that 
\[
B\big(X(x_0,t_0;-t_0),(C+1)r\big)\subseteq \Omega_{p_l}(0)^c.
\]
Take $M$ to be large such that $\gamma_M\leq r$ by \eqref{L2}, and we have
\[
B\big(X(x_0,t_0;-t_0),Cr\big)\subseteq \Omega_{p_m}(0)^c.
\]
Then Proposition \ref{P.4.6} yields that
$X(x_0,t_0;-t_0+r^2)\in \Omega_{p_m}(0)^c$ provided that $C$ is large depending only on the universal constants. Since $x_0\in\Omega_{p_m}(t_0)$, the streamline passing through $(x_0,t_0)$ must reach the free boundary at some  time, i.e., there exists $\tau_0\geq r^2$ such that 
\[
X(x_0,t_0;-t_0+\tau_0)\in \Gamma_{p_m}(\tau_0).
\]
This and Proposition \ref{L.5.3} (with $\eta_0$ replaced by $r^2$) imply that, for some $c_r>0$ and for all $R$ sufficiently small (all independent of $m$),
\[
\mint_{B(X(x_0,t_0;-t_0+\tau_0),R)}p_m(x,\tau_0)\,dx\geq c_r R^{2-\frac{1}\gamma}.
\]
Using the assumption on $b$ and \eqref{mono stream}, we know that at later times, the average of $p_m$ over a small ball centered at points on the same streamline is bounded from below, i.e., there exists $c_r'>0$ such that for all $R'$ sufficiently small, and all $t\in [\tau_0,T)$,
\beq\lb{6.666}
\mint_{B(X(x_0,t_0;t-t_0),R')}p_m(x,t)\,dx\geq c_r' (R')^{2-\frac1\gamma}.
\eeq
Here $c_r',R'>0$ are independent of $m$.

However, by the assumption that $B(x_0,Ar)\subseteq \Omega_{p_l}({t_0})^c$ and Proposition \ref{P.4.6}, for all $t\in [{t_0},{t_0}+r^2]$, if $A$ is large enough,
\[
p_l\big(z+X(x_0,{t_0};t-{t_0}),t\big)\equiv 0 \quad \text{for }z\in B\left(0,{Ar}/2\right).
\]
This contradicts with \eqref{6.666} when $\beta_M$ is small enough.
\end{proof}

As an immediate corollary, we obtain that for any $\eta_0,r>0$, if $\infty>l, m\geq M\gg 1$, then for any $t\in [\eta_0,T)$, we have
\[
\Omega_{p_m}(t)\subseteq \text{\,the $Cr$-neighbourhood of $\Omega_{p_l}(t)$},
\]
which further implies that $d_H(\Omega_{p_l}(t),\Omega_{p_m}(t))\to 0$ as $l,m\to \infty$.
This will be included in Theorem \ref{T.main_weaker_assumption} below.

%
%
%

It is then natural to ask whether there is convergence of $\Omega_{p_m}(t)$ to $\Omega_{p_\infty}(t)$ in the Hausdorff distance as $m\to\infty$.
Since the limiting solution $p_\infty$ is not defined pointwise (cf.\;Theorem \ref{thm: incompressible limit}) and may not be continuous, in order to determine $\Omega_{p_\infty}(t)$, we shall take a special version of $p_\infty$ in the rest of the paper as follows.
Let
\[
\tilde{p}_\infty(x,t): = \limsup_{\varepsilon\to 0^+}\frac{1}{\varepsilon}\int_{0}^\varepsilon \mint_{B(x,\varepsilon)} p_\infty(y,t+s)\,dy\,ds.
\]
It is known that $\tilde{p}_\infty=p_\infty$ almost everywhere in $Q_T$, so we shall simply take $\tilde{p}_\infty$ as the special version of $p_\infty$, still denoted by $p_\infty$ in the rest of the paper.
It then holds pointwise that 
\beq
p_\infty(x,t) = \limsup_{\varepsilon\to 0^+}\frac{1}{\varepsilon}\int_{0}^\varepsilon \mint_{B(x,\varepsilon)} p_\infty(y,t+s)\,dy\,ds.
\label{eqn: property of p_inf by def}
\eeq
In particular, if $p_\infty$ is almost everywhere $0$ in a space-time open set $U\subseteq\bbR^{d+1}$, then $p_\infty\equiv 0$ in $U$ pointwise.  
With the pointwise value of $p_\infty$ given by \eqref{eqn: property of p_inf by def}, we can have $\Omega_{p_\infty}(t)$ well-defined.

We also show the following useful property of $p_\infty$: for any $x_0\in \bbR^d$, $t_0\geq 0$, and any $r>0$,
\beq
\mint_{B(x_0,r)}p_\infty(x,t_0)\,dx 
\geq \limsup_{\eps\to 0^+} \frac{1}{\eps} \int_0^\eps \mint_{B(x_0,r)} p_\infty(x,t_0+s)\,dx\,ds.
\label{eqn: link integral of p_inf at a time with a spacetime integral}
\eeq
Indeed, by \eqref{eqn: property of p_inf by def}, the Fatou's lemma, and the fact that $p_\infty$ is a priori bounded (cf.\;Lemma \ref{uniformb} and \eqref{L1}), 
\begin{align*}
\mint_{B(x_0,r)}p_\infty(x,t_0)\,dx 
= &\;\mint_{B(x_0,r)}\limsup_{\eps\to 0^+} \frac{1}{\eps}\int_0^\eps \mint_{B(x,\eps)}p_\infty(y,t_0+s)\,dy\,ds\,dx\\
\geq &\; \limsup_{\eps\to 0^+} \frac{1}{\eps} \;\mint_{B(x_0,r)}\int_0^\eps \mint_{B(x,\eps)}p_\infty(y,t_0+s)\,dy\,ds\,dx \\
\geq &\; \limsup_{\eps\to 0^+} \frac{1}{\eps} \int_0^\eps \frac{1}{|B_r|}\int_{B(x_0,r-\eps)} p_\infty(x,t_0+s)\,dx\,ds\\
= &\; \limsup_{\eps\to 0^+} \frac{1}{\eps} \int_0^\eps \mint_{B(x_0,r)} p_\infty(x,t_0+s)\,dx\,ds.
\end{align*}
In the third line, we used the fact that 
$\mathds{1}_{B_r}*\mathds{1}_{B_\eps}
\geq |B_\eps|\cdot \mathds{1}_{B_{r-\eps}}$.



\medskip

It turns out that the convergence of $\Omega_{p_m}(t)$ to $\Omega_{p_\infty}(t)$ is generally false under the current assumptions, but we can prove the following partial result.

\begin{theorem}\lb{T.main_weaker_assumption}
Assume \eqref{1.7}--\eqref{1.8}, \eqref{L1}--\eqref{L2}, and let $p_m\geq0$ solve \eqref{1.1} in $Q_T$ with the continuous initial data $p_m^0$. Suppose that, uniformly for all $m\geq 1$, 
the conclusion of Proposition \ref{L.5.3} holds.
Then for any $0<\eta_0\ll 1$, there exists $C(\eta_0)>0$ such that for any $0<r\ll 1$, there is $M\gg 1$ satisfying that for all $m\in [M,\infty]$, $l\in [M,\infty)$ and all $t_0\in [\eta_0,T)$, 
\beq\lb{part2}
\Omega_{p_m}(t_0)\subseteq \text{\,the $Cr$-neighbourhood of $\Omega_{p_l}(t_0)$}.
\eeq
Here $M$ depends on $r,\eta_0$ and the conditions.
\end{theorem}

\begin{proof}
When both $m$ and $l$ are finite, the result follows from Lemma \ref{L.6.3}.

When $m=\infty$ and $l<\infty$, since the result holds for finite $m$ and $l$, we know that $p_m(x,t)=0$ for all $m\geq l\geq M$ and $(x,t)$ such that
\[
t\in [\eta_0,T)\quad\text{and}\quad 
x\notin\text{\,the $Cr$-neighbourhood of $\Omega_{p_l}(t)$}.
\]
Thus, after passing to the limit $m\to\infty$, \eqref{L1} and \eqref{eqn: property of p_inf by def} imply that $p_\infty=0$ pointwise in the interior of the same region, which concludes the proof of \eqref{part2}.
\end{proof}

\begin{remark}\label{rmk: counterexample for FB convergence to p_inf}
We remark that, in general,
\[
\Omega_{p_m}(t_0)\not\subseteq \text{\,the $Cr$-neighbourhood of $\Omega_{p_\infty}(t_0)$},
\]
even for $m$ sufficiently large.
Let us provide an example in a formal way.

Take $b\equiv 0$ and $f(x,t,p) := 2-p$.
For $m>1$, define
\[
p_m^0(x):=
\begin{cases}
(\frac12 + \frac12|x|^2)^{m-1} &\mbox{if }|x|\leq 1,\\
1&\mbox{if }|x|\in (1,\frac32],\\
4-2|x|&\mbox{if }|x|\in (\frac32,2],\\
0&\mbox{if }|x|>2.
\end{cases}
\]
We define $\varrho_m^0 = P_m^{-1}(p_m^0)$ and let $\varrho^0$ be the $L^1(\bbR^d)$-limit of $\varrho_m^0$ as in \eqref{main3} and \eqref{assumption: L^1 convergence of initial density}.
Consider \eqref{main}--\eqref{main2} with the initial data $\varrho_m^0$ and its incompressible limit \eqref{1.2} (also see Theorem \ref{thm: incompressible limit}) with the initial data $\varrho^0$.
In this setting, all the assumptions of Theorem \ref{T.main_weaker_assumption} can be verified (cf.\;Theorem \ref{thm: incompressible limit} and Lemma \ref{L.4.4}).

Since $\{p_m^0>0\} = \{\varrho_m^0>0\} = B_2$ and the problem is rotation-invariant, as time goes by, one can expect that $\Omega_{p_m}(t) = \{p_m(\cdot,t)>0\}$ remains to be a disk centered at the origin.
On the other hand,
\[
\varrho^0(x):=
\begin{cases}
\frac12 + \frac12|x|^2 &\mbox{if }|x|\leq 1,\\
1&\mbox{if }|x|\in [1,2),\\
0&\mbox{if }|x|\geq 2.
\end{cases}
\]
In view of the density constraint $\varrho_\infty\leq 1$, $\varrho^0$ has the ``saturated" region $\{|x|\in [1,2)\}$ and the ``unsaturated" regions elsewhere.
Given the complementarity condition $p_\infty(1-\varrho_\infty) = 0$ in \eqref{1.2}, we expect $\{p_\infty(\cdot,t)>0\}$ to be an annular region for $t\ll 1$, which has two separate free boundaries;
in particular, $\{p_\infty(\cdot,0)>0\}= \{|x|\in (1,2)\}$.
Therefore, for $t\ll 1$, $\Omega_{p_m}(t)$ is not contained in a small neighborhood of $\Omega_{p_\infty}(t)$ even for $m\gg 1$.

Interested readers may consult \cite{kim2018porous} for rigorous analysis of the solutions and the free boundaries in a similar setting, where the density in the ``unsaturated" region is assumed to be strictly less than $1$.

\end{remark}

In the following theorem, we further study convergence of the free boundaries.
This involves studying the distance between points on $\Gamma_{p_m}$ and $\Omega_{p_l}$, as well as the distance between points on $\Gamma_{p_m}$ and the complement of $\Omega_{p_l}$.
Here and in what follows, we shall use the notations $\Gamma_{p_\infty}(t):=\pa\Omega_{p_\infty}(t) = \pa\{p_\infty(\cdot,t)>0\}$ and $\Gamma_{p_\infty}: = \cup_{t\in (0,T)}\Gamma_{p_\infty}(t)\times \{t\}$.

\begin{theorem}\lb{C.6.2_weaker_assumption}
Under the assumptions of Theorem \ref{T.main_weaker_assumption}, for any $0<\eta_0\ll 1$, there exists $C(\eta_0)>0$ such that for any $0<r\ll 1$ and then $M\gg 1$, we have for any $m\in [M,\infty)$, $l\in [M,\infty]$, any $t_0\in [\eta_0,T)$ and $x_0\in\Gamma_{p_m}(t_0)$, it holds that
\[
d\big(x_0,\Omega_{p_l}(t_0-s)\big)\leq Cr\quad\text{for all }s\in [0,r^2],
\]
and
\[
d\big(x_0,\Omega_{p_l}(t_0-r)^c\big)\leq Cr.
\]
Here $M$ depends on $\eta_0$, $r$ and the conditions.

Consequently, 
\[
\sup_{x_0\in \Gamma_{p_m}(t_0) }d \big(x_0,\Gamma_{p_l}(t_0-r^2)\big)\leq Cr.
\]
\end{theorem}

\begin{remark}
As is discussed before, due to the presence of the drift, topological changes might occur on the support of the solutions. One should expect that holes can form in the support and then get filled up after some time. At the time when a hole disappears, the (spatial) Hausdorff distance between the free boundaries can change drastically. Thus, when comparing solutions with different indices, we cannot hope for a small spatial Hausdorff distance between their free boundaries at the same instant. 
This drives us to bound the space-time Hausdorff distance between the free boundaries. 

Indeed, our result essentially implies that, after any positive time, 
\begin{itemize}
\item the space-time Hausdorff distance between $\Gamma_{p_m}$ and $\Gamma_{p_l}$ diminishes as $m,l\to\infty$;

\smallskip

\item moreover, as $m\to \infty$, $\Gamma_{p_m}$ will get close to $\Gamma_{p_\infty}$, but may not approach every point of it, which is natural given the example in Remark \ref{rmk: counterexample for FB convergence to p_inf}.
\end{itemize}
\end{remark}

\begin{proof}
It follows from Theorem \ref{T.main_weaker_assumption} that, for any $\eta_0>0$ and any sufficiently small $r>0$, there exists $M$ sufficiently large such that for any $(x_0,t_0)\in \Gamma_{p_m}$ with $t_0\geq \eta_0>0$, $m\in [M,\infty)$, and $l\in [M,\infty)$, we have
$d(x_0,\Omega_{p_l}(t_0))\leq Cr$.
Then the first conclusion with finite $m$ and $l$ follows from Proposition \ref{P.4.6}.

In the case $m <\infty$ and $l = \infty$, we argue as in the proof of Lemma \ref{L.6.3}.
By Proposition \ref{L.5.3}, 
there exists some $c\ll 1$ such that, 
for any $r\ll 1$ and any $x_0\in \Gamma_{p_m}(t_0)$ with $t_0\in [\eta_0,T)$,
\[
\mint_{B(x_0,r)}p_m(x,t_0)\,dx\geq c r^{2-\frac{1}\gamma}.
\]
By \eqref{L1}, \eqref{mono stream}, and the assumptions on $b$, there exists $c'\ll 1$ such that for any $r'\ll 1$, if $m,k\in[M,\infty)$ with $M\gg 1$ depending on $r'$, 
\[
\mint_{B(x_0,r')}p_k(x,t)\,dx
\geq C\;\mint_{B(X(x_0,t_0;t-t_0),2r')} p_k(x,t)\,dx
\geq c' (r')^{2-\frac{1}\gamma}
\]
holds for any $t\in [t_0,t_0+\delta]$ with $\delta\ll r'$.
Taking the limit $k\to \infty$ and using \eqref{L1}, we obtain that for almost every $t\in [t_0,t_0+\delta]$,
\[
\mint_{B(x_0,r')} p_\infty(x,t)\,dx\geq c' (r')^{2-\frac{1}\gamma}.
\]
Fix $r'\ll 1$.
Thanks to \eqref{eqn: link integral of p_inf at a time with a spacetime integral},
\begin{align*}
\mint_{B(x_0,r')}p_\infty(x,t_0)\,dx 
\geq &\; \limsup_{\eps\to 0^+} \frac{1}{\eps} \int_0^\eps \mint_{B(x_0,r')} p_\infty(x,t_0+s)\,dx\,ds\\
\geq &\; \limsup_{\eps\to 0^+} \frac{1}{\eps} \int_0^\eps c'(r')^{2-\frac{1}{\gamma}}\,ds = c'(r')^{\frac{1}{\gamma}}.
\end{align*}
Hence, for $r'\ll 1$, there exists $M\gg 1$, such that for any $(x_0,t_0)\in \Gamma_{p_m}$ with $t_0\geq \eta_0>0$ and $m\in [M,\infty)$, we have
$d(x_0,\Omega_{p_\infty}(t_0))\leq Cr'$.

By Proposition \ref{P.4.6} and Lemma \ref{streamline}, for any $r\ll 1$ and $s\in [0,r^2]$, there exists $x_s\in \Gamma_{p_m}(t_0-s)$ such that $|x_0-x_s|\leq Cr$.
Repeating the above argument with $(x_0,t_0)$ replaced by $(x_s,t_0-s)$, we obtain the first conclusion for $m<\infty$ and $l = \infty$ as desired.


\medskip

Next we prove the second conclusion.
Since $m\in [M,\infty)$ and $x_0\in \Gamma_{p_m}(t_0)$, Proposition \ref{L.5.3} gives that for $s\in [0,\tau]$ and $\tau\ll1$ depending on $\eta_0$,
\[
B\big(X(x_0,t_0;-s),C_*s^\gamma)\big)\subseteq \Omega_{p_m}(t_0-s)^c.
\]
Fix an arbitrary $s\in [0,\tau]$. 
Taking $r$ such that $Cr\leq \frac12 C_*(s/2)^\gamma$ with $C$ from Theorem \ref{T.main_weaker_assumption}, we find that if $M=M(r,\eta_0)$ is sufficiently large, for all $l\in[M,\infty)$ and $\zeta\in (s/2,s]$,
\[
B\big(X(x_0,t_0;-\zeta), C_*\zeta^\gamma/2)\big)\subseteq \Omega_{p_l}(t_0-\zeta)^c.
\]
Thanks to \eqref{L1} and \eqref{eqn: property of p_inf by def}, this also holds for $l = \infty$.
Therefore, with $\zeta = s$,
\[
d\big(x_0,\Omega_{p_l}(t_0-s)^c\big)\leq d\big(x_0,B(X(x_0,t_0;-s), C_*s^\gamma/2)\big)\leq Cs.
\]

\end{proof}


In order to obtain improved convergence results involving the limiting solution, especially ruling out the case described in Remark \ref{rmk: counterexample for FB convergence to p_inf},
we additionally make the following assumption:
\beq\lb{L3}
\gamma_M':= \sup_{M\leq m<\infty} d_H\big(\{p_m^0(\cdot)>0\},\,\{p_\infty(\cdot,0)>0\}\big)\to 0\quad\text{as }M\to\infty.
\eeq
%
%
%
We also introduce the following notion of the ``good part of the boundary" 
of $\{p_\infty(\cdot,t)>0\}$:
\beq\lb{6.112}
\begin{aligned}
\tilde{\Gamma}_{p_\infty}(t):= \{x\in \bbR^d:&\;\text{for any small $r>0$, there exist space-time open sets }U_r,V_r\subseteq B(x,r)\times (-r,r)\\
& \text{ such that } p_\infty(\cdot, t)\text{ is essentially positive in $U_r$ and $p_\infty=0$ in $V_r$\}}.    
\end{aligned}
\eeq
Note that this may not coincide with $\Gamma_{p_\infty}(t)$.
Also denote $\tilde{\Gamma}_{p_\infty}:=\cup_{t\in (0,T)} \tilde{\Gamma}_{p_\infty}(t)\times \{t\}$.

Then we can show the following result.



\begin{theorem}\lb{T.main}
Under the assumptions of Theorem \ref{T.main_weaker_assumption} and also \eqref{L3},
for any $0<\eta_0\ll 1$, there exists $C(\eta_0)>0$ such that for any $0<r\ll 1$, there is $M\gg 1$ which depends on $\eta_0$, $r$ and the conditions, satisfying that: 
\begin{enumerate}
\item For any $m\in [M,\infty)$ and $t_0\in [\eta_0,T)$,
\beq\lb{part2_repeat}
\Omega_{p_m}(t_0)\subseteq \text{\,the $Cr$-neighbourhood of $\Omega_{p_\infty}(t_0)$}.
\eeq
Consequently, \eqref{part2} holds for all $m,l\in [M,\infty]$ with $M$ sufficiently large.


\item 

For any $l\in [M,\infty)$, $t_0\in [\eta_0,T)$ and any $x_0\in\tilde{\Gamma}_{p_\infty}(t_0)$, 
\[
d(x_0,\Omega_{p_l}(t_0-s))\leq Cr\quad\text{for all }s\in [0,r^2].
\]
Moreover, with $M$ additionally depending on the initial data and yet with $l$, $t_0$, and $x_0$ satisfying the same conditions as above, 
it holds that
\[
d(x_0,\Omega_{p_l}(t_0-r)^c)\leq Cr.
\]
Therefore, for all $l\in [M,\infty]$ with $M$ sufficiently large,
\[
\sup_{x_0\in \tilde{\Gamma}_{p_\infty}(t_0) }d \big(x_0,\Gamma_{p_l}(t_0-r^2)\big)\leq Cr.
\]
\end{enumerate}


\end{theorem}

\begin{proof}

To prove \eqref{part2_repeat}, we start by showing that $\Omega_{p_\infty}(t)$ is non-decreasing along the streamlines for $t\geq 0$. 
Indeed, for any $x\in \Omega_{p_\infty}(t)$, by \eqref{eqn: property of p_inf by def}, for any sufficiently small $\eps>0$, $\int_0^\eps\int_{B(x,\eps)}p_\infty(y,t+s)\,dy\,ds>0$. By \eqref{L1}, the same holds with $p_l$ in place of $p_\infty$ for all $l$ sufficiently large. Then \eqref{mono stream}, \eqref{L1}, and \eqref{eqn: property of p_inf by def} yield the claim first for $t>0$ and then for all $t\geq 0$.
%

Let $y_0\in\Omega_{p_m}(t_0+r^2)$ and assume for contradiction that $B(y_0,2Ar)\subseteq\Omega_{p_\infty}(t_0+r^2)^c$ for some large $A>0$.
Proposition \ref{P.4.6} and the regularity of $b$ yield that there exists $x_0\in \Omega_{p_m}(t_0)$ such that $|x_0-y_0|\leq Cr$ and, if $A\geq 2C$, $ B(x_0,\frac32Ar)\subseteq \Omega_{p_\infty}(t_0+r^2)^c$. The monotonicity property of $\Omega_{p_\infty}$ then yields that $B(x_0,Ar)\subseteq \Omega_{p_\infty}(t_0)^c$.
With the monotonicity property again and \eqref{L3}, an identical argument as in Lemma \ref{L.6.3} can show that \eqref{6.666} holds for all $t\in [t_0,t_0+r^2]$ and all large but finite indices $m$. 
Hence, by \eqref{L1}, $p_\infty$ cannot be identically $0$ in
$B(x_0,Cr)\times [t_0,t_0+r^2]$.
So there exists $\delta\in [0,r^2]$ such that
\[
B(x_0,Cr)\cap \Omega_{p_\infty}(t_0+\delta)\neq \varnothing.
\]
%
Since $\Omega_{p_\infty}(t)$ is non-decreasing along the streamlines, for some $C'>C$,
\[
B(y_0,C'r)\cap \Omega_{p_\infty}(t_0+r^2)\neq \varnothing,
\]
which implies
\[
\Omega_{p_m}(t_0+r^2)
\subseteq \text{\,the $C'r$-neighbourhood of $\Omega_{p_\infty}(t_0+r^2)$}.
\]
Replacing $t_0$ by $t_0-r^2$ yields \eqref{part2_repeat}.



%

\medskip

Next we prove the second part of the statement.
It follows from Theorem \ref{T.main_weaker_assumption} that, for any $\eta_0>0$ and any sufficiently small $r>0$, there exists $M$ sufficiently large such that for any $(x_0,t_0)\in \Gamma_{p_\infty}$ with $t_0\geq \eta_0>0$ and any $l\in [M,\infty)$, we have $d(x_0,\Omega_{p_l}(t_0))\leq Cr$.
The first conclusion then follows from Proposition \ref{P.4.6}.

By the definition of $\tilde{\Gamma}_{p_\infty}(t_0)$, for any small $r>0$, there exists $(x_1,t_1)\in B(x_0,r)\times (t_0-r,t_0+r)$ such that $B(x_1,r_1)\subseteq \Omega_{p_\infty}(t_1)^c$ for some $r_1\in (0,r)$. 
By \eqref{part2_repeat}, $B(x_1,r_1/2)\subseteq \Omega_{p_l}(t_1)^c$ for all $l$ sufficiently large.
It follows from Lemma \ref{streamline} that for any $s\in [t_0-r,t_1)$,
\[
\{ X(y,t_1;s-t_1):y\in B(x_1,r_1/2)\}\subseteq \Omega_{p_l}(s)^c.
\]
Hence, for all $l\in [M,\infty]$ with $M$ sufficiently large, 
\[
d\left(x_0,\Omega_{p_l}(t_0-r)^c\right)\leq d\left(x_0,X(x_1,t_1;t_0-r-t_1)\right)\leq Cr.
\]
Let us remark that here $M$ depends on $(x_0,t_0)\in \tilde{\Gamma}_{p_\infty}$. By a compactness argument, it then only depends on the initial data, the assumptions, and the universal constants.
\end{proof}

\begin{remark}\lb{R.6.1}
Theorems \ref{T.main_weaker_assumption}--\ref{T.main} address the convergence after some positive time $\eta_0$, with $\eta_0$ being arbitrary.
Let us briefly discuss behavior of the supports of the solutions within time $[0,\eta_0]$ with $\eta_0\ll 1$.
In this regime, the convergence stems directly from that of the initial data.
\begin{enumerate}[(1)]
\item When $t\in [0,\eta_0]$, under the assumption \eqref{L2}, there exists a universal $C$ such that
\beq\lb{6.111}
\Omega_{p_m}(t)\subseteq \text{\,the $C\eta_0^{1/2}$-neighbourhood of $\Omega_{p_l}(t)$}
\eeq
for all $m,l\in [M,\infty)$ such that $\gamma_M\leq \eta_0^{1/2}$.
This follows immediately from Lemma \ref{streamline} and Proposition \ref{P.4.6}. 

\item For the case $m=\infty$ and $l\in [M,\infty)$, by Proposition \ref{P.4.6}, $p_l(x,t)=0$ for all $t\in [0,\eta_0]$ and $x$ outside the $C\eta_0^{1/2}$-neighborhood of $\{p_{l}^0(\cdot)>0\}$.
Thanks to \eqref{L2}, when $M$ is sufficiently large depending on $\eta_0$, for any $l,m'\in [M,\infty)$, 
$p_{m'}(x,t)=0$ for all $t\in [0,\eta_0]$ and $x$ outside the $C\eta_0^{1/2}$-neighborhood of $\{p_l^0(\cdot)>0\}$ with a larger $C$.
Sending $m'\to \infty$ and using \eqref{L1}, we obtain that $p_\infty(x,t)=0$ almost everywhere in the same region.
Thanks to \eqref{eqn: property of p_inf by def} and Lemma \ref{streamline}, we obtain \eqref{6.111} with $m=\infty$ and $l\in [M,\infty)$. Note that the assumption \eqref{L3} is not needed here.

\item To have \eqref{6.111} valid for $m\in [M,\infty)$ and $l=\infty$, we need to assume \eqref{L3}. Indeed, this follows from \eqref{L3}, the monotonicity property of $\Omega_{p_\infty}(\cdot)$ (see the proof of Theorem \ref{T.main}), and Proposition \ref{L.5.3}.  

\end{enumerate}
\end{remark}

\section{Hausdorff Dimensions of the Free Boundaries}\lb{S6}

In this section, we estimate the Hausdorff dimension of the free boundary $\Gamma_{p_m}(t)$ for each $t>0$ and finite $m$, and then extend that to $\tilde{\Gamma}_{p_\infty}(t)$.

Let us start with some assumptions. 
The first one is on the density variable of the solution to the PME-type equations; it can be verified under suitable conditions (see e.g.\;Theorem \ref{thm: properties of solutions to PME}).

\begin{itemize}
    \item[{\bf(H1)}] Stability of the densities in $L^1$: there exists $C$ depending on the universal constants such that, if $\varrho_1,\varrho_2$ are two continuous, non-negative solutions to \eqref{main}, then for all $t\in (0,T)$,
\[
\left|\int_{\bbR^d}\varrho_1(x,t)-\varrho_2(x,t)\, dx\right|\leq C\int_{\bbR^d}|\varrho_1-\varrho_2|(x,0)\, dx.
\]
Moreover, we assume Lipschitz continuity in $t$ of the total mass: for any $t,s\in [0,T)$, 
\[
\left|\int_{\bbR^d}\varrho_1(x,t)\, dx-\int_{\bbR^d}\varrho_1(x,s)\, dx\right|\leq C|t-s|.
\]
\end{itemize}

The next condition is technical, which is a strengthening of \eqref{cond} and is used to guarantee that certain modifications of the density variables are sub- or super-solutions to \eqref{main}; see Lemma \ref{L.6.1}.
\begin{itemize}
\item[{\bf(H2)}] There exists $\tilde{\sigma}>0$ such that
\[
\nabla \cdot b(x,t)+f(x,t,p)\geq \tilde{\sigma}>0\quad\text{and}\quad f_p(x,t,p)\leq 0\quad\text{ for }(x,t,p)\in Q_T\times [0,\infty).
\]
\end{itemize}

Finally, we also need the initial density to enjoy $L^1$-stability under certain pertubations.
\begin{itemize}
    \item[{\bf(H3)}] There exists $\zeta:(0,1)\times [2,\infty)\to (0,\infty)$ satisfying
\[
\limsup_{r\to0}r^{-\sigma_m}\zeta(r,m)\leq C_m \quad\text{ for some }C_m>0,\sigma_m\in (0,1],
\]
and that for all $r$ sufficiently small and $m\geq 2$, the initial density variable satisfies
\beq\lb{6.0}
\int_{\bbR^d}\left(\sup_{y\in B(0,r)}\varrho_m(x+y,0) -\inf_{y\in B(0,r)}\varrho_m(x+y,0)\right)dx\leq \zeta(r,m).
\eeq
\end{itemize}

\begin{remark}\lb{R3}
Recall that $\varrho_m(\cdot,0)=P_m^{-1}(p_m^0)$. If $p_m^0$ are characteristic functions of some bounded open sets whose boundaries have uniformly bounded finite $(d-1)$-dimensional Hausdorff measure, then the condition {\bf(H3)} holds with $\zeta(r,m)\equiv Cr$ for some $C>0$.

As for continuous initial datum, if we assume
\begin{enumerate}[(a)]
\item \eqref{introgr} holds with the power $2-\varsigma_0$ replaced by $\varsigma_m\in (0,1)$;
\smallskip

\item $p_m^0$ is uniformly bounded and uniformly Lipschitz continuous for all $m\geq2$; 

\smallskip

\item and for all $m\geq 2$, $\partial\{p_m^0>0\}$ has uniformly bounded finite $(d-1)$-dimensional Hausdorff measure;
\end{enumerate}
then the condition {\bf(H3)} holds with $\zeta(r,m)=C r+\frac{C}{m-1}r^{1-\varsigma_m}$ for some $C$ independent of $m$.
Indeed, by virtue of the assumptions, for fixed $m\geq 2$ and all sufficiently small $r$, measure of the set 
\[
\calN:= \big\{x\in \bbR^d \,|\, d(x,\pa\Omega_{p_m}(0))\leq 2r\big\} 
\]
is bounded by $Cr$. 
If $x\in\Omega_{p_m}(0)$ and $d(x,\Omega_{p_m}(0)^c)\geq r$, by \eqref{introgr} with $2-\varsigma_0$ replaced by $\varsigma_m$, we have $p_m(x,0)\geq \gamma_0 r^{\varsigma_m}$. Therefore, if $x,y$ are such points, the Lipschitz condition yields
\[
|\varrho_m(x,0)-\varrho_m(y,0)|\leq \big|(p_m^0)^\frac1{m-1}(x)-(p_m^0)^\frac1{m-1}(y)\big|\leq \frac{C}{m-1}r^{-\varsigma_m}|x-y|.
\]
Therefore,
\begin{align*}
&\;\int_{\bbR^d}\left(\sup_{y\in B(x,r)}\varrho_m(y,0) -\inf_{y\in B(x,r)}\varrho_m(y,0)\right)dx\\
\leq &\;\int_{\calN}\sup_{y\in B(x,r)}\varrho_m(y,0)\,dx
+\int_{\Omega_{p_m}(0)\setminus \calN}\left(\sup_{y\in B(x,r)}\varrho_m(y,0) -\inf_{y\in B(x,r)}\varrho_m(y,0)\right)dx\\
\leq &\; C r+\frac{C}{m-1}r^{1-\varsigma_m},
\end{align*}
which implies the claim. 
In particular, $
\limsup_{r\to0}r^{-1+\varsigma_m}\zeta(r,m)\leq C$.
Moreover, since $\varsigma_m\in(0,1)$, we have 
\[
\lim_{m\to\infty}\zeta(r,m)\leq Cr\quad\text{with $C$ being independent of $m$.}
\]    
\end{remark}

The strategy of bounding the Hausdorff dimension of $\Gamma_{p_m}(t)$ is motivated by \cite{kim2019singular} while there are notable differences as discussed in the introduction. The major tool is the inf- and sup-convolution technique.
Suppose $\rho\in C^\infty(\bbR^d\times(0,T))$ and let $r = r(t)\in C^\infty((0,T))$ satisfying $0<r\leq 1$. Define
\begin{equation*}
\rho_1(x,t):=\sup_{y\in B(x,r(t))}
\rho(y,t),\qquad 
\rho_2(x,t):=\inf_{y\in B(x,r(t))}
\rho(y,t).
\end{equation*}
Then $\rho_1$ and $\rho_2 $ are Lipschitz continuous. 
They are called the sup- and inf-convolution of the smooth function $\rho$, respectively.

Let us mention some basic properties of the sup-convolution of smooth functions.
Let $y_{1,t}(\cdot)\in \overline{B(\cdot,r(t))}$ be such that $\rho_1(\cdot,t)=\rho(y_{1,t}(\cdot),t)$.
Then we have the following:
\beq\lb{6.1}
(\Delta \rho_1)(x,t)\geq (\Delta\rho)(y_{1,t}(x),t),\quad (\nabla \rho_1)(x,t)=(\nabla\rho)(y_{1,t}(x),t)
\eeq
and
\beq\lb{6.2}
(\partial_t \rho_1)(x,t)=(\partial_t \rho)(y_{1,t}(x),t)+r'(t)|\nabla \rho|(y_{1,t}(x),t).
\eeq
The first inequality in \eqref{6.1} is understood in the sense of distribution. 
The proof can be found in \cite{cbook,kim2019singular,kimzhang21}.
Similarly, assuming $y_{2,t}(\cdot)\in\overline{B(\cdot,r(t))}$ to satisfy that $\rho_2(\cdot,t)=\rho(y_{2,t}(\cdot),t)$, then
\[
(\Delta \rho_2)(x,t)\leq (\Delta\rho)(y_{2,t}(x),t),\quad (\nabla \rho_2)(x,t)=(\nabla\rho)(y_{2,t}(x),t),
\]
\[
(\partial_t \rho_2)(x,t)=(\partial_t \rho)(y_{2,t}(x),t)-r'(t)|\nabla \rho|(y_{2,t}(x),t).
\]

Let $\varrho=\varrho_m$ be a solution in $Q_T$ to \eqref{main} with $m\geq 2$. We are first going to show that a modified version of the  sup- (resp.\;inf-) convolution of $\varrho$ is a subsolution (resp.\;a supersolution) to \eqref{main}. 
By Lemma \ref{approximation}, one only needs to prove this for smooth $\varrho$.

\begin{lemma}\label{L.6.1}
Assume \eqref{1.7}, \eqref{1.8}, and {\rm\bf (H2)}.
Let $\varrho=\varrho_m$ be a solution in $Q_T$ to \eqref{main} with $m\geq 2$.
Then there exist constants $L,C\geq 1$ and $\tau_0>0$ depending only on the universal constants and $\tilde{\sigma}$ such that, for all $r_0>0$ sufficiently small and $\alpha:=Cr_0<\frac12$, if $r(t):=r_0e^{-Lt}$ and
\begin{equation}
\begin{split}
u_1(x,t):=&\; (1-\alpha)^{\frac{1}{m-1}}\sup_{y\in B(x,r(t))}\varrho(y,(1-\alpha)t),\\
u_2(x,t):=&\;(1+\alpha)^{\frac{1}{m-1}}\inf_{y\in B(x,r(t))}\varrho(y,(1+\alpha)t),
\end{split}
\label{eqn: modified sup and inf convolution}
\end{equation}
then $u_1$ is a subsolution to \eqref{main} and $u_2$ is a supsolution to \eqref{main} for $t\in (0,\tau_0)$.
\end{lemma}

\begin{proof}
We will only show that $u_1$ is a subsolution, and the proof for $u_2$ being a supersolution is similar. 
Below we write $u=u_1$ and $y_t=y_{1,t}$.
Let $\calG$ denote the operator in \eqref{main}, i.e.,
\[
\calG(\rho):=\partial_t\rho -\Delta\rho^m-\nabla\cdot (\rho\, { b }(x,t))-\rho f(x,t,P_m(\rho))
\]
and the goal is to show that $\calG(u)\leq 0$ in $\bbR^d\times(0,\tau_0)$. Thanks to Lemma \ref{approximation}, it suffices to prove this with $u$ being Lipschitz continuous. We will only give a formal proof.

Below we write $\varrho$ and its derivatives as those evaluated at $(y_t(x),(1-\alpha)t)$, and $r=r(t)$. 
Let us estimate each term in $\calG (u)$. 
First, by \eqref{6.2}, we have that
\begin{equation}\label{partial}
\begin{aligned}
    \partial_t u &= (1-\alpha)^\frac{m}{m-1}(\partial_t\varrho)+(1-\alpha)^\frac{1}{m-1}r'(t)|\nabla \varrho|\\
    &= (1-\alpha)^\frac{m}{m-1}(\partial_t\varrho)-(1-\alpha)^\frac{1}{m-1}Lr|\nabla \varrho|.
\end{aligned}
\end{equation}
It follows from \eqref{6.1} that
\[
-\Delta u^m\leq -(1-\alpha)^{\frac{m}{m-1}}\Delta \varrho^m\quad (\text{in distribution}),
\]
and $\nabla u=(1-\alpha)^\frac{1}{m-1}\nabla \varrho$.
Also using the regularity assumption on ${b}$ and $|y_t-x|\leq r$, we have
\begin{equation}
    \label{6.3}
\begin{aligned}
    -\nabla (u b)(x,t)& 
 =-(1-\alpha)^\frac{1}{m-1}(\nabla \varrho \cdot b(x,t)+\varrho\nabla\cdot b(x,t))\\
    &\leq -(1-\alpha)^\frac{1}{m-1}(\nabla \varrho \cdot b(y_t,(1-\alpha)t)+\varrho(\nabla\cdot b)(y_t,(1-\alpha)t))+C(|\nabla\varrho|+\varrho)(r+\alpha t).
\end{aligned}
\end{equation}
Using the regularity of $f$ and that $f_p \leq 0$, direct computation yields
\begin{equation}
    \label{6.4}
    \begin{aligned}
-uf(x,t,P_m(u))   &=-(1-\alpha)^\frac{1}{m-1}\varrho f(x,t,(1-\alpha) P_m(\varrho)) \\
   &\leq -(1-\alpha)^\frac{1}{m-1}\varrho f(y_t,(1-\alpha)t,(1-\alpha)P_m(\varrho))+ C\varrho (r+\alpha t)\\
      &\leq -(1-\alpha)^\frac{1}{m-1}\varrho f(y_t,(1-\alpha)t, P_m(\varrho))+ C\varrho (r+\alpha t).   
    \end{aligned}
\end{equation}
Note that $\nabla \cdot b+ f\geq \tilde{\sigma}>0$ and, since $\alpha\in (0,\frac12)$ and $m\geq 2$, $(1-\alpha)^\frac{1}{m-1}-(1-\alpha)^\frac{m}{m-1}\in [\frac12\alpha,\alpha]$ for all $m\geq 2$.
Therefore, \eqref{6.3} and \eqref{6.4} yield that
\[
\begin{aligned}
&\;-\nabla (u b)(x,t)-uf(x,t,P_m(u))\\
\leq &\;-(1-\alpha)^\frac{1}{m-1}\nabla \cdot (\varrho  b)(y_t,(1-\alpha)t)-(1-\alpha)^\frac{1}{m-1}\varrho f(y_t,(1-\alpha)t, P_m(\varrho))+C(|\nabla\varrho|+\varrho)(r+\alpha t)\\
\leq &\; -(1-\alpha)^\frac{m}{m-1}(\nabla \cdot (\varrho  b) +\varrho f)-\big((1-\alpha)^\frac{1}{m-1}-(1-\alpha)^\frac{m}{m-1}\big)(\varrho\nabla \cdot b +\varrho f+\nabla \varrho\cdot  b)\\
&\;+C(|\nabla\varrho|+\varrho)(r+\alpha t)\\
\leq &\; -(1-\alpha)^\frac{m}{m-1}\left(\nabla\cdot ( \varrho b)+\varrho f\right) -\frac12\alpha\tilde{\sigma}\varrho +\alpha\|b\|_\infty |\nabla \varrho|+C(|\nabla\varrho|+\varrho)(r+\alpha t) .
\end{aligned}
\]
%
Combining this with \eqref{partial} and using \eqref{main}, $\alpha\in (0,\frac12)$, and $m\geq 2$, we obtain that, for some universal $C_1\geq 1$,
\begin{align*}
{\calG}(u)&\leq  (1-\alpha)^\frac{m}{m-1}\calG(\varrho)(y_t,(1-\alpha)t)-(1-\alpha)^\frac{1}{m-1}Lr|\nabla \varrho|-\frac12\alpha\tilde{\sigma}\varrho+C(|\nabla\varrho|+\varrho)(r+\alpha t)+C\alpha|\nabla\varrho|\\
& \leq  -\frac12 Lr|\nabla \varrho| -\frac12 \alpha\tilde{\sigma}\varrho+C_1(|\nabla\varrho|+\varrho)(r+\alpha t)+C_1\alpha|\nabla\varrho|.
\end{align*}
Now choose 
\[
\alpha:=\frac{4C_1r_0}{\tilde{\sigma}},\quad 
L:=4C_1+\frac{8eC_1^2}{\tilde{\sigma}},\quad \tau_0:=\min\left\{\frac{1}L,\,\frac{\tilde{\sigma}}{4eC_1},\,\frac{T}{2}\right\}.
\]
By requiring $r_0$ to be sufficiently small, we can make $\alpha<\frac12$.
Then $\alpha t\leq r_0e^{-1}\leq r$ for $t\in (0,\tau_0)$, and $Lr\geq 4C_1r+ 2C_1\alpha$. 
Therefore, we obtain that
$\calG(u)\leq 0$ for all $(x,t)\in\bbR^d\times (0,\tau_0)$.
\end{proof}

In the next lemma, we further assume \textbf{(H1)} and \textbf{(H3)}. 
We will apply the inf- and sup-convolution construction to show that the property \eqref{6.0} propagates to all finite times.

\begin{lemma}\lb{L.6.2}
Assume \eqref{1.7}, \eqref{1.8}, and {\rm\bf (H1)--(H3)}. 
Suppose that $\varrho_m$ is a continuous solution to \eqref{main} in $Q_T$.
Then there exist universal $\tilde{r}_0>0$ and $C>0$ such that, for all $r\in (0,\tilde{r}_0)$ and $m\geq 2$, we have
\beq\lb{6.9}
\sup_{ t\in [0,T) }\int_{\bbR^d}\left(\sup_{y\in B(0,r)}\varrho_m(x+y,t)-\inf_{y\in B(0,r)}\varrho_m(x+y,t)\right) dx
\leq C \big(r+\zeta(Cr,m)\big).
\eeq
\end{lemma}
\begin{proof}
Let $r_0>0$ be sufficiently small from Lemma \ref{L.6.1}, and define $\alpha=4C_1r_0/\tilde{\sigma}$ and $r(t)=r_0e^{-Lt}$ as before.
Let $u_1$ and $u_2$ be defined as in \eqref{eqn: modified sup and inf convolution}.
We have shown that, for some $\tau_0>0$, $u_1$ is a subsolution to \eqref{main} in $\bbR^d\times (0,\tau_0)$, while $u_2$ is a supersolution to \eqref{main} in $\bbR^d\times (0,\tau_0)$.

Let $\rho_1$ and $\rho_2$ be solutions to \eqref{main} with initial data $u_1(\cdot,0)$ and $u_2(\cdot,0)$, respectively. 
By the comparison principle,
\beq\lb{6.8}
u_1\leq \rho_1\quad\text{and}\quad\rho_2\leq u_2\quad\text{in $\bbR^d\times (0,\tau_0)$.}
\eeq
Thanks to \textbf{(H3)} and the compact support of $\varrho(\cdot,0)$,
\[
\int_{\bbR^d} |\rho_1-\varrho|(x,0)\,dx
+
\int_{\bbR^d} |\rho_2-\varrho|(x,0)\,dx
\leq \zeta(r_0,m)+C\big(1-(1-\alpha)^{\frac1{m-1}}\big).
\]
By the $L^1$-stability of solutions in \textbf{(H1)}, we get for all $t\in [0,T)$,
\[
\int_{\bbR^d} \rho_1(x,t)-\varrho(x,t)\,dx\leq C\zeta(r_0,m) + C\alpha\quad\text{and}\quad \int_{\bbR^d} \varrho(x,t)-\rho_2(x,t)\, dx\leq C\zeta(r_0,m) + C\alpha.
\]
Since the $L^1$-norm of the solutions is Lipschitz in time by \textbf{(H1)}, we get for $t\in (0,T/2)$,
\beq\lb{7.100}
\begin{aligned}
&\;\int_{\bbR^d} \rho_1\left(x,\frac{t}{1-\alpha}\right)-\rho_2\left(x,\frac{t}{1+\alpha}\right) dx \\
\leq &\; C\zeta(r_0,m) + C\alpha+\int_{\bbR^d} \varrho\left(x,\frac{t}{1-\alpha}\right)-\varrho\left(x,\frac{t}{1+\alpha}\right) dx\\
\leq &\;C\zeta(r_0,m) + Cr_0.  
\end{aligned}
\eeq
In the last line, we used the fact $\alpha\leq Cr_0$.
Then \eqref{eqn: modified sup and inf convolution}, \eqref{6.8} and \eqref{7.100} imply that, for $t\in [0,\tau_0/2)$, 
\begin{align*}
&\;\int_{\bbR^d}\sup_{y\in B(0,r(t))}\varrho(x+y,t) -\inf_{y\in B(0,r(t))}\varrho(x+y,t)\, dx\\
=&\; \int_{\bbR^d}(1-\alpha)^{-\frac{1}{m-1}}u_1\left(x,\frac{t}{1-\alpha}\right)-(1+\alpha)^{-\frac{1}{m-1}}u_2\left(x,\frac{t}{1+\alpha}\right) dx\\
\leq&\; \int_{\bbR^d}\rho_1\left(x,\frac{t}{1-\alpha}\right)-\rho_2\left(x,\frac{t}{1+\alpha}\right) dx+C\alpha\\
\leq &\; C\zeta(r_0,m)+C r_0.    
\end{align*}
By iteration, there exists $C>0$ such that for all $t\in [0,T)$ we have
\[
\int_{\bbR^d}\sup_{y\in B(0,r(t))}\varrho(x+y,t)-\inf_{y\in B(0,r(t))}\varrho(x+y,t)\, dx\leq C \big(r_0+\zeta(r_0,m)\big).
\]
Recall that $r(t)=r_0e^{-Lt}$.
We then take $\tilde{r}_0 = r(T) = r_0 e^{-LT}$ and obtain the desired claim.
\end{proof}

Now we are ready to prove the main result on the Hausdorff dimensions of the free boundaries. 

\begin{theorem}
\label{thm: estimate Hausdorff dim of FB}
Suppose that for some $\eta_0\in [0,T)$, there exists $c_*=c_*(\eta_0,T)>0$, $r_* = r_*(\eta_0,T)>0$, and $\mu\in (0,2)$ such that, for all $r\in (0,r_*)$ and $m\geq 2$,
\beq\lb{6.5}
\mint_{B(x_0,r)}p_m(x,t_0)\, dx\geq c_* r^\mu\quad\text{for any $t_0\in[\eta_0,T)$ and }x_0\in \Gamma_{p_m}(t_0).
\eeq
Assume \eqref{1.7}, \eqref{1.8}, and {\rm\bf (H1)--(H3)}.
Then there exists $C>0$ independent of $m\geq 2$ such that
\[
\calH^{d_m}(\Gamma_{p_m}(t))\leq CC_m\quad\text{for all }t\in [\eta_0,T),
\]
where $d_m:=d-\sigma_m+\frac{\mu}{m-1}$, and $C_m$ and $\sigma_m$ are from {\rm\bf(H3)}. 

Furthermore, if there exists $C$ independent of $m$ such that for each sufficiently small $r$ we have
\beq\lb{6.13}
\liminf_{m\to\infty}\zeta(r,m)\leq Cr,
\eeq
and the conclusion of Theorem \ref{T.main} holds, then $\tilde{\Gamma}_{p_\infty}(t)$ has finite $(d-1)$-dimensional Hausdorff measure for any $t\in [\eta_0,T)$, where $\tilde{\Gamma}_{p_\infty}(t)$ is defined in \eqref{6.112}.
\end{theorem}

Let us remark that the assumption \eqref{6.5} is proved in Proposition \ref{L.5.3} under suitable conditions. 
Also, if $\lim_{m\to\infty}\sigma_m= 1$, then the Hausdorff dimension of the free boundary $\Gamma_{p_m}(t)$ decreases to $d-1$ as $m\to\infty$. 
This is the case for the two typical scenarios discussed in Remark \ref{R3}.


\begin{proof}
Take an arbitrary $m\geq 2$.
Fix $t\in [\eta_0,T)$,  
and take $R\in (0,r_*)$ to be chosen.
Let $\calO$ be the collection of all closed balls of radius $R$ with their centers lying in $\Gamma_{p_m}(t_0)$. 
It follows from the Vitali's covering lemma that there is a family of disjoint balls $\calB:=\{B^i\}\subset \calO$, which is at most finite (cf.\;Lemma \ref{uniformb}), such that $\{3B^i\}$ covers the boundary $\Gamma_{p_m}(t_0)$.
Here $3B^i$ denotes the ball having the same center as $B^i$ and yet with the radius tripled.
It suffices to find an upper bound for the cardinality of $\calB$, denoted by $\|\calB\|$.

Define
\[
\bar\rho(x):=\sup_{y\in B(x,R)}\varrho_m(y,t_0),
\qquad 
\un\rho(x):=\inf_{y\in B(x,R)}\varrho_m(y,t_0).
\]
Writing 
$
\overline\Omega :=\{x:\, \bar\rho(x)>0\}$ and $\un\Omega :=\{x:\, \un\rho(x)>0\}$,
it is easy to see that
\beq\lb{6.11}
\un\Omega  \subseteq \overline\Omega \quad\text{and}\quad B^i\subseteq B_R(\Gamma_{\varrho_m}(t_0))\subseteq \overline\Omega \setminus  \un\Omega  =:\calN.
\eeq

Suppose that $y_1\in\Gamma_{p_m}(t_0)$ is the center of $B^1$. 
By \eqref{6.5}, 
\[
\mint_{B(y_1,R/2)}p_m(x,t_0)\, dx\geq c_*2^{-\mu} R^{\mu}.
\]
Hence, there exists at least one point $z\in B(y_1,R/2)$ such that 
\[
{\varrho_m}(z,t_0)=\left(\frac{m-1}{m}p_m(z,t_0)\right)^\frac{1}{m-1}\geq c R^{\mu/(m-1)}
\]
for some $c>0$ depending only on $c_*$. Notice that 
\[
z\in  B(y_1,R/2)\subseteq B(x,R)\quad\text{for any }x\in B(y_1,R/2).
\]
Thus, by the sup-convolution construction, for any $x\in B(y_1,R/2)$, we have
\[
\bar\rho(x)\geq {\varrho_m}(z,t_0)\geq cR^{\mu/(m-1)}.
\]
This implies that 
\[
\int_{B^1}\bar\rho(x)\, dx\geq \int_{B(y_1,R/2)}\bar\rho(x)\, dx\geq c |B^1| R^{\mu/(m-1)}
\]
for some $c$ depending only on $c_*$. This, together with \eqref{6.11}, yields that
\beq\lb{6.10}
\int_\calN 
\bar\rho(x)\, dx\geq \sum_i\int_{B^i}\bar\rho(x)\, dx\geq c\|\calB\| R^{d+\mu/(m-1)}.
\eeq
 
We further assume $R$ to be smaller than $\tilde{r}_0$ from Lemma \ref{L.6.2}.
Observe that $\bar\rho(x)\geq\un\rho(x)$ and $\un\rho(x)=0$ in $\calN$ by \eqref{6.11}. Therefore, by \eqref{6.9} with $r$ replaced by $R$, we get
\begin{align*}
\int_\calN \bar\rho(x)\, dx& \leq     \int_{\bbR^d} \bar\rho(x)-\un\rho(x)\, dx=\int_{\bbR^d}\sup_{y\in B(0,R)}\varrho_m(x+y,t_0)-\inf_{y\in B(0,R)}\varrho_m(x+y,t_0)\, dx \\
&\leq C\big(R+\zeta(CR,m)\big)\leq CC_m R^{\sigma_m}.
\end{align*}
Combining this with \eqref{6.10}, we obtain that
\beq\lb{6.12}
\|\calB\|\leq CC_m R^{\sigma_m-d-\mu/(m-1)}
\eeq
with $C$ being independent of $m\geq 2$, $t_0\in [\eta_0,T)$ and all $R$ sufficiently small.
This implies that the Hausdorff dimension of $\Gamma_{p_m}(t_0)$ is at most $d_m:=d-\sigma_m+\frac{\mu}{m-1}$ and
\[
\calH^{d_m}(\Gamma_{p_m}(t_0))\leq CC_m
\]
with $C>0$ independent of $m\geq 2$ and $t_0\in [\eta_0,T)$.

\medskip

Finally, we use \eqref{6.13} and the convergence of free boundaries in the space-time Hausdorff distance to conclude that $\tilde{\Gamma}_{p_\infty}(t)$ has finite $(d-1)$-dimensional Hausdorff measure for $t\in [\eta_0,T)$.
Indeed, let $\calO_\infty$ be the collection of all closed balls centered at $\tilde{\Gamma}_{p_\infty}(t)$ with radius $R>0$. 
As before, there is a finite family of disjoint balls $\calB_\infty:=\{B^i_\infty\}\subset \calO_\infty$ such that $\{3B^i_\infty\}$ covers $\tilde{\Gamma}_{p_\infty}(t)$.

By Theorem \ref{T.main}, there exists $c>0$ such that for any $x^i$ being the center of $B^i_\infty$,
\[
d(x^i,\Gamma_{p_m}(t'))< R/2\quad\text{with }t':=t-cR^2
\]
when $m$ is sufficiently large. Thus, each $\frac12B^i_\infty$ intersects with $\Gamma_{p_m}(t')$. Therefore, for each $i$ we can adjust the center of $\frac{1}{2}B_\infty^i$ to obtain another collection of balls $\{\tilde{B}^i\}$ such that, $\tilde{B}^i\subseteq B^i_\infty$, each $\tilde{B}^i$ has radius $\frac{1}2R$ and its center lies on $\Gamma_{p_m}(t')$. It is clear that $\{\tilde{B}^i\}$ are disjoint, and $\{8\tilde{B}^i\}$ covers $\Gamma_{p_m}(t')$ provided that $m$ is sufficiently large.
Then the previous argument implies a parallel version of \eqref{6.12}:
\[
\|\calB_\infty\|\leq C\zeta(R,m)R^{-d-\mu/(m-1)}
\]
with $C>0$ independent of $m$. 
Letting $m\to\infty$ and using \eqref{6.13} yield $\|\calB_\infty\|\leq CR^{1-d}$, which implies that $\tilde{\Gamma}_{p_\infty}(t)$ has finite $(d-1)$-dimensional Hausdorff measure.
\end{proof}

\appendix

\section{Proof of Theorem \ref{thm: incompressible limit}}
\label{sec: proof of incompressible limit}

The proof is lengthy but standard.
It proceeds in several steps.
\begin{enumerate}
\item Show that $\{\varrho_m\}_{m>1}$ and $\{p_m\}_{m>1}$ are uniformly bounded and uniformly compactly supported, which has been done in Lemma \ref{uniformb}.

\item Derive uniform-in-$m$ estimates for $\{\varrho_m\}$ and $\{p_m\}$ with $m$ being sufficiently large.
\item Pass to the limit to justify the incompressible limit.
\item Finally, show that the incompressible limit has a unique solution.
\end{enumerate}
A major part of the following argument is adapted from that in \cite{David_S}.

\subsection{Uniform-in-$m$ a priori estimates}

It is clear that Lemma \ref{uniformb} implies uniform $L^1$-bound for $p_m$ and also $\rho_m$.
More precisely, there exists a universal constant $C>0$, such that
\beq
\|p_m(\cdot,t)\|_{L^1(\bbR^d)}\leq C,\quad
\|\varrho_m(\cdot,t)\|_{L^1(\bbR^d)}\leq C
\label{eqn: spatial L^1 bound for p_m and rho_m}
\eeq
holds for all $t\in [0,T)$ and $m>1$.
Here $C$ is universal, only depending on $d$, $T$, $b$, $f$, and $R_0$.


We integrate \eqref{1.1} in space-time to find that
\[
\lim_{t\to T^-}\int_{\bbR^d}p_m(x,t)\,dx-\int_{\bbR^d}p_m(x,0)\,dx
= \int_{Q_T} (m-1) p_m(\Delta p_m+\nabla\cdot b +f)+\nabla p_m\cdot(\nabla p_m+b)\,dx\,dt.
\]
Integrating by parts yields that
\[
(m-2) \int_{Q_T} |\na p_m|^2\,dx\,dt +\lim_{t\to T^-} \|p_m(\cdot,t)\|_{L^1}
= \|p_m(\cdot,0)\|_{L^1}
+ \int_{Q_T} (m-1) p_m(\nabla\cdot b +f) -  p_m\na \cdot b \,dx\,dt,
\]
and thus
\begin{align*}
&\;\int_{Q_T} |\na p_m|^2\,dx\,dt + \frac{1}{m-2} \lim_{t\to T^- } \|p_m(\cdot,t)\|_{L^1}\\
= &\; \frac{1}{m-2}\|p_m(\cdot,0)\|_{L^1}
+ \int_{Q_T} p_m\left(\nabla\cdot b +\frac{m-1}{m-2} f \right) dx\,dt.
\end{align*}
Therefore, there exists $C>0$, such that, for any $m\geq 3$,
\beq
\|\na p_m\|_{L^2(Q_T)} \leq C.
\label{eqn: spacetime L^2 estimate for grad p}
\eeq

In what follows, we derive uniform-in-$m$ space-time $W^{1,1}$-estimate for $\varrho_m$ and $p_m$.
We differentiate \eqref{main} with respect to $x_i$ $(i = 1,\cdots, d)$ to find that
\begin{align*}
\partial_t \partial_i \varrho_m
= &\;
\pa_i \varrho_m \D p_m + \na \pa_i \varrho_m \cdot \na p_m
+ \varrho_m \D \pa_i p_m + \na \varrho_m \cdot \na \pa_i p_m\\
&\; + \pa_i \varrho_m \na\cdot b + b\cdot \na \pa_i \varrho_m
+ \varrho_m \na\cdot \pa_i b + \pa_i b\cdot \na \varrho_m
\\
&\; + \pa_i \varrho_m \cdot f + \varrho_m \big[\pa_{x_i} f(x,t,p_m) + \pa_p f(x,t,p_m) \cdot \pa_i p_m\big].
\end{align*}
Multiplying it by $\sgn(\pa_i \varrho_m)$ and using the Kato's inequality $\sgn(\pa_i p_m)\D(\pa_i p_m)\leq \D|\pa_i p_m|$, we obtain that
\begin{align*}
\partial_t |\partial_i \varrho_m|
%
\leq &\;
\na\cdot \big[|\pa_i \varrho_m| \na p_m + \varrho_m \na |\pa_i p_m| + b |\pa_i \varrho_m | \big]\\
&\; + \sgn(\pa_i \varrho_m)\big[\varrho_m \na\cdot \pa_i b + \pa_i b\cdot \na \varrho_m\big] + |\pa_i \varrho_m| f
\\
&\;+ \sgn(\pa_i \varrho_m) \varrho_m \pa_{x_i} f(x,t,p_m) + \pa_p f(x,t,p_m)\cdot m \varrho_m^{m-1}|\pa_i \varrho_m|.
\end{align*}
Using the assumption $\pa_p f \leq 0$ and integrating on both sides,
\begin{align*}
\frac{d}{dt}\|\pa_i \varrho_m\|_{L^1}
\leq &\;\int_{\bbR^d} \sgn(\pa_i \varrho_m)\big[\varrho_m \na\cdot \pa_i b + \pa_i b\cdot \na \varrho_m\big] + |\pa_i \varrho_m| f + \sgn(\pa_i \varrho_m) \varrho_m \pa_{x_i} f(x,t,p_m)\,dx\\
\leq &\; C\|\varrho_m\|_{L^1} \|\na^2 b\|_{L^\infty(B_{R(T)}\times [0,T])} + C\|\na b\|_{L^\infty(B_{R(T)}\times [0,T])}\sum_{j = 1}^d\| \pa_j \varrho_m\|_{L^1}\\
&\; + \|\pa_i \varrho_m\|_{L^1} \|f\|_{L^\infty(B_{R(T)}\times [0,T]\times [0,C])} + \|\varrho_m\|_{L^1} \| \pa_{x_i} f\|_{L^\infty (B_{R(T)}\times [0,T]\times [0,C])}.
\end{align*}
We sum over $i$ to obtain that
\begin{align*}
\frac{d}{dt}\sum_{i = 1}^d\|\pa_i \varrho_m\|_{L^1}
\leq &\;C\|\varrho_m\|_{L^1} \big(\|\na^2 b\|_{L^\infty(B_{R(T)}\times [0,T])} + \| \pa_x f\|_{L^\infty (B_{R(T)}\times [0,T]\times [0,C])}
\big)\\
&\;+ C\big(\|\na b\|_{L^\infty(B_{R(T)}\times [0,T])}+ \|f\|_{L^\infty(B_{R(T)}\times [0,T]\times [0,C])}\big)\sum_{i = 1}^d\| \pa_i \varrho_m\|_{L^1}.
\end{align*}
Then under the assumption that $\sup_{m>1}\|\na \varrho_m(\cdot,0)\|_{L^1}<+\infty$, the Gronwall's inequality and \eqref{eqn: spatial L^1 bound for p_m and rho_m} imply that there exists a constant $C>0$, such that
\beq
\|\na \varrho_m(\cdot,t)\|_{L^1(\bbR^d)}\leq C\big(\|\na \varrho_m(\cdot,0)\|_{L^1(\bbR^d)}+1\big)
\label{eqn: spatial L^1 bound for grad rho_m}
\eeq
for all $t\in [0,T)$ and $m>1$.

Similarly, differentiating \eqref{main} in $t$ gives that
\begin{align*}
\partial_t \pa_t\varrho_m
= 
&\; m\Delta(\partial_t \varrho_m \cdot \varrho_m^{m-1})
+ b\cdot \na \pa_t \varrho_m + \pa_t \varrho_m (\na\cdot b) + \na \varrho_m\cdot \pa_t b  + \varrho_m \na \cdot \pa_t b \\
&\; +\pa_t \varrho_m f(x,t,p_m) + \varrho_m \pa_t f(x,t,p_m) + \varrho_m \pa_p f(x,t,p_m) \pa_t p_m.
\end{align*}
Multiplying this by $\sgn(\pa_t \varrho_m)$ and arguing as above, we find that
\begin{align*}
\partial_t |\pa_t\varrho_m|
%
%
\leq &\;
m \D \big(|\pa_t \varrho_m| \varrho_m^{m-1} \big)
+ \na\cdot \big(b|\pa_t \varrho_m|\big)
+ \big[\na \varrho_m\cdot \pa_t b  + \varrho_m \na \cdot \pa_t b \big]\sgn(\pa_t \varrho_m) \\
&\; +|\pa_t \varrho_m| f(x,t,p_m) + \sgn(\pa_t \varrho_m) \varrho_m \pa_t f(x,t,p_m) + \varrho_m \pa_p f(x,t,p_m) |\pa_t p_m|.
\end{align*}
Hence, under the assumption that $\pa_p f\leq 0$, 
\begin{align*}
&\;\frac{d}{dt} \|\pa_t\varrho_m\|_{L^1}
+\int_{\bbR^d} \varrho_m |\pa_p f(x,t,p_m)| |\pa_t p_m|\,dx\\
\leq &\;
\|\na \varrho_m\|_{L^1}\|\pa_t b\|_{L^\infty}
+ \|\varrho_m\|_{L^1} \|\na \pa_t b\|_{L^\infty} 
+ \|\pa_t \varrho_m\|_{L^1}\| f\|_{L^\infty} + \|\varrho_m \|_{L^1}\| \pa_t f\|_{L^\infty}.
\end{align*}
Using \eqref{eqn: spatial L^1 bound for p_m and rho_m} and \eqref{eqn: spatial L^1 bound for grad rho_m}, we conclude that there exists a constant $C>0$, such that for all $t\in [0,T]$ and all $m>1$,
\[
\|\pa_t\varrho_m(\cdot,t)\|_{L^1(\bbR^d)}
\leq C\big(\|\pa_t\varrho_m(\cdot,0)\|_{L^1}+\|\na \varrho_m(\cdot,0)\|_{L^1}+1\big),
\]
and thus by \eqref{main},
\beq
\|\pa_t\varrho_m(\cdot,t)\|_{L^1(\bbR^d)}
\leq C\big(
\|\Delta (\varrho_m(\cdot,0)^m)\|_{L^1}+\|\na \varrho_m(\cdot,0)\|_{L^1}+1\big).
\label{eqn: spacetime L^1 bound for pa_t rho_m}
\eeq
Moreover,
\beq
\int_{Q_T}\varrho_m |\pa_p f(x,t,p_m)| |\pa_t p_m|\,dx\,dt
\leq
C\big(
\|\Delta (\varrho_m(\cdot,0)^m)\|_{L^1} 
+\|\na \varrho_m(\cdot,0)\|_{L^1}+1\big).
\label{eqn: weighted L^1 spacetime estimate p_t}
\eeq

Finally, also by \eqref{1.1},  
\begin{align*}
&\; \|\pa_t p_m\|_{L^1(Q_T)} \\
\leq &\;
\int_{Q_T} \pa_t p_m \,dx \,dt
+2\int_{Q_T} (m-1) p_m|(\Delta p_m+\nabla\cdot b+f)_-|\,dx\,dt
+ \int_{Q_T} |\nabla p_m||b|- \na p_m\cdot b \,dx\,dt\\
\leq &\;
\int_{\bbR^d}p_m(x,T)-p_m(x,0) \,dx
+2\int_{Q_T} (m-1) p_m|(\Delta p_m+\nabla\cdot b+f)_-|\,dx\,dt
+ C\|\na p_m\|_{L^2(Q_T)}.
\end{align*}
If \eqref{R.1.1} holds, thanks to the Aronson-B\'{e}nilan estimate (cf.\;Remark \ref{R.1}), 
\[
\|\pa_t p_m\|_{L^1(Q_T)}
\leq
\lim_{t\to T^- }\int_{\bbR^d}p_m(x,t)-p_m(x,0) \,dx
+C\int_{Q_T} p_m\,dx\,dt + C\|\na p_m\|_{L^2(Q_T)},
\]
and thus by \eqref{eqn: spacetime L^2 estimate for grad p},
\beq
\|\pa_t p_m\|_{L^1(Q_T)} \leq C.
\label{eqn: spacetime L^1 bound for pa_t p_m}
\eeq
Alternatively, under the assumption that $\pa_p f\leq -\alpha$ for some $\alpha>0$, this estimate can be proved by using \eqref{eqn: weighted L^1 spacetime estimate p_t}. See the proof in \cite{David_S}.

\subsection{The incompressible limit}
Suppose $(\varrho_m,p_m)$ $(m>1)$ are solutions to \eqref{main}--\eqref{main2} in $Q_T$.
Thanks to Lemma \ref{uniformb} and the bounds \eqref{eqn: spatial L^1 bound for p_m and rho_m}--\eqref{eqn: spacetime L^1 bound for pa_t rho_m} and \eqref{eqn: spacetime L^1 bound for pa_t p_m}, we apply the Kolmogorov-Riesz-Fr\'{e}chet theorem \cite[Theorem 4.26]{Brezis2010FunctionalAS} to find that there exists a subsequence $\{(\varrho_{m_k},p_{m_k})\}_{k = 1}^\infty$ as well as $\varrho_\infty\in BV(Q_T)$ and $p_\infty\in BV(Q_T)$ with $\na p_\infty \in L^2(Q_T)$, such that, as $k\to+\infty$,
\beq
\varrho_{m_k}\to \varrho_\infty\mbox{ in }L^1(Q_T),\quad
p_{m_k}\to p_\infty\mbox{ in }L^1(Q_T),
\label{eqn: strong spacetime L^1 convergence of rho and p}
\eeq
and
\beq
\na p_{m_k} \rightharpoonup \na p_\infty\mbox{ in }L^2(Q_T).
\label{eqn: weak spacetime L^2 convergence of grad p_m}
\eeq
Thanks to the uniform $L^\infty$-bounds and interpolation with \eqref{eqn: strong spacetime L^1 convergence of rho and p}, we further obtain that $p_\infty$ is bounded, and for any $q\in [1,+\infty)$, as $k\to +\infty$,
\beq
\varrho_{m_k}\to \varrho_\infty\mbox{ in }L^q(Q_T),\quad
p_{m_k}\to p_\infty\mbox{ in }L^q(Q_T).
\label{eqn: strong spacetime L^q convergence of rho and p}
\eeq
By taking a further subsequence if necessary, we may assume that the convergence in \eqref{eqn: strong spacetime L^q convergence of rho and p} also holds in the almost everywhere sense.
Then taking the limit in
\[
\varrho_{m_k} = \left(\frac{m_k-1}{m_k}p_{m_k}\right)^{\frac{1}{m_k-1}}\leq C^{\frac{1}{m_k-1}},
\quad \mbox{and}\quad
\varrho_{m_k} \cdot \frac{m_k-1}{m_k}p_{m_k} = \left(\frac{m_k-1}{m_k}p_{m_k}\right)^{1+\frac{1}{m_k-1}},
\]
we readily obtain that
\beq
\varrho_\infty\leq 1,\quad
p_\infty(1-\varrho_\infty)=0\quad \mbox{almost everywhere.}
\label{eqn: Hele-Shaw graph}
\eeq

The weak formulation of \eqref{main} (i.e., \eqref{def sol2}) reads that, for any $\varphi = \varphi(x,t)\in C_0^\infty(\bbR^d\times [0,T))$,
\[
\int_{Q_T} \varrho_{m}\partial_t \varphi\,dx\,dt = -\int_{\mathbb{R}^d} \varrho_m^0(x)\varphi(0,x)\, dx +\int_{Q_T}  \big(\varrho_m\nabla p_m+\varrho_m b \big)\nabla\varphi
- \varrho_m f(x,t,p_m)\varphi \,dx\,dt.
\]
Taking $m = m_k$ and sending $k\to +\infty$, we can justify by \eqref{eqn: weak spacetime L^2 convergence of grad p_m}, \eqref{eqn: strong spacetime L^q convergence of rho and p}, and the dominated convergence theorem that
\[
\int_{Q_T} \varrho_\infty \partial_t \varphi\,dx\,dt  =
-\int_{\mathbb{R}^d} \varrho^0(x)\varphi(0,x)\, dx
+
\int_{Q_T}  \big(\varrho_\infty\nabla p_\infty+\varrho_\infty b \big)\nabla\varphi
- \varrho_\infty f(x,t,p_\infty)\varphi \,dx\,dt.
\]
Hence, in the sense of distribution, $(\varrho_\infty,p_\infty)$ satisfies
\beq
\partial_t\varrho_\infty =\nabla\cdot (\varrho_\infty\nabla p_\infty + \varrho_\infty  b )+\varrho_\infty f(x,t,p_\infty),
\label{eqn: PDE for incompressible limit}
\eeq
with $\varrho_\infty(x,0) = \varrho^0(x)$.
By \eqref{eqn: Hele-Shaw graph}, it also holds in distribution that
\beq
\partial_t\varrho_\infty =\D p_\infty + \na\cdot (\varrho_\infty  b )+\varrho_\infty f(x,t,p_\infty).
\label{eqn: PDE for incompressible limit 2}
\eeq

\begin{remark}
Under suitable additional assumptions, one can further derive finer estimates 
for $\na p_m$ and 
$\Delta p_m$, which eventually leads to the conclusion that the incompressible limit should satisfy the complementarity condition 
$ p_\infty(\Delta p_\infty+\nabla\cdot b +f)=0$
in the sense of distribution (see \eqref{1.2}).
However, this is not needed in proving the uniqueness of the incompressible limit or the space-time $L^1$-convergence of $p_m$, so we shall omit that.
We refer the readers to \cite{chu2022,David_S} for more details.
\end{remark}

\subsection{Uniqueness of the limit}
It remains to prove that the compactly supported solution to \eqref{eqn: Hele-Shaw graph} and \eqref{eqn: PDE for incompressible limit 2} is unique.
Once this is achieved, we can conclude that the convergence in \eqref{eqn: strong spacetime L^1 convergence of rho and p}--\eqref{eqn: strong spacetime L^q convergence of rho and p} actually holds for the whole sequence.

\begin{lem}
Assume \eqref{1.7}, $\pa_p f\leq 0$, and that $|\pa_{pp}f |+ |\pa_{tp}f|$ is locally finite in $Q_T\times [0,+\infty)$. 
Given $T>0$ and the initial data $\varrho_\infty(x,0) =\varrho^0\in [0,1]$ that is compactly supported, the equations \eqref{eqn: Hele-Shaw graph} and \eqref{eqn: PDE for incompressible limit 2} have a unique solution $(\varrho_\infty,p_\infty)$ in $Q_T$ satisfying that $\varrho_\infty, p_\infty \in L^\infty \cap BV(Q_T)$ are compactly supported, and $\nabla p_\infty \in L^2(Q_T)$.

\begin{proof}
The argument is standard, employing the Hilbert duality method.
We only sketch it here.
One can find more details in e.g.~\cite[Section 5]{David_S}.

With slight abuse of the notations, suppose $(\varrho_1, p_1)$ and $(\varrho_2,p_2)$ are two compactly supported solutions on $\bbR^d \times [0,T]$.
Assume that for a sufficiently large $R$, the supports of $\varrho_i$ and $p_i$ ($i = 1,2$) are contained in $B_R$ for all $t\in [0,T]$.
Subtracting the equations \eqref{eqn: PDE for incompressible limit 2} for $\varrho_1$ and $\varrho_2$, we find that, in the sense of distribution,
\[
\partial_t(\varrho_1-\varrho_2) =\D (p_1-p_2) + \na\cdot ((\varrho_1-\varrho_2)  b )+(\varrho_1 f_1 -\varrho_2 f_2),
\]
where $f_i: = f(x,t,p_i)$.
That means, for any smooth test function $\psi \in C^\infty(B_R\times [0,T])$ satisfying that $\psi(\cdot,T)\equiv 0$ and $\psi|_{\pa B_R}\equiv 0$, 
\beq
\int_{B_R\times [0,T]} (\varrho_1-\varrho_2)\pa_t \psi  +(p_1-p_2)\D \psi - (\varrho_1-\varrho_2)  b \cdot \na \psi+(\varrho_1 f_1 -\varrho_2 f_2)\psi\,dx\,dt = 0.
\label{eqn: weak form of equation for the difference}
\eeq
Here we used the fact $\varrho_1(x,0) = \varrho_2(x,0)$.
Denote
\[
A :=  \frac{\varrho_1-\varrho_2}{\varrho_1-\varrho_2+p_1-p_2},\quad 
B := \frac{p_1-p_2}{\varrho_1-\varrho_2+p_1-p_2},\quad 
D := -\varrho_2\frac{f_1-f_2}{p_1-p_2}.
\]
We define $A = 0$ whenever $\varrho_1 = \varrho_2$ (even when $p_1 =p_2$), and $B = 0$ whenever $p_1 = p_2$ (even when $\varrho_1 = \varrho_2$).
When $p_1=p_2$, we define $D = -\varrho_2 \pa_p f(x,t,p_i)$.
Since $(\varrho_i,p_i)$ satisfies \eqref{eqn: Hele-Shaw graph}, we have $A,B\in [0,1]$.
By virtue of the assumptions on $f$, $D\in [0, C]$ for some universal constant $C$.
Then \eqref{eqn: weak form of equation for the difference} can be rewritten as 
\beq
\int_{B_R\times [0,T]} (\varrho_1-\varrho_2+p_1-p_2)\left[A\pa_t \psi  +B\D \psi - A  b \cdot \na \psi+(Af_1 - BD)\psi\right] dx\,dt = 0.
\label{eqn: weak form of equation for the difference recast}
\eeq

In view of this, we introduce smooth approximations of $A,B,D, b, f_1$ in $B_R\times [0,T]$, denoted by $A_n, B_n, D_n, b_n,f_{1,n}$ respectively, such that
\begin{align*}
\|A_n-A\|_{L^2(B_R\times [0,T])}
&+\|B_n-B\|_{L^2(B_R\times [0,T])}
+\|D_n-D\|_{L^2(B_R\times [0,T])}\\
&+\|b_n-b\|_{L^\infty(B_R\times [0,T])}
+\|f_{1,n}-f_1\|_{L^2(B_R\times [0,T])}\leq \frac{C}{n},
\end{align*}
and
\[
A_n,B_n\in \left[1,\frac{1}{n}\right],\quad
D_n, |b_n|,|\na b_n|,|f_{1,n}|\in [0,C],\quad 
\|\na f_{1,n}\|_{L^2(B_R\times [0,T])}+\|\pa_t D_n\|_{L^1(B_R\times [0,T])}\leq C,
\]
where $C>0$ are universal constants.
We note that a uniform $L^2$-bound for $\na f_{1,n}$ is possible because 
\[
\na f_1 = \pa_x f(x,t,p_1) + \pa_p f(x,t,p_1)\na p_1\in L^2(B_R\times [0,T]).
\]
A uniform $L^1$-bound for $\pa_t D_n$ stems from the following formal calculation 
\begin{align*}
\pa_t D 
= &\; -\pa_t \varrho_2 \cdot \frac{f_1-f_2}{p_1-p_2} - \varrho_2 \cdot \frac{\pa_t f(x,t,p_1)-\pa_t f(x,t,p_2)}{p_1-p_2}\\
&\; - \varrho_2 \pa_t p_1\cdot \frac{\pa_p f(x,t,p_1) - \frac{f_1-f_2}{p_1-p_2}}{p_1-p_2}
- \varrho_2 \pa_t p_2\cdot \frac{\frac{f_1-f_2}{p_1-p_2} -\pa_p f(x,t,p_2)}{p_1-p_2},
\end{align*}
as well as the assumptions on $(\varrho_i,p_i)$ and $f$.

Take an arbitrary $\eta\in C_0^\infty(B_R\times [0,T])$, and consider the approximate dual problem
\[
\pa_t \psi +\frac{B_n}{A_n}\Delta \psi - b_n\cdot \na \psi 
+ \left( f_{1,n} - \frac{B_nD_n}{A_n}\right)\psi = \eta,\quad
\psi(\cdot,T) \equiv 0,\quad 
\psi|_{\pa B_R} \equiv 0.
\]
Since $B_n/A_n\in [n^{-1},n]$, and all the coefficients are smooth, this equation admits a unique smooth solution $\psi_n = \psi_n(x,t)$.
Then one can follow the argument in \cite[Section 5]{David_S} to show that
\[
\|\psi_n\|_{L^\infty(B_R\times [0,T])}
+\sup_{t\in [0,T]}\|\na \psi_n(\cdot,t)\|_{L^2(B_R)}
+\big\|(B_n/A_n)^{1/2}(\Delta \psi_n - D_n \psi_n)\big\|_{L^2(B_R\times [0,T])}\leq C,
\]
where $C$ is independent of $n$.
Then we take $\psi$ in \eqref{eqn: weak form of equation for the difference recast} to be $\psi_n$ and derive that 
\begin{align*}
0 
= &\; \int_{B_R\times [0,T]} (\varrho_1-\varrho_2+p_1-p_2)\big[A\pa_t \psi_n  +B\D \psi_n - A  b \cdot \na \psi_n+(Af_1 - BD)\psi_n\big] \, dx\,dt\\
= &\; \int_{B_R\times [0,T]} (\varrho_1-\varrho_2+p_1-p_2)A \left[\pa_t \psi_n  +\frac{B_n}{A_n}\D \psi_n - b_n \cdot \na \psi_n+\left(f_{1,n} - \frac{B_nD_n}{A_n}\right)\psi_n\right] dx\,dt\\
&\; +\int_{B_R\times [0,T]} (\varrho_1-\varrho_2+p_1-p_2)\\
&\;\qquad \cdot
\left[ \left(B-A\frac{B_n}{A_n}\right)\D \psi_n - A  (b-b_n) \cdot \na \psi_n+ A(f_1-f_{1,n})\psi_n - \left( BD -A\frac{B_n}{A_n}D_n \right)\psi_n\right] dx\,dt\\
= &\; \int_{B_R\times [0,T]} (\varrho_1-\varrho_2) \eta \, dx\,dt + I_{1,n}+I_{2,n}+I_{3,n}+I_{4,n},
\end{align*}
where 
\begin{align*}
I_{1,n} := &\;\int_{B_R\times [0,T]} (\varrho_1-\varrho_2+p_1-p_2)
(B-B_n)(\D \psi_n -D_n \psi_n)\, dx\,dt,\\
I_{2,n} := &\;\int_{B_R\times [0,T]} (\varrho_1-\varrho_2+p_1-p_2)
(A_n -A)\cdot \frac{B_n}{A_n}(\D \psi_n -D_n \psi_n)\, dx\,dt,\\
I_{3,n} := &\;- \int_{B_R\times [0,T]} (\varrho_1-\varrho_2+p_1-p_2) B(D-D_n) \psi_n\, dx\,dt,\\
I_{4,n} := &\; \int_{B_R\times [0,T]} (\varrho_1-\varrho_2)
\big[-
(b-b_n) \cdot \na \psi_n+ (f_1-f_{1,n})\psi_n\big]\, dx\,dt.
\end{align*}
Using the assumptions on the approximations as well as the uniform estimates above, it is not difficult to show that $I_{j,n}\to 0$ $(j = 1,2,3,4)$ as $n\to +\infty$.
Therefore, for any $\eta\in C_0^\infty(B_R\times [0,T])$,
\[
\int_{B_R\times [0,T]} (\varrho_1-\varrho_2) \eta \, dx\,dt  = 0.
\]
This implies that $\varrho_1 = \varrho_2$ almost everywhere in $B_R\times [0,T]$.
Combining this with \eqref{eqn: weak form of equation for the difference}, we also find that $p_1 = p_2$ almost everywhere in $B_R\times [0,T]$.
\end{proof}
\end{lem}

\section{Proof of Lemma \ref{L.5.12}}
\label{sec: proof of strict expansion lemma 1}

Fix $m\geq 2$.
Take a free boundary point $x_0\in\Gamma_{p_m}(0)$; in the rest of the proof, we shall omit the subscript $p_m$ whenever it is convenient. 
Since $\Omega(0)$ is Lipschitz (though its Lipschitz constant can possibly depend on $m$), there exists $C_m>0$, such that for any sufficiently small $\eps>0$, we are able find $z_0\in \Omega(0)$ satisfying that $\eps=d(z_0,\Gamma(0))$ and $d(x_0,z_0)\leq C_m\eps$. 
After shifting, we assume $z_0=0$. The goal is to find some $t_\eps>0$, which converges to $0$ as $\eps\to 0$, such that $B(X(x_0,0;t_\eps),r_\eps)\subseteq \Omega_{p_m}(t_\eps)$ for some $r_\eps>0$.

We apply a barrier argument.
Define $r_0:= \eps-\eps^{\frac{1}{1-\varsigma_0/4}}$.
By taking $\eps$ to be sufficiently small, we assume $r_0\in [\eps/2,\eps)$.
By the assumption \eqref{introgr},
\beq \lb{5.10}
p_m^0(x)\geq \gamma_0(\eps-|x|)_+^{2-\varsigma_0}
\geq \gamma_0(\eps-r_0)^{2-\varsigma_0} \mathds{1}_{\{|x|\leq r_0\}} 
\geq \gamma_0 \eps^{\frac{2-\varsigma_0}{1-\varsigma_0/4}}\cdot \eps^{-2}
\left(r_0^2 - |x|^2\right)_+.
\eeq
Denote the coefficient above by $A_0 := \gamma_0 \eps^{\frac{2-\varsigma_0}{1-\varsigma_0/4}-2}$.
By requiring $\eps$ to be sufficiently small, we can make $A_0 = \gamma_0 \eps^{-\frac{2\varsigma_0}{4-\varsigma_0}} \geq 2(\|\na b\|_{\infty}+1 )$. 
%
%
For some large $L>0$ to be determined, we define
\[
A(t):=\frac{A_0}{LA_0t+1},\quad r(t)=r_0(LA_0t+1)^\frac1L,\quad\text{and}\quad \tau_0:= \min\left\{\frac{A_0-\|\nabla b\|_\infty}{LA_0\|\nabla b\|_\infty},\,\frac{A_0-1}{LA_0}\right\}.
\]
It is straightforward to verify that for $t\in [0,\tau_0]$,
\beq\lb{A2}
A'=-LA^2,\quad r'=Ar,\quad\text{and}\quad  A\geq \max\{\|\nabla b\|_\infty,1\}.
\eeq
Then let
\[
\phi(x,t):= A(t)\big(r(t)^2-|x|^2\big)_+.
\]
It follows from \eqref{5.10} that $p_m^0(x)\geq \phi(x,0)$.

We shall compare $\phi(x,t)$ with $v(x,t):=p_m(x+X(t),t)$, which satisfies $\calL(v)=0$.
Here
\beq\lb{5.11}
\calL(g):=g_t - (m-1)g(\Delta g+F)- |\nabla g|^2-\nabla g\cdot \big( b(x+X,t)-b (X,t)\big),
\eeq
with $X:=X(t)$ defined in \eqref{ode}, and 
$F:=\nabla\cdot b (x+X,t)+f(x+X,t,v(x,t))$.
Note that $F$ is viewed as a given function of $(x,t)$, not depending on $g$.
Since $v$ is a priori bounded, $F$ is bounded as well. 

Let us show that $\phi(x,t)$ is a subsolution to \eqref{5.11}. 
Direct calculation yields that, for $|x|\leq r(t)$,
\begin{align*}
\calL(\phi) 
\leq &\; A'(r^2-|x|^2)+2Arr'-(m-1)A(r^2-|x|^2)(-2dA+F)\\
&\;-4A^2|x|^2+2Ax\cdot ({ b }(x+{X},t)-{ b }({X},t))\\
\leq &\; \left(A'+(m-1)(2dA^2+A\|F\|_\infty)\right)(r^2-|x|^2)+2Arr'-4A^2|x|^2+2\|\nabla b \|_\infty A|x|^2.
\end{align*}
By \eqref{A2} and $A\geq 1$, we get
\begin{align*}
\calL(\phi) 
&\leq \left(-LA^2+(m-1)(2dA^2+A^2\|F\|_\infty)\right)(r^2-|x|^2)+2A^2 r^2-4A^2|x|^2+ 2A^2|x|^2 \\
& \leq\big(-L+m(2d+\|F\|_\infty)\big)(r^2-|x|^2)A^2.
\end{align*}
Choosing $L:=m(2d+\|F\|_\infty)$, we obtain $\calL(\phi)\leq 0$. 

Then the comparison principle implies  $v\geq\phi$ for all $t\in [0,\tau_0]$. Since $v= p_m(x+X(t),t)$, we get
\[
B_{r(t)}\subseteq \{x-X(t)\,|\, x\in \Omega(t)\}.
\]
Note that $\tau_0\geq c_m>0$ for some $c_m$ independent of $\eps$. 

It remains to find some $t_\eps$, such that $t_\eps\to0$ as $\eps\to 0$, and that $B_{r(t_\eps)}+X(t_\eps)$ contains a neighbourhood of $X(x_0,0;t_\eps)$. 
Since $\|\nabla b \|_\infty \leq C$, $|x_0|\leq C_m\eps$, and
\[
\frac{d}{dt}|X(x_0,0;t)-X(t)|\leq \|\nabla b \|_\infty |X(x_0,0;t)-X(t)|.
\]
we find
\[
|X(x_0,0;t)-X(t)|\leq e^{Ct}C_m\eps.
\]
Therefore, we only need
\[
r(t_\eps)=r_0(LA_0 t_\eps+1)^\frac1L\geq  2  e^{C t_\eps}C_m\eps.
\]
Recall that $r_0 \geq \eps/2$, $L=m(2d+\|F\|_\infty)$, and $A_0= \gamma_0 \eps^{-\frac{2\varsigma_0}{4-\varsigma_0}}$. 
We can pick $t_\eps:=\eps^{\varsigma_0/2}$ and let $\eps$ be suitably large to make this inequality true.
We thus conclude that $B_{r(t_\eps)}+X(t_\eps)$ contains $B(X(x_0,0;t_\eps),r(t_\eps)/2)$ for all $\eps>0$ sufficiently small depending on $m$. 

Since $x_0\in \Gamma(0)$ is arbitrary and by Lemma \ref{streamline}, this completes the proof.

\section{Proof of Lemma \ref{L.4.4}}
\label{sec: proof of L.4.4}

Take an arbitrary $m>1$.
In what follows, we shall omit the subscripts $p_m$ whenever it is convenient.
Take an arbitrary $x_0\in\Gamma(0)$, and let $z_0\in \Omega(0)$ be such that $d(x_0,z_0)= d(z_0,\Gamma(0))=:\eps$. 
This can be achieved thanks to the interior ball assumption when $\eps$ is sufficiently small. Up to a suitable shifting, let us assume $z_0=0$. 
It follows from \eqref{introgr} that, with arbitrary $\varsigma \in(0, \varsigma_0)$,
\beq\lb{5.12}
p_m^0(x)\geq \gamma_0(\eps-|x|)_+^{2-\varsigma_0} 
\geq \gamma_0 \eps^{\varsigma-\varsigma_0}(\eps-|x|)_+^{2-\varsigma}.
\eeq
Hence, we can adjust $\gamma_0$ arbitrarily at the cost of making $\varsigma_0$ to be slightly smaller and requiring $\eps$ to be sufficiently small. 
Thus, without loss of generality and with slight abuse of the notations, we assume that $p_m^0(x)\geq \gamma_0(\eps-|x|)_+^{2-\varsigma_0}$ with
\beq\lb{6.15}
\sigma>2d\gamma_0,\quad \gamma_0>\|\nabla b\|_\infty,\quad \varsigma_0\in (0,1).
\eeq

Next, set
$r_0:=\eps-\eps^{\frac{1}{1-\varsigma_0/4}}$ as in the proof of Lemma \ref{L.5.12}. 
For some $\alpha>0$ to be determined, define
\[
\gamma(t):=e^{-2\alpha t}\gamma_0,\quad 
r(t):=e^{\alpha t}r_0, \quad \text{and}\quad 
\phi(x,t):= \gamma(t)(r(t)^2-|x|^2)_+.
\]
It follows from the proof of \eqref{5.10} that $p_m^0(x) \geq \gamma_0 (r_0^2 - |x|^2)_+$ and thus 
$p_m^0(x)\geq \phi(x,0)$. As before, we shall compare $\phi(x,t)$ with $v(x,t):=p_m(x+X(t),t)$ which satisfies $\tilde\calL(v)=0$.
Here $\tilde\calL$ is defined by
\[
\tilde\calL(g):=g_t - (m-1)g(\Delta g+\tilde F+(\partial_p f) g)- |\nabla g|^2-\nabla g\cdot (b(x+X,t)- b (X,t)),
\]
with $\partial_p f:=\partial_pf(x+X,t,v(x,t))$ and
$\tilde F:=\nabla\cdot b (x+X,t)+f(x+X,t,v)-\partial_p f(x+X,t,v)v$.
Note that $\partial_p f$ and $\tilde F$ are treated as finite given functions of $(x,t)$, which are independent of $g$. 
By the assumption \eqref{cond}, we have $\tilde F\geq \sigma>0$.

To show that $\phi$ is a subsolution to $\tilde\calL$, direct calculation yields for $|x|\leq r(t)$, 
\begin{align*}
\tilde\calL(\phi)
\leq &\;\gamma'(r^2-|x|^2)+2\gamma rr'-(m-1)\gamma(r^2-|x|^2)(-2d\gamma+\tilde F+(\partial_pf)\phi)\\
&\;-4\gamma^2|x|^2+2\gamma x\cdot \big( b (x+X,t)-b (X,t)\big)\\
\leq &\; \big[-2\alpha\gamma-(m-1)\gamma(-2d\gamma+\sigma-\|\partial_p f\|_\infty \gamma r^2)\big](r^2-|x|^2)+2\alpha\gamma r^2-4\gamma^2|x|^2+2\|\nabla{ b }\|_\infty \gamma|x|^2.
\end{align*}
In view of \eqref{6.15}, we can find $\alpha>0$ and $\delta>0$ independent of $m$ and $\eps$ such that for all $t\in [0,\delta]$,
\[
\|\nabla b\|_\infty+\delta \leq\alpha\leq 2e^{-2\alpha\delta}\gamma_0-\|\nabla b\|_\infty\leq 2\gamma-\|\nabla b\|_\infty.
\]
Since $\sigma>2d\gamma_0$, for all $r_0$ sufficiently small,
\[
\sigma\geq 2d\gamma_0+\|\partial_p f\|_\infty \gamma_0r_0^2\geq 2d\gamma+\|\partial_p f\|_\infty \gamma r^2.
\]
As a consequence, we obtain for all $t\in [0,\delta]$ that
\begin{align*}
\tilde\calL(\phi)
\leq -2\alpha\gamma(r^2-|x|^2)+2\alpha\gamma r^2-2(2\gamma-\|\nabla{ b }\|_\infty) \gamma|x|^2\leq  0.
\end{align*}
The comparison principle then implies $p_m\geq\phi$ for all $t\in [0,\delta]$, and thus
\[
B_{r(t)}\subseteq \{x-X(t)\,|\, x\in \Omega(t)\}\quad\text{for }t\in [0,\delta].
\]

Now we look for $t_\eps>0$ satisfying $\lim_{\eps\to 0}t_\eps=0$, such that $B_{r(t_\eps)}+X(t_\eps)$ contains $B(X(x_0,0;t_\eps),r_\eps)$, with $r_\eps := e^{(\alpha-\delta)t_\eps}\eps^2$. 
Since $|x_0|=\eps$ and $\alpha\geq\|\nabla b\|_\infty+\delta$,
\[
|X(x_0,0;t_\eps)-X(t_\eps)|\leq e^{\|\nabla b\|_\infty t_\eps}\eps\leq e^{(\alpha-\delta)t_\eps}\eps.
\]
Hence, it suffices to have $r(t_\eps)\geq  e^{(\alpha-\delta)t_\eps}(1+\eps)\eps$, which reduces to 
\[
e^{\delta  t_\eps }\geq \frac{(1+\eps)\eps}{r_0} = \frac{1+\eps}{1-
\eps^{\frac{\varsigma_0}{4-\varsigma_0}}
}= 1+O\big(\eps^{\frac{\varsigma_0}{4-\varsigma_0}}\big).
\]
This clearly holds if we choose $t_\eps:=\eps^{\varsigma_0/4}$ and let $\eps$ be sufficiently small. 
Besides, $r_\eps$ can be easily represented as a continuous function of $t_\eps$.
In view of Lemma \ref{streamline}, the proof is then completed.

\bigskip 


\end{document}